\documentclass[a4paper]{article}

\usepackage[all]{xy}\usepackage[latin1]{inputenc}       
\usepackage[dvips]{graphics,graphicx}
\usepackage{amsfonts,amssymb,amsmath,xcolor,mathrsfs, amstext}
\usepackage{amsbsy, amsopn, amscd, amsxtra, amsthm, authblk, enumerate,dsfont}
\usepackage{upref}
\usepackage{geometry}
\usepackage[displaymath]{lineno}
\usepackage{float}
\usepackage{bbm}
\usepackage{upgreek}

\allowdisplaybreaks

\usepackage[colorlinks,
            linkcolor=blue,
            anchorcolor=green,
            citecolor=blue
            ]{hyperref}

\numberwithin{equation}{section}

\newtheorem{theorem}{Theorem}[section]
\newtheorem{lemma}[theorem]{Lemma}
\newtheorem{proposition}[theorem]{Proposition}
\newtheorem*{proposition*}{Proposition}
\newtheorem{corollary}[theorem]{Corollary}
\newtheorem{conjecture}[theorem]{Conjecture}
\newtheorem*{corollary*}{Corollary}
\newtheorem{definition}[theorem]{Definition}
\newtheorem*{definitions*}{Definitions}

\newtheorem*{example*}{\bf Example}
\theoremstyle{remark}
\newtheorem{remark}[theorem]{\bf Remark}

\numberwithin{equation}{section}

\newcommand{\wt}{\widetilde}

\renewcommand{\1}{\mathbf{1}}
\newcommand\dd{\mathop{}\!\mathrm{d}} 
\newcommand{\E}{\mathbb{E}} 
\newcommand{\CLE}{\mathrm{CLE}} 
\newcommand{\CR}{\mathrm{CR}} 
\newcommand{\QD}{\mathrm{QD}}
\newcommand{\QA}{\mathsf{QA}} 
\newcommand{\LF}{\mathrm{LF}} 
\DeclareMathOperator{\Weld}{Weld} 

\newcommand{\Md}{{\mathcal{M}}^\mathrm{disk}}

\def\e{\varepsilon}

\newcommand{\eqb}{\begin{equation}}
\newcommand{\eqe}{\end{equation}}

\title{Mixing rate exponent of planar Fortuin-Kasteleyn percolation}
\author{Haoyu Liu\thanks{Peking University} \qquad Baojun Wu$^*$ \qquad Zijie Zhuang\thanks{University of Pennsylvania}}

\begin{document}

\maketitle

\begin{abstract}
    Duminil-Copin and Manolescu (2022) recently proved the scaling relations for planar Fortuin-Kasteleyn (FK) percolation. In particular, they showed that the one-arm exponent and the mixing rate exponent are sufficient to derive the other near-critical exponents. The scaling limit of critical FK percolation is conjectured to be a conformally invariant random collection of loops called the conformal loop ensemble (CLE). In this paper, we define the CLE analog of the mixing rate exponent. Assuming the convergence of FK percolation to CLE, we show that the mixing rate exponent for FK percolation agrees with that of CLE. We prove that the CLE$_\kappa$ mixing rate exponent equals $\frac{3 \kappa}{8}-1$, thereby answering Question 3 of Duminil-Copin and Manolescu (2022). The derivation of the CLE exponent is based on an exact formula for the Radon-Nikodym derivative between the marginal laws of the odd-level and even-level CLE loops, which is obtained from the coupling between Liouville quantum gravity and CLE.
\end{abstract}

\setcounter{tocdepth}{2}
\tableofcontents

\section{Introduction}\label{sec:intro}

\subsection{Near-critical FK percolation and mixing rate exponent}\label{subsec:intro}

Bernoulli percolation is a fundamental lattice model in statistical physics with phase transition and exhibits rich critical behavior. It is obtained by declaring each edge open or closed independently with the same probability.
Fortuin-Kasteleyn (FK) percolation (also called the random-cluster model) introduced in~\cite{FK72} is a generalization of Bernoulli percolation obtained by weighting each cluster of open edges by a factor $q$. The planar FK percolation model is of particular interest due to its relation with the Ising and Potts models and 2D conformal field theories. For $q \in (0,4]$, planar FK percolation on the square lattice is conjectured to undergo a continuous phase transition, and moreover, at criticality converges to a conformally invariant random collection of loops called the conformal loop ensemble (CLE) with $\kappa=4\pi/\arccos(-\sqrt{q}/2) \in [4,8)$. We refer to~\cite{DC-notes} for reviews of recent progress on this model.

For lattice models with continuous phase transition, many natural observables at (near-)criticality decay polynomially and are encoded by critical exponents. These exponents are believed to satisfy a family of universal scaling relations even for models belonging to different universality classes~\cite{Widom-1965,Fisher-1967}. Such scaling relations were first proved for planar Bernoulli percolation in a seminal work by Kesten~\cite{Kes87a}, where the one-arm exponent $\xi_1$ and the four-arm exponent $\xi_4$ are sufficient to derive all the other near-critical exponents. Recently, Kesten's scaling relations were extended to planar FK percolation by Duminil-Copin and Manolescu~\cite{DCM22}; see Equations (R1)--(R7) therein. They noted that the same relations hold, with the four-arm exponent replaced by the mixing rate exponent $\iota = \iota(q)$ defined below in~\eqref{eq:def-iota}. Therefore, once the one-arm and the mixing rate exponents are obtained, the value of all other near-critical exponents can be derived. For the one-arm exponent, one can define its CLE analog and show that they are equal under the assumption that critical FK percolation interfaces converge to CLE. Moreover, the value of the CLE one-arm exponent was obtained in~\cite{SSW09}. The main focus of this paper is to formulate the CLE analog of the mixing rate exponent $\iota$ and rigorously derive its value, as predicted in~\cite{DCM22} based on physical predictions for other near-critical exponents. 

Throughout this paper, we focus on FK percolation on the square lattice with cluster weight $q \in (1,4]$. For $p \in (0,1)$ and $R>1$, let $\phi_{\Lambda_R,p,q}^1$ and $\phi_{\Lambda_R,p,q}^0$ denote the probability measures of FK percolation with edge weight $p$ and cluster weight $q$ on the box $\Lambda_R:=[-R,R]^2 \cap \mathbb{Z}^2$ equipped with wired and free boundary conditions, respectively. We refer to Section~\ref{subsec:prelim-fk} for the precise definition and basic properties of FK percolation. Let $e$ be an edge next to the origin, write $\omega_e=1$ if $e$ is open, and $\omega_e=0$ otherwise. Following~\cite{DCM22}, the mixing rate is defined as
\begin{equation}\label{eq:mixing-rate-def-1}
    \Delta_{p,q}(R):=\phi_{\Lambda_R,p,q}^1[\omega_e]-\phi_{\Lambda_R,p,q}^0[\omega_e].
\end{equation}
Roughly speaking, $\Delta_{p,q}(R)$ captures the influence of boundary conditions on $e$. The model undergoes a continuous phase transition at $p_c=\frac{\sqrt{q}}{1+\sqrt{q}}$~\cite{BDC12, DCST17}, at which the mixing rate is expected to decay algebraically. The mixing rate exponent $\iota=\iota(q)$ is then defined as
\begin{equation}\label{eq:def-iota}
    \iota=-\lim_{R \to \infty} \frac{\log \Delta_{p_c,q}(R)}{\log R}
\end{equation}
provided that such a limit exists. In other words, $\Delta_{p_c,q}(R)=R^{-\iota+o(1)}$ as $R \to \infty$.

One of the main results of this paper is the following theorem, which answers Question 3 of~\cite{DCM22}.

\begin{theorem}\label{thm:iota-main}
    For $q \in (1,4]$ and $\kappa=4\pi/\arccos(-\sqrt{q}/2) \in [4,6)$, assuming the convergence of critical $\mathrm{FK}_q$ percolation interfaces to $\CLE_\kappa$ (Conjecture~\ref{conj:fk-to-cle}), $\iota(q)$ exists and equals $\frac{3\kappa}{8}-1$.
\end{theorem}

Plugging Theorem~\ref{thm:iota-main} into~\cite[Theorems 1.11]{DCM22} gives the following corollary.

\begin{corollary}\label{cor:near-critical}
    For $q \in (1,4]$ and $\kappa=4\pi/\arccos(-\sqrt{q}/2) \in [4,6)$, assuming the convergence of critical $\mathrm{FK}_q$ percolation interfaces to $\CLE_\kappa$ (Conjecture~\ref{conj:fk-to-cle}), the near-critical exponents $\beta, \gamma, \delta$ described in~\cite[Section 1.4]{DCM22} exist and equal to the values given there. The same holds for $\alpha$ when $q \in [2,4]$. 
\end{corollary}

In~\cite{DCM22}, the authors suggested two approaches to derive all near-critical exponents. Since $\xi_1$ has already been determined, it suffices to derive one other exponent. The first approach is to use Yang-Baxter integrability to derive the free energy exponent $\alpha$; see Question 2 therein. The second approach is to use conformal invariance and CLE convergence to derive $\iota$. This is Question 3 there, which we settle in Theorem~\ref{thm:iota-main}. Question 1 of~\cite{DCM22} asked whether the scaling relation concerning $\alpha$ can be extended to $q \in (1,2)$. If this is true, then Corollary~\ref{cor:near-critical} also holds for $\alpha$ when $q \in (1,2)$. 

The rest of Section~\ref{sec:intro} is organized as follows. In Section~\ref{subsec:mixing-rate-CLE}, we define the CLE analog of the mixing rate exponent for any $\kappa \in (8/3,8)$ and show that it equals $\frac{3\kappa}{8}-1$ except for $\kappa = 6$ (Theorem~\ref{thm:iota-continuum-easy}). In Section~\ref{subsec:discrete-conv}, we provide the relation between the mixing rate of FK percolation and that of CLE. Assuming the convergence of FK percolation to CLE, we prove that the mixing rate exponent for FK percolation agrees with that for CLE (Theorem~\ref{thm:iota}), thus proving Theorem~\ref{thm:iota-main}. In Sections~\ref{subsec:proof-overview} and~\ref{subsec:outlook}, we present the proof strategy and an outlook.

\subsection{Mixing rate exponent for CLE}\label{subsec:mixing-rate-CLE}

The conformal loop ensemble (CLE) introduced in~\cite{She07, Sheffield-Werner-CLE} is a random collection of non-crossing planar loops, indexed by a parameter $\kappa \in (8/3,8)$. When $\kappa \in (8/3,4]$, each loop in CLE$_\kappa$ is a simple loop and does not touch other loops or the boundary. When $\kappa \in (4,8)$, the loops are non-simple and may touch each other or the boundary. We will consider a nested CLE$_\kappa$ on the box $[-1,1]^2$ for $\kappa \in (8/3,8)$ which can be constructed from non-nested ones using an iteration procedure. Each loop is naturally associated with a nesting level equal to the number of distinct loops surrounding it plus 1.
Now we focus on the collection of loops surrounding a given point, say, the origin. For $0<r_1<r_2<1$, let $\mathcal{A}_{r_1,r_2}=(-r_2,r_2)^2 \setminus [-r_1,r_1]^2$, and let $\mathsf{A}^{\mathrm{odd}}_{r_1,r_2}$ (resp. $\mathsf{A}^{\mathrm{even}}_{r_1,r_2}$) be the event that there exists a non-contractible $\CLE_\kappa$ loop in $\mathcal{A}_{r_1,r_2}$ with odd (resp. even) nesting level. For $\kappa \in (8/3,8)$, $\delta>0$ and $r \in (0,\frac{1}{1+\delta})$, we define the mixing rate of CLE$_\kappa$ as
\begin{equation}\label{eq:mixing-rate-cont-def}
    \Updelta_\kappa(r;\delta)=\frac{\mathbb{P}[\mathsf{A}^{\mathrm{odd}}_{r,(1+\delta)r}]-\mathbb{P}[\mathsf{A}^{\mathrm{even}}_{r,(1+\delta)r}]}{\mathbb{P}[\mathsf{A}^{\mathrm{even}}_{r,(1+\delta)r}]} \,.
\end{equation}
Let $\delta$ be sufficiently small but fixed. The mixing rate exponent $\upiota=\upiota(\kappa)$ for CLE$_\kappa$ is defined as
\begin{equation}\label{eq:def-iota-cont}
    \upiota=\lim_{r \to 0} \frac{\log |\Updelta_\kappa(r;\delta)|}{\log r}.
\end{equation}

The following theorem provides the sign and up-to-constant asymptotics of the mixing rate.

\begin{theorem}\label{thm:iota-continuum-easy}
   Fix $\kappa \in (8/3,8)$. There exists $\delta_0=\delta_0(\kappa)>0$ such that the following holds for any $\delta \in (0,\delta_0)$ (the implicit constants in $\asymp$ depend only on $\kappa$).
    \begin{itemize}
        \item For $\kappa \in (8/3,6)$, we have $\Updelta_\kappa(r;\delta)>0$, and $\Updelta_\kappa(r;\delta) \asymp r^{\frac{3\kappa}{8}-1}$ as $r \to 0$.
        \item For $\kappa=6$, we have $\Updelta_\kappa(r;\delta)=0$ for all $r \in (0, \frac{1}{1+\delta})$.
        \item For $\kappa \in (6,8)$, we have $\Updelta_\kappa(r;\delta)<0$, and $-\Updelta_\kappa(r;\delta) \asymp r^{\frac{3\kappa}{8}-1}$ as $r \to 0$.
    \end{itemize}
    In particular, for $\kappa \in (8/3,8)\!\setminus\!\{6\}$, the CLE mixing rate exponent $\upiota(\kappa)$ exists and equals $\frac{3\kappa}{8}-1$.
\end{theorem}

In fact, we will prove a stronger result for general annular regions, not just for $\mathcal{A}_{r_1,r_2}$; see Proposition~\ref{prop:iota-continuum}.
The proof of Theorem~\ref{thm:iota-continuum-easy} is based on an exact formula that compares the marginal laws of CLE$_\kappa$ loops with even and odd nesting levels, which may be of independent interest.

\begin{definition}\label{def:modulus}
    Let $D \subsetneq \mathbb{C}$ be a simply-connected domain (e.g. $[-1,1]^2$ or the unit disk $\mathbb{D}$).
    For a \emph{non-boundary-touching} (possibly self-touching) loop $\eta$ in $D$, let $A_\eta$ be the doubly connected domain bounded by $\partial D$ and the outer boundary of $\eta$. Then, $A_\eta$ is conformally equivalent to the annulus $\mathcal{A}_r:=\{z \in \mathbb{C}: |z| \in (r,1)\}$ for some unique $r=r(\eta) \in (0,1)$. We call $\tau=\frac{1}{2\pi} \log \frac{1}{r}$ the \emph{modulus} of $A_\eta$.
\end{definition}

\begin{theorem}\label{thm:cle-partition}
    Fix $\kappa \in (8/3,8)$, and consider a nested $\CLE_\kappa$ on a simply connected domain $D \subsetneq \mathbb{C}$ that contains the origin. Let $\mathsf{m}_\kappa^{\mathrm{odd}}$ (resp.\ $\mathsf{m}_\kappa^{\mathrm{even}}$) denote the law of a loop sampled from the counting measure of non-boundary-touching $\CLE_\kappa$ loops surrounding the origin with odd (resp. even) nesting levels. Then
    \begin{equation}\label{eq:cle-partition}
        \frac{\mathsf{m}_\kappa^{\mathrm{odd}}(\dd \eta)}{\mathsf{m}_\kappa^{\mathrm{even}}(\dd \eta)}=\frac{\sum_{m \in \mathbb{Z}} (-1)^m \sin(\frac{\kappa}{4}(m+1)\pi) r^{\frac{\kappa}{8}m^2+(\frac{\kappa}{4}-1)m}}{\sum_{m \in \mathbb{Z}} \sin(\frac{\kappa}{4}(m+1)\pi) r^{\frac{\kappa}{8}m^2+(\frac{\kappa}{4}-1)m}} \,.
    \end{equation}
    The right-hand side is $1+4\cos(\frac{\kappa-4}{4}\pi) r^{\frac{3\kappa}{8}-1} + O(r^{2(\frac{3\kappa}{8}-1)})$ as $r \to 0$.
\end{theorem}

\begin{remark}
    \begin{itemize}
        \item (Boundary-touching case) When $\kappa \in (4,8)$, the odd-level CLE loops, particularly the first level, may touch the boundary while the even-level loops do not. We remind the reader that in~\eqref{eq:cle-partition} we only compare the Radon-Nikodym derivatives of loops that do not touch the boundary.
        
        \item (Partition function of SLE curves) The idea of expressing Radon-Nikodym derivatives for SLE curves or SLE loops via ratios of partition functions has been considered previously in, e.g.,~\cite{Werner-loop, Lawler-partition, Dub09}. Theorem~\ref{thm:cle-partition} can be interpreted as determining the partition functions of CLE loops. In fact, the formula on the right-hand side of~\eqref{eq:cle-partition} can be extracted from the partition function of FK percolation on the annulus~\cite{Cardy06}, which is based on a non-rigorous Coulomb gas approach. We refer to Theorem~\ref{thm:mod} and Remark~\ref{rem:pfn} for further explanation.

        \item (The case $\kappa =6$) In the setting of Bernoulli percolation (i.e., cluster weight $q=1$), the mixing rate similarly defined as in~\eqref{eq:mixing-rate-def-1} is always zero. This matches with the fact that the right-hand side of~\eqref{eq:cle-partition} is identically 1 and the CLE mixing rate $\Updelta_\kappa(r;\delta)=0$ when $\kappa=6$.

        \item  (Generalizations) If we replace $\mathsf{m}_\kappa^{\mathrm{odd}}$ (resp. $\mathsf{m}_\kappa^{\mathrm{even}}$) with the law of the outer boundary of odd-level (resp. even-level) non-boundary-touching CLE$_\kappa$ loops (which is a measure on simple loops), then~\eqref{eq:cle-partition} still holds. In Theorem~\ref{thm:cle-partition}, we compare the laws of all the odd-level and even-level CLE loops. In fact, the same approach can be used to derive the Radon-Nikodym derivatives between loops at any two given levels.        
    \end{itemize}
\end{remark}

\subsection{Comparison of mixing rates between FK percolation and CLE}
\label{subsec:discrete-conv}

In this section, we explain the relation between the mixing rate of FK percolation and that of CLE. First, the former has an alternative description as follows. For a rectangle $D=[a,b] \times [c,d]$, let $\mathcal{C}(D)$ be the event that there exists a horizontal open crossing connecting its left and right boundaries. For $q \in (1,4]$, $p \in (0,1)$, and $1 \le r<R$, define
\begin{equation}\label{eq:mixing-rate-def-2}
    \Delta_{p,q}(r,R):=\phi_{\Lambda_R,p,q}^1[\mathcal{C}(\Lambda_r)]-\phi_{\Lambda_R,p,q}^0[\mathcal{C}(\Lambda_r)].
\end{equation}
By~\cite[Theorem 1.6(ii)]{DCM22}, we have $\Delta_{p_c,q}(r,R) \asymp \Delta_{p_c,q}(R)/\Delta_{p_c,q}(r)$. Hence, the mixing rate exponent should also satisfy the asymptotics $\Delta_{p_c,q}(r,R)=(r/R)^{\iota+o(1)}$ as $R/r \to \infty$. As noted already, it is \emph{a priori} not clear whether such an exponent $\iota=\iota(q)$ exists.

While the original definitions of the FK mixing rate given in~\eqref{eq:mixing-rate-def-1} and~\eqref{eq:mixing-rate-def-2} are useful for deriving scaling relations with other near-critical exponents~\cite{DCM22}, the crossing events $\{\omega_e=1\}$ and $\mathcal{C}(\Lambda_r)$ are difficult to describe in terms of percolation interfaces. Hence, their scaling limits are hard to compute using CLE. To address this, we introduce another event $A(r;\delta)$ defined below whose conjectured scaling limits are easy to describe in terms of CLE while still capturing the mixing rate.

FK percolation on the square lattice naturally induces a dual configuation on the dual graph; see Section~\ref{subsec:prelim-fk} for the precise definitions. We call an edge open on the origin graph as primal open, and an edge open on the dual graph as dual open. For $r \ge 1$, let $A(r;\delta)$ denote the event that there exist a non-contractible primal open circuit $\eta$ and a non-contractible dual open circuit $\widetilde\eta$ in the annulus $\Lambda_{r,(1+\delta)r}=\Lambda_{(1+\delta)r} \!\setminus\! \Lambda_r$, with $\widetilde\eta$ lying inside of $\eta$. The following proposition states that the event $A(r;\delta)$ also captures the mixing rate.

\begin{proposition}\label{prop:event-A}
    For any fixed $q \in (1,4]$ and $\delta \in (0,\frac{1}{2})$, there exist constants $0<c<C$ (depending only on $q$ and $\delta$), such that for all $1000\delta^{-1} \le 8r<R$, we have
    \begin{equation}\label{eq:event-A}
        c \cdot \Delta_{p_c,q}(r,R) \le \frac{\phi_{\Lambda_R,p_c,q}^1[A(r;\delta)]-\phi_{\Lambda_R,p_c,q}^0[A(r;\delta)]}{\phi_{\Lambda_R,p_c,q}^0[A(r;\delta)]} \le C \cdot \Delta_{p_c,q}(r,R).
    \end{equation}
\end{proposition}

Percolation interfaces of FK percolation (also referred to as the loop configuration) are non-self-crossing loops on the medial graph that separate open clusters and dual open clusters (see Section~\ref{subsec:prelim-fk}). We associate each loop with a \emph{nesting level} equal to the number of distinct loops surrounding it plus 1. Under the wired (resp. free) boundary conditions, $A(r;\delta)$ is equivalent to the event that there exists a non-contractible percolation interface in $\Lambda_{r,(1+\delta)r}$ with odd (resp. even) nesting level. The scaling limit of this event corresponds to $\mathsf{A}^{\mathrm{odd}}_{r,(1+\delta)r}$ (resp. $\mathsf{A}^{\mathrm{even}}_{r,(1+\delta)r}$). In light of this, we can combine Theorem~\ref{thm:iota-continuum-easy} and Proposition~\ref{prop:event-A} to prove the following theorem, which directly implies Theorem~\ref{thm:iota-main}.

\begin{theorem}\label{thm:iota}
    For $q \in (1,4]$ and $\kappa=4\pi/\arccos(-\sqrt{q}/2) \in [4,6)$, assuming the convergence of critical $\mathrm{FK}_q$ percolation interfaces to $\CLE_\kappa$ (Conjecture~\ref{conj:fk-to-cle}), $\iota(q)$ defined in~\eqref{eq:def-iota} exists and we have $\iota(q) = \upiota(\kappa)$.
\end{theorem}

\begin{proof}[Proof of Theorem~\ref{thm:iota} assuming Theorem~\ref{thm:iota-continuum-easy} and Proposition~\ref{prop:event-A}]
    Fix $\delta \in (0,\frac{1}{2} \wedge \delta_0)$, where $\delta_0$ is the constant in Theorem~\ref{thm:iota-continuum-easy}.
    Fix $\varepsilon>0$ sufficiently small, and consider $r=\varepsilon R$. Fix $\theta \in (0,1)$ small. Assuming Conjecture~\ref{conj:fk-to-cle}, for all sufficiently large $R$, with high probability all loops surrounding the origin in the loop configuration with diameters larger than $r/4$, after rescaled by $1/R$, are each within distance $\frac{\delta}{2}(\varepsilon-\varepsilon^{1+\theta}) \le \frac{\delta}{2}(\varepsilon^{1-\theta}-\varepsilon)$ from the corresponding CLE loop.
    On this event and under the free boundary conditions, we have $\mathsf{A}_{\varepsilon^{1-\theta},(1+\delta)\varepsilon^{1+\theta}}^{\mathrm{even}} \subseteq A(\varepsilon R;\delta) \subseteq \mathsf{A}_{\varepsilon^{1+\theta},(1+\delta)\varepsilon^{1-\theta}}^{\mathrm{even}}$, which implies $\limsup_{R \to \infty} \phi_{\Lambda_R,p_c,q}^0[A(\varepsilon R;\delta)] \le \mathbb{P}[\mathsf{A}_{\varepsilon^{1-\theta},(1+\delta)\varepsilon^{1+\theta}}^{\mathrm{even}}]$ and $\liminf_{R \to \infty} \phi_{\Lambda_R,p_c,q}^0[A(\varepsilon R;\delta)] \ge \mathbb{P}[\mathsf{A}_{\varepsilon^{1+\theta},(1+\delta)\varepsilon^{1-\theta}}^{\mathrm{even}}]$. Letting $\theta \to 0$, we find that $\lim_{R \to \infty} \phi_{\Lambda_R,p_c,q}^0[A(\varepsilon R;\delta)]=\mathbb{P}[\mathsf{A}_{\varepsilon,(1+\delta)\varepsilon}^{\mathrm{even}}]$. Similarly, we have $\lim_{R \to \infty} \phi_{\Lambda_R,p_c,q}^1[A(\varepsilon R;\delta)]=\mathbb{P}[\mathsf{A}_{\varepsilon,(1+\delta)\varepsilon}^{\mathrm{odd}}]$. Combining Proposition~\ref{prop:event-A} and Theorem~\ref{thm:iota-continuum-easy}, we get $\liminf_{R \to \infty} \Delta_{p_c,q}(\varepsilon R,R) \asymp \limsup_{R \to \infty} \Delta_{p_c,q}(\varepsilon R,R) \asymp \varepsilon^{\upiota(\kappa)}$. It then follows from a standard application of quasi-multiplicativity (see~\cite[Theorem 1.6(ii)]{DCM22}) that $\Delta_{p_c,q}(r,R)=(r/R)^{\upiota(\kappa)+o(1)}$ as $R/r \to \infty$.
\end{proof}

\begin{proof}[Proof of Theorem~\ref{thm:iota-main}]
    This follows directly from Theorems~\ref{thm:iota-continuum-easy} and~\ref{thm:iota}.
\end{proof}

For $q=2$, the convergence of critical FK-Ising percolation interfaces to CLE$_{16/3}$ was established in a series of works~\cite{Smi10,KS19,KS16}. Therefore, in this case, Theorem~\ref{thm:iota-main} can be stated unconditionally. This result is part of~\cite[Theorem 1.13]{DCM22}, but we give a new proof.

\begin{corollary}\label{cor:FK-Ising}
    The mixing rate exponent of the critical FK-Ising model is 1.
\end{corollary}

\subsection{Proof overview}\label{subsec:proof-overview}

Let us provide a proof outline and organization of the paper.

In Section~\ref{sec:discrete}, we will prove Proposition~\ref{prop:event-A}, which provides an alternative description of the mixing rate of FK percolation. The proof is based on the increasing coupling of FK percolation developed in~\cite{DCM22}.
The key ingredient is Proposition~\ref{prop:boost-bc}, which says that boosting the boundary conditions would increase the probability of the event $A(r;\delta)$ by a uniformly positive amount. The main difficulty is that $A(r;\delta)$ is not an increasing event, so a direct increasing coupling does not apply. We overcome this using a modified coupling approach.

In Sections~\ref{sec:continuum}--\ref{sec:simple}, we prove Theorems~\ref{thm:iota-continuum-easy} and~\ref{thm:cle-partition} which are purely about CLE and require no knowledge of FK percolation. In Section~\ref{sec:continuum}, we first prove Theorem~\ref{thm:iota-continuum-easy} based on Theorem~\ref{thm:cle-partition}. This is achieved by integrating Theorem~\ref{thm:cle-partition} over the loop configuration in an annulus. One subtlety is that there might be multiple non-contractible CLE loops in the annulus. We address this using the inclusion-exclusion principle and by choosing the annulus to be sufficiently thin so that with high probability there exists at most one non-contractible loop.

In Section~\ref{sec:lqg}, we prove Theorem~\ref{thm:cle-partition} for the non-simple region $\kappa \in (4,8)$ using the coupling of CLE and Liouville quantum gravity (LQG). This idea originated from Sheffield's observation that SLE curves can be viewed as the conformal welding interfaces between LQG surfaces~\cite{She16, DMS21}. This method was carried forward in~\cite{AHS21} to obtain various integrability results for SLE curves, thanks to the exact solvability of Liouville conformal field theory (LCFT)~\cite{DKRV-sphere, Liouville-review}.
Our approach is another application of this idea, based on LQG on the annulus. In Section~\ref{subsec:welding}, using the coupling of target-invariant SLE processes and LQG, we prove that a CLE loop of even or odd nesting level can be described as the conformal welding interface of a generalized quantum disk~\cite{DMS21,MSW21-nonsimple,AHSY23} and a generalized quantum annulus (Theorem~\ref{thm:odd-even-loop-welding}). The key observation in proving Theorem~\ref{thm:cle-partition} is that, while the laws of moduli of these two quantum annuli are different, given the modulus, they are described by the same Liouville fields. Thus, the Radon-Nikodym derivative of $\mathsf{m}_\kappa^{\mathrm{odd}}$ and $\mathsf{m}_\kappa^{\mathrm{even}}$ can be encoded by the modulus laws of these two quantum annuli. Following the techniques of~\cite{ARS22}, the laws of moduli can be extracted from the boundary length joint distributions of these quantum annuli which are in turn described by LCFT, thereby proving Theorem~\ref{thm:cle-partition} for $\kappa \in (4,8)$. We refer to Sections~\ref{subsec:symmetry} and~\ref{subsec:pfn} for details. 
The simple region $\kappa \in (8/3,4]$ of Theorem~\ref{thm:cle-partition} follows a similar strategy, so we will only sketch the proof in Section~\ref{sec:simple}.

\subsection{Outlook}\label{subsec:outlook}

In this section, we discuss some future work and open questions.

\begin{itemize} 
\item We first discuss the relation between the mixing rate exponent and the nested loop exponent introduced in~\cite{MN04, dN83}. For critical FK percolation on $\Lambda_R$ with wired boundary conditions, let $\ell_R$ be the number of percolation interfaces that surround the origin. For fixed $a \in \mathbb{R}$, the nested loop exponent $X_{\rm NL}(a)$ is defined to satisfy
\begin{equation}\label{eq:nested-loop}
\phi^1_{\Lambda_R, p_c,q} [a^{\ell_R}] = R^{-X_{\rm NL}(a) + o(1)} \quad \mbox{as } R \rightarrow \infty.
\end{equation}
When $a>0$, assuming the conformal invariance conjecture, the nested loop exponent can be derived based on~\cite{SSW09}; see Section 1.2 of~\cite{ASYZ24}. When $a<0$, it is \emph{a prior} unclear whether $X_{\rm NL}(a)$ exists. However, when $a = - 1$, by a duality argument, one can show that $X_{\rm NL}(-1)$ agrees with the mixing rate exponent. It would be interesting to derive the nested loop exponent for other $a<0$.

\item The CLE mixing rate defined in~\eqref{eq:mixing-rate-cont-def} is reminiscent of the twist operators for the $\mathrm{O}(n)$ loop model considered in~\cite{CG06}.
In particular, the scaling dimension of the twist operator is the same as the CLE mixing rate exponent. We expect that the CLE mixing rate can be used to construct the twist operator in the continuum by appropriately sending $\delta$ and $r$ to 0. The same mechanism should extend to multiple points by taking into account the Radon-Nikodym derivatives of even/odd nesting loops surrounding them. More interestingly, from the conformal field theory (CFT) perspective~\cite{BPZ84a}, the twist operator is expected to be a degenerate operator at level $(1,2)$~\cite{CG06} and thus to satisfy the BPZ equation~\cite{BPZ84a,Pel19}.
This gives a possible way to define and compute a family of CLE correlation functions with more than three points. See~\cite{ACSW-three-pt} for recent work on CLE three-point correlation functions.

\item In the case of the Ising model, the above approach can be used to construct the Ising energy operator using CLE$_{16/3}$ and the Ising spin operator using CLE$_3$. Furthermore, CLE$_{16/3}$ and CLE$_3$ can be coupled together through a continuum version of the Edwards-Sokal coupling~\cite{MSW2017}. Under this coupling, we can in fact construct the multiple-point correlation functions involving both the Ising spin and energy operators. In future work with Xin Sun, we plan to carry this out and extend to other CFT minimal models~\cite{BPZ84a}. We note that recently \cite{CF24a-fk} gives another way to construct the Ising spin correlation functions using CLE$_{16/3}$ gasket.

\end{itemize}

\noindent{\bf Notations.}
The parameters $q$ and $\kappa$ will be fixed throughout this paper. For two functions $f,g:X \to (0,\infty)$ defined on some space $X$, we write $f \asymp g$ if there exist constants $0<c<C$ \emph{depending only on} $q$ or $\kappa$ such that $cf(x) \le g(x) \le Cf(x)$ for all $x \in X$.
For a non-self-crossing loop $\lambda \subset \mathbb{C}$, let $\lambda^o$ be the open set of points it surrounds, that is, the points with non-zero winding number.

\medskip
\noindent\textbf{Acknowledgements.} We thank Xin Sun for enlightening discussions and encouragement.
We also thank Tiancheng He and Yu Feng for their helpful discussions and comments on an earlier draft of this paper. H.L. and B.W.\ are partially supported by National Key R\&D Program of China (No.\ 2023YFA1010700). Z.Z.\ is partially supported by NSF grant DMS-1953848.
\section{FK percolation and mixing rate}
\label{sec:discrete}

The main purpose of this section is to prove Proposition~\ref{prop:event-A}, providing an equivalent description of the mixing rate of FK percolation. The proof mainly builds on an increasing coupling between FK percolation with different boundary conditions, developed in~\cite{DCM22}. We will first recall this coupling and some basic properties in Section~\ref{subsec:prelim-fk}, and then prove Proposition~\ref{prop:event-A} using a modified coupling method.

\subsection{FK percolation: basic properties and increasing couplings}\label{subsec:prelim-fk}

First, we define the FK percolation model and briefly recall its basic properties. For a detailed review, we refer the reader to~\cite{Grimmett-rcm,BDC-notes,DC-notes}.

The planar square lattice is the graph $(\mathbb{Z}^2,E(\mathbb{Z}^2))$ with vertex set $\mathbb{Z}^2$ and edges between nearest neighbors. Slightly abusing notations, we refer to the graph itself as $\mathbb{Z}^2$. For a finite subgraph $G=(V,E)$ of $\mathbb{Z}^2$, we define its vertex boundary as $\partial V:=\{v\in V:\deg_G(v)<4\}$. An edge configuration on $G$ is an element $\omega \in \{0,1\}^E$, where an edge $e \in E$ is said to be \emph{open} if $\omega_e=1$, and \emph{closed} otherwise. With a slight abuse of notation, we can view $\omega$ as a subgraph of $G$ with vertex set $V$ and edge set $o(\omega):=\{e \in E: \omega_e=1\}$. Partitions $\xi$ of $\partial V$ are called \emph{boundary conditions} on $G$.

\begin{definition}\label{def:fk}
    For $p \in (0,1)$ and $q>0$, the FK percolation on $G$ with edge weight $p$, cluster weight $q$ and boundary conditions $\xi$ is a probability measure on $\{0,1\}^E$ given by
\[ \phi_{G,p,q}^\xi[\omega]:=\frac{1}{Z_{G,p,q}^\xi} p^{|o(\omega)|} (1-p)^{|E \setminus o(\omega)|} q^{k(\omega^\xi)} \,, \]
where $\omega^\xi$ is the graph obtained from $\omega$ by identifying vertices belonging to the same partition element of $\xi$ as wired together and let $k(\omega^\xi)$ be the number of connected components of the corresponding graph. The normalizing constant $Z_{G,p,q}^\xi$ is called the partition function.
\end{definition}

For boundary conditions $\xi$ and $\xi'$, we write $\xi \le \xi'$ if each partition element of $\xi$ is a subset of a partition element of $\xi'$. The \emph{free} boundary conditions (denoted by 0) refer to the finest partition that each vertex in $\partial V$ forms a singleton, and the \emph{wired} boundary conditions (denoted by $1$) refer to the coarsest partition where the whole set $\partial V$ is a partition element.

For two edge configurations $\omega,\omega' \in \{0,1\}^E$, we write $\omega \le \omega'$ if $\omega_e \le \omega'_e$ for all $e \in E$. An event $A$ is called \emph{increasing} if for any $\omega \le \omega'$, $\omega \in A$ implies $\omega' \in A$. For $q \ge 1$, FK percolation is increasing with respect to the boundary condition: for any two boundary conditions $\xi \le \xi'$ and an increasing event $A$, $\phi_{G,p,q}^{\xi'}[A] \ge \phi_{G,p,q}^{\xi}[A]$. In particular, this implies that both $\Delta_p(R)$ and $\Delta_p(r,R)$ defined in~\eqref{eq:mixing-rate-def-1} and~\eqref{eq:mixing-rate-def-2} are non-negative.

The FK percolation model enjoys the following domain Markov property. Consider a subgraph $G'=(V',E')$ of $G$. For any boundary conditions $\xi$ on $G$ and any fixed configuration $\tilde\omega \in \{0,1\}^{E' \setminus E}$,
\begin{equation}\label{eq:dmp}
    \phi_{G,p,q}^\xi[\cdot_{|E'} \mid \omega_e=\tilde\omega_e \mbox{ for all } e \in E \setminus E']=\phi_{G',p,q}^{\tilde\xi}[\cdot] \,,
\end{equation}
where $\tilde\xi$ is the boundary conditions induced by $\xi$ on $G'$ (i.e. $u,v \in V'$ are wired together if they are connected in $\tilde\omega^\xi$).

In this paper, we focus on the critical FK percolation, referring to the case $p=p_c(q):=\frac{\sqrt{q}}{1+\sqrt{q}}$ ~\cite{BDC12}. The cluster weight $q \in (1,4]$ will also be fixed throughout, therefore we omit the subscripts and write $\phi_G^\xi$ (resp. $\Delta(r,R)$) instead of $\phi_{G,p_c,q}^\xi$ (resp. $\Delta_{p_c,q}(r,R)$) for simplicity.

Crossing probabilities in rectangles are crucial observables in (near-)critical lattice models. Crossing estimates originated from the study of Bernoulli percolation~\cite{Rus78,Rus81,SW78} and were later called Russo-Seymour-Welsh (RSW) theory. This theory has been widely extended to several correlated percolation models~\cite{DCHN11,DCST17,DCT20,KST23}. More recently, stronger crossing estimates for FK percolation in general quads were derived in~\cite{DCMT21}.
For our purpose, we are content with the following formulation from~\cite[Theorem 2.1]{DCM22}. Recall the crossing event $\mathcal{C}(D)$ for a rectangle $D$ defined right before~\eqref{eq:mixing-rate-def-2}.

\begin{theorem}[Crossing estimates]\label{thm:rsw}
    Fix $q \in (1,4]$. For any $\varepsilon>0$ and $L>1$, there exists $c=c(q,\varepsilon,L)>0$ such that for any graph $G=(V,E)$ containing $[-\varepsilon n, (1+\varepsilon)n] \times [-\varepsilon m, (1+\varepsilon)m]$ and every boundary conditions $\xi$ on $G$, if $L^{-1} \le \frac{m}{n} \le L$, then
    \begin{equation}\label{eq:rsw-1}
        c \le \phi_G^\xi \big[\mathcal{C}([0,n] \times [0,m]) \big] \le 1-c \,.
    \end{equation}
\end{theorem}

Next, we discuss the \emph{increasing coupling} of FK percolation models with different boundary conditions. For a finite graph $G=(V,E)$ and boundary conditions $\xi \le \xi'$, we may construct a probability measure $\mathbb{P}$ via a (random) exploration tree $(e_1,\ldots,e_{|E|})$ as follows. We write $e_{[t]}=(e_1,\ldots,e_t)$ for the edges revealed before time $t$ and $G_t=G \setminus \{e_1,\ldots,e_t\}$ for the unexplored region.
Let $\omega_{[t]}=(\omega_{e_1},\ldots,\omega_{e_t})$ and $\xi_t$ be the boundary conditions on $G_t$ induced by $\omega_{[t]}^\xi$. We similarly define $\omega'_{[t]}=(\omega_{e_1},\ldots,\omega_{e_t})$ and $\xi'_t$.
Starting with a fixed edge $e_1$, for each $k=1,\ldots,|E|$,
\begin{itemize}
    \item Let $\mathsf{p}_k=\phi_{G_{k-1}}^{\xi_{k-1}}[\omega_{e_k}]$ and $\mathsf{p}'_k=\phi_{G_{k-1}}^{\xi'_{k-1}}[\omega'_{e_k}]$, then independently sample $U_k$ from a uniform distribution on $[0,1]$.
    \item Declare whether $e_k$ is open or not in $\omega$ and $\omega'$ by
    \[ \omega_{e_k}=\mathbf{1}_{U_k \le \mathsf{p}_k} \quad \mbox{and} \quad \omega_{e_k}=\mathbf{1}_{U_k \le \mathsf{p}'_k} \,. \]
    \item Choose the next queried edge $e_{k+1}$ from $E \setminus \{e_1,\ldots,e_k\}$ via a \emph{deterministic} function of $(\omega_{[k]}, \omega'_{[k]})$.
\end{itemize}

By the domain Markov property~\eqref{eq:dmp} and monotonicity, the marginal distributions of $\omega$ and $\omega'$ are $\phi_G^\xi$ and $\phi_G^{\xi'}$, respectively, and $\omega \le \omega'$ almost surely. Moreover, for any \emph{stopping time} of the exploration tree, i.e. a random variable $\tau$ taking values in $\{1,\ldots,|E|,\infty\}$ such that $\{\tau \le k\}$ is measurable with respect to $(\omega_{[k]}, \omega'_{[k]})$, the conditional distributions of $\omega$ and $\omega'$ on $G_\tau$ given $(\omega_{[\tau]},\omega'_{[\tau]})$ are $\phi_{G_\tau}^{\xi_\tau}$ and $\phi_{G_\tau}^{\xi'_\tau}$, respectively. This allows us to change exploration algorithms at stopping times.

The aforementioned coupling method is general in that any exploration rule can be used to construct such an increasing coupling. In~\cite{DCM22}, the authors construct a specific revealment algorithm and show that the mixing rate can be interpreted as the probability of non-coupling. To start with, we recall some terminology in~\cite{DCM22}.

\begin{definition}\label{def:flower-domain}
    An \emph{inner flower domain} on $\Lambda_R$ is a simply connected finite domain $\mathcal{F} \subset \mathbb{Z}^2$ containing $\Lambda_R$ whose boundary consists of an even number of nearest-neighbor arcs $(a_i a_{i+1})_{1 \le i \le 2k}$ with each endpoint $a_i$ lying on $\partial \Lambda_R$, where $a_{2k+1}=a_1$ by convention.
    We say $\mathcal{F}$ is $\frac{1}{2}$-well-separated if the Euclidean distance between any two distinct endpoints $a_i$ and $a_j$ is greater than $R/2$.
\end{definition}

\begin{definition}\label{def:boost-bc}
    For an inner flower domain $\mathcal{F}$ as in Definition~\ref{def:flower-domain}, we call a pair of boundary conditions $(\xi,\xi')$ \emph{boosting} if there exist $\xi \le \xi_1 \le \xi_2 \le \xi'$ such that
    \begin{itemize}
        \item both in $\xi_1$ and $\xi_2$, all vertices on the arc $(a_j a_{j+1})$ are wired together for each odd $j$, and all vertices on the arc $(a_j a_{j+1})$ (except for endpoints) are singletons for each even $j$;
        \item for some odd $j_1<j_2$, two arcs $(a_{j_1} a_{j_1+1})$ and $(a_{j_2} a_{j_2+1})$ are wired together in $\xi_2$, but not in $\xi_1$.
    \end{itemize}
    In particular, $\xi=0$ and $\xi'=1$ always form a boosting pair.
\end{definition}

The following theorem is a combination of Theorems 4.1 and 4.9 of~\cite{DCM22}.

\begin{theorem}\label{thm:coupling-fk}
    Fix $q \in (1,4]$. Let $1 \le 4r<R$ and $\mathcal{F}$ be a $\frac{1}{2}$-well-separated inner flower domain on $\Lambda_R$. For any boosting pair of boundary conditions $(\xi,\xi')$ on $\mathcal{F}$, there exists an increasing coupling $\mathbb{P}$ of $\phi_{\mathcal{F}}^\xi$ and $\phi_{\mathcal{F}}^{\xi'}$ on $\mathcal{F} \setminus \Lambda_r$ and a stopping time $\tau$ such that on the event $\{\tau<\infty\}$,
    \begin{itemize}
        \item the unexplored region $\mathcal{F}_\tau=\mathcal{F} \setminus \{e_1,\ldots,e_\tau\}$ is a $\frac{1}{2}$-well-separated inner flower domain on $\Lambda_r$;
        \item the boundary conditions induced by $\omega_{[\tau]}^\xi$ and $(\omega'_{[\tau]})^{\xi'}$ on $\mathcal{F}_\tau$ form a boosting pair.
    \end{itemize}
    Additionally, if $\zeta$ and $\zeta'$ are the boundary conditions induced by $\omega^\xi$ and $(\omega')^{\xi'}$ on $\partial \Lambda_r$, respectively, then
    \begin{equation}\label{eq:coupling-fk}
        \mathbb{P}[\zeta \neq \zeta'] \asymp \mathbb{P}[\tau<\infty] \asymp \Delta(r,R) \,,
    \end{equation}
    where the constants in $\asymp$ depend only on $q$.
\end{theorem}

Finally, we discuss the dual model and loop configuration of FK percolation. For a subgraph $G=(V,E)$ of $\mathbb{Z}^2$, its dual graph $G^*=(V^*,E^*)$ is a subgraph of $\mathbb{Z}^2+(\frac{1}{2},\frac{1}{2})$, where each dual vertex in $V^*$ corresponds to a face of $G$, and each dual edge $e^* \in E^*$ intersects a unique primal edge $e \in E$. Given an edge configuration $\omega$ on $G$, we define its \emph{dual} configuration $\omega^*$ on $G^*$ by declaring $\omega^*_{e^*}=1-\omega_e$. We say $e^*$ is \emph{dual open} if $\omega^*_{e^*}=1$. For $q \ge 1$, the critical point $p_c(q)$ is self-dual~\cite{BDC12} in the sense that the dual of a critical FK$_q$ percolation with free boundary conditions is that with wired boundary conditions, and vice versa. 

Let $G^\diamond$ be the \emph{medial} graph with vertex set given by the midpoints of the edges of $G$ and $G^*$, and edges between nearest neighbors. For an edge configuration $\omega$ on $G$, we draw all non-self-crossing paths on $G^\diamond$ that do not cross any edges of $\omega$ and $\omega^*$. When the boundary conditions are free or wired, each path is in fact a loop along the inner or outer boundary of a cluster of $\omega$. This collection of loops is called the \emph{loop configuration} of $\omega$. For an illustrative example, we direct the reader to~\cite[Figure 14]{DCM22} or~\cite[Figure 2]{KS19}.

The loop configuration admits a natural parity structure as follows. We associate each loop with a \emph{nesting level} equal to the number of distinct loops surrounding it plus 1, and separate the loops into even and odd ones depending on the parity of their nesting levels. For instance, the outermost loops have nesting level 1 and are odd. Then, under the free boundary conditions, a primal cluster is squeezed inside an odd loop and outside of all the even loops it surrounds, and a dual cluster is squeezed inside an even loop and outside of all the odd loops it surrounds. By duality, the opposite holds under the wired boundary conditions.

\subsection{An equivalent description of the mixing rate: proof of Proposition~\ref{prop:event-A}}\label{subsec:event-A}

Fix $q \in (1,4]$ and $\delta \in (0,\frac{1}{2})$. For $1 \le r<R$, write $\Lambda_{r,R}=\Lambda_R \setminus \Lambda_r$. For an edge configuration $\omega$, a primal (resp. dual) open path is a sequence of nearest-neighbor primal (resp. dual) vertices $u_1,u_2,\ldots,u_n$ with all edges $u_i u_{i+1}$ primal (resp. dual) open. Paths can be naturally embedded in $\mathbb{R}^2$ by viewing edges as closed line segments. A circuit is a path with the same starting and ending points. For a primal circuit $\lambda$ viewed as a loop on $\mathbb{C}$, let $\lambda^o$ be the open set of points it surrounds, and let $\overline{\lambda^o}$ be its closure. Let $D(\lambda)$ be the domain enclosed by $\lambda$, that is, the subgraph with vertex set $\overline{\lambda^o} \cap \mathbb{Z}^2$ and edge set consisting of those edges that are completely contained in $\overline{\lambda^o}$, with edges lying on $\lambda$ excluded.

Before moving onto the proof of Proposition~\ref{prop:event-A}, we remind that since we consider $\delta>0$ fixed, there exists $\tilde c=\tilde c(\delta,q)>0$ such that $\phi_{\Lambda_R}^{\xi}[A(r;\delta)] \ge \tilde c$ uniformly for any boundary conditions $\xi$, thanks to the RSW theory. Hence, it will be equivalent to proving~\eqref{eq:event-A} with the denominator removed. We intentionally use this formulation to be consistent with the continuum mixing rate~\eqref{eq:mixing-rate-cont-def}. In view of Theorem~\ref{thm:iota-continuum-easy}, we also expect (but do not prove) that~\eqref{eq:event-A} holds uniformly for all small $\delta$; that is, the constants $c$ and $C$ depend only on $q$.

The upper bound of~\eqref{eq:event-A} is rather obvious: Let $\mathbb{P}$ be the increasing coupling of $\omega^0$ and $\omega^1$ on $\Lambda_{2r,R}$, whose boundary conditions are free and wired, respectively, on $\partial \Lambda_R$ as in Theorem~\ref{thm:coupling-fk}. Let $\zeta$ and $\zeta'$ be the boundary conditions induced on $\partial \Lambda_{2r}$ by $\omega^0$ and $\omega^1$, respectively, then
\begin{equation}\label{eq:upper-bound}
    \phi_{\Lambda_R}^1[A(r;\delta)]-\phi_{\Lambda_R}^0[A(r;\delta)] \le \mathbb{P}[\zeta \neq \zeta'] \cdot \max_{\xi} \phi_{\Lambda_{2r}}^\xi[A(r;\delta)] \le C \cdot \Delta(r,R) \phi_{\Lambda_R}^0[A(r;\delta)],
\end{equation}
where in the second inequality, we used~\eqref{eq:coupling-fk} and the mixing property $\max_{\xi} \phi_{\Lambda_{2r}}^\xi[A(r;\delta)] \le C_1 \cdot \phi_{\Lambda_R}^0[A(r;\delta)]$ (see, e.g.~\cite[Theorem 5]{DCST17}).

Now we proceed to the proof of the converse bound, for which we will need the following lemma adapted from~\cite{DCM22}. For $1 \le r<R$, let $\mathcal{E}_{r,R}$ be the event that there exists a non-contractible primal open circuit in the annulus $\Lambda_{r,R}$.

\begin{lemma}\label{lem:boost-bc}
    For each $\delta \in (0,\frac{1}{2})$, there exists $\mathsf{c}=\mathsf{c}(\delta,q)>0$ such that for all $r \ge 100\delta^{-1}$, every $\frac{1}{2}$-well-separated inner flower domain $\mathcal{F}$ on $\Lambda_{2r}$, and all pair of boosters $(\xi,\xi')$ of boundary conditions on $\mathcal{F}$, we have
    \[ \phi_{\mathcal{F}}^{\xi'}\big[\mathcal{E}_{(1+\delta')r,(1+\delta)r}\big]-\phi_{\mathcal{F}}^{\xi}\big[\mathcal{E}_{(1+\delta')r,(1+\delta)r}\big] \ge \mathsf{c} \]
    uniformly for all $\delta' \in [\frac{\delta}{3},\frac{2\delta}{3}]$.
\end{lemma}

\begin{proof}
    In~\cite[Theorem 3.6]{DCM22}, the authors proved the case for the event $\mathcal{E}_{r,2r}$. The exact same argument applies to $\mathcal{E}_{(1+\delta')r,(1+\delta)r}$ for any $0<\delta'<\delta<\frac{1}{2}$, implying
    \[ \phi_{\mathcal{F}}^{\xi'}\big[\mathcal{E}_{(1+\delta')r,(1+\delta)r}\big]-\phi_{\mathcal{F}}^{\xi}\big[\mathcal{E}_{(1+\delta')r,(1+\delta)r}\big] \ge \mathsf{c}_1(\delta',\delta,q) \]
    for some constant $\mathsf{c}_1 = \mathsf{c}_1(\delta',\delta,q)>0$. Moreover, the lower bound $\mathsf{c}_1(\delta',\delta,q)$ as obtained in \cite[Theorem 3.6]{DCM22} arises from a combination of several crossing estimates. Fix $\delta>0$. For $\delta' \in [\frac{\delta}{3},\frac{2\delta}{3}]$, the ratio $\frac{1+\delta}{1+\delta'} \ge 1+\frac{\delta}{9}$ is uniformly bounded away from 1, hence we can get a uniform positive lower bound $\mathsf{c}(\delta,q)=\inf_{\delta' \in [\frac{\delta}{3},\frac{2\delta}{3}]} \mathsf{c}_1(\delta',\delta,q)>0$ from those RSW-type estimates.
\end{proof}

To prove the lower bound, the key ingredient will be the following proposition, which is an analog of Lemma~\ref{lem:boost-bc}. However, the event $A(r;\delta)$ is not increasing, so we need a modified version of the coupling.

\begin{proposition}\label{prop:boost-bc}
    For each $\delta \in (0,\frac{1}{2})$, there exists $c_1=c_1(\delta,q)>0$ such that for all $r \ge 100\delta^{-1}$, every $\frac{1}{2}$-well-separated inner flower domain $\mathcal{F}$ on $\Lambda_{2r}$, and every boosting pair of boundary conditions $(\xi,\xi')$ on $\mathcal{F}$, we have
    \[ \phi_{\mathcal{F}}^{\xi'}[A(r;\delta)]-\phi_{\mathcal{F}}^{\xi}[A(r;\delta)] \ge c_1. \]
\end{proposition}

\begin{proof}
    We recommend inspecting Figure~\ref{fig:event-A}, right, for a quick impression of the proof idea.
    Consider the coupling $\mathbf{P}$ of $\omega'$ and $\omega$ on $\mathcal{F}$, whose boundary conditions are $\xi'$ and $\xi$, respectively, constructed as follows.
    \begin{itemize}
        \item First, we reveal edges from the outside in (starting from $\partial \mathcal{F}$) for both $\omega$ and $\omega'$, until the stopping time $\tau_1'$ that we discover a non-contractible open circuit of $\omega'$ in $\Lambda_{r,(1+\delta)r}$, and let $\eta'$ be this open circuit. If no such open circuits exist, let $\tau_1'=\infty$ and reveal the remaining edges in arbitrary order. Note that $\mathbf{P}$ is an increasing coupling at this stage, that is $\omega_e \le \omega_e'$ for all edges $e$ revealed before $\tau_1'$. Hence no $\omega$-open circuits are found before $\tau_1'$ (it is however possible that $\eta'$ is also open in $\omega$).
        \item For $\tau_1'<\infty$, we reveal edges from the outside in (starting from $\eta'$) \emph{only for} $\omega$, until the stopping time $\tau_1$ that we discovered an open circuit of $\omega$ in $\Lambda_{r,(1+\delta)r}$, and let $\eta$ be this open circuit. If $\tau_1'=\infty$ or no such open circuits exist, set $\tau_1=\infty$, reveal the remaining edges in arbitrary order and define $\eta$ as the outer boundary circuit of $\Lambda_r$. Otherwise $\tau_1' \le \tau_1<\infty$ and $D(\eta) \subseteq D(\eta')$, we reveal edges in $D(\eta') \setminus D(\eta)$ for $\omega'$ independently in arbitrary order.
        \item For $\tau_1<\infty$, let $\zeta'$ and $\zeta$ be the boundary conditions induced on $D(\eta)$ by $\omega'$ and $\omega$, then $\zeta$ is wired, thus $\zeta' \le \zeta$. Therefore, we may reveal all the remaining edges from the outside in (starting from $\eta$) for both $\omega$ and $\omega'$ using an arbitrary exploration algorithm, so that $\mathbf{P}$ is an increasing coupling between $\phi_{D(\eta)}^{\zeta'}$ and $\phi_{D(\eta)}^1$, i.e. $\omega_e' \le \omega_e$ for all edges $e$ in $D(\eta)$.
    \end{itemize}
    Let $\mathbf{E}$ be the expectation associated with $\mathbf{P}$. Under this coupling, we have
    \begin{equation}\label{eq:phi_10}
        \phi_{\mathcal{F}}^{\xi'}[A(r;\delta)]=\mathbf{E}\left[\1_{\tau_1'<\infty} \phi_{D(\eta')}^1[E_{\eta'}]\right], \quad \mbox{and} \quad \phi_{\mathcal{F}}^{\xi}[A(r;\delta)]=\mathbf{E}\left[\1_{\tau_1<\infty} \phi_{D(\eta)}^1[E_{\eta}]\right],
    \end{equation}
    where, for each circuit $\mathcal{L}$ in $\Lambda_{r,(1+\delta)r}$, we denote by $E_{\mathcal{L}}$ the event that there exists a non-contractible dual open circuit in the annular domain $D(\mathcal{L}) \setminus \Lambda_r$. Recall that when $\tau_1'<\infty=\tau_1$, $\eta$ is defined as the outer boundary circuit of $\Lambda_r$; in this case, we naturally set $\phi_{D(\eta)}^1[E_{\eta}]=0$. By~\eqref{eq:phi_10} we may write
    \begin{equation}\label{eq:phi_1-phi_0}
        \phi_{\mathcal{F}}^{\xi'}[A(r;\delta)]-\phi_{\mathcal{F}}^{\xi}[A(r;\delta)]=\mathbf{E}\left[\1_{\tau_1'<\infty} \left(\phi_{D(\eta')}^1[E_{\eta'}]-\phi_{D(\eta)}^1[E_{\eta}] \right) \right].
    \end{equation}
    For any realization of $\eta'$ and $\eta$, since $E_{\eta}$ is a decreasing event and $\mathbf{P}$ is an increasing coupling in $D(\eta)$, the occurrence of $E_{\eta}$ for $\omega$ implies the occurrence of $E_{\eta'}$ for $\omega'$.
    In particular, this shows that~\eqref{eq:phi_1-phi_0} is always non-negative as long as $\xi \le \xi'$ (not necessarily boosting).
    To show that~\eqref{eq:phi_1-phi_0} is lower bounded by a uniform positive constant when $(\xi,\xi')$ forms a boosting pair, it suffices to show that $A(r;\delta)$ occurs for $\omega'$ but not for $\omega$ on some event with strictly positive probability.

    \begin{figure}
	\centering
	\begin{tabular}{cc} 
		\includegraphics[scale=0.4]{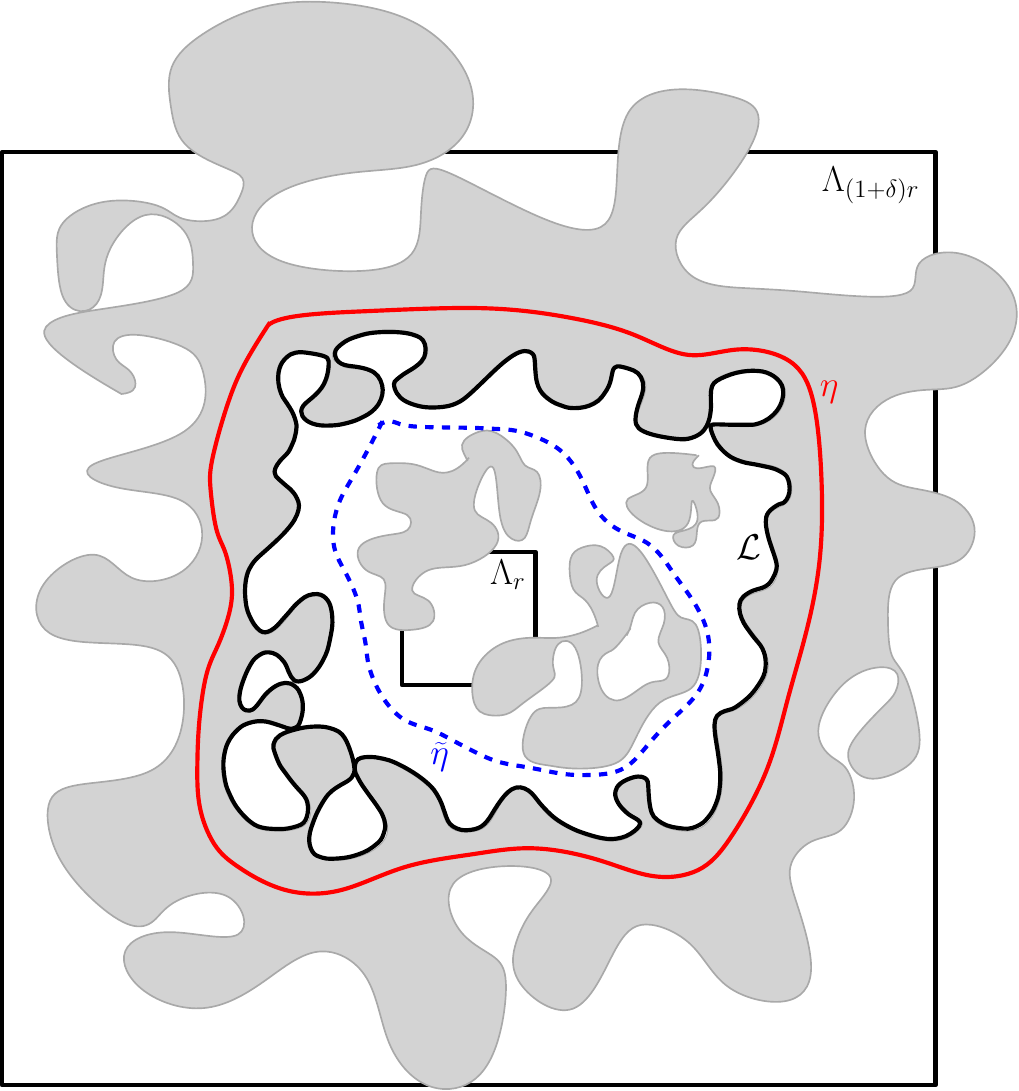}
		& \ \ \ \
		\includegraphics[scale=0.4]{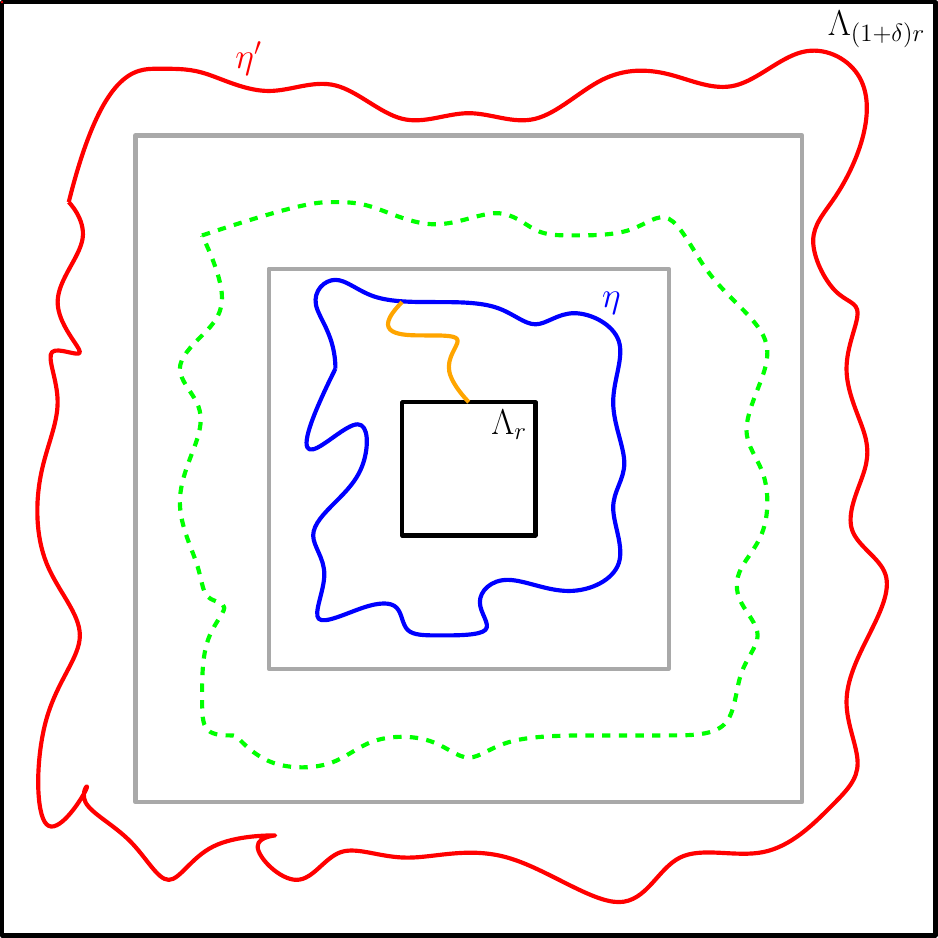}
	\end{tabular}
	\caption{\textbf{Left}: Primal and dual clusters are shown in gray and white, respectively. The event $A(r;\delta)$ requires a non-contractible dual circuit $\widetilde{\eta}$ (in dashed blue) surrounded by a non-contractible primal circuit $\eta$ (in red) in $\Lambda_{r,(1+\delta)r}$. Equivalently, the outer boundary $\mathcal{L}$ of a dual cluster lies in $\Lambda_{r,(1+\delta)r}$ and surrounds the origin, which should converge to a CLE loop of odd or even nesting level depending on the boundary conditions. \textbf{Right}: An illustration of the proof of Proposition~\ref{prop:boost-bc}. We first reveal edges from outside in to find the outermost $\omega'$-open circuit $\eta'$ (in red) and $\omega$-open circuit $\eta$ (in blue) in $\Lambda_{r,(1+\delta)r}$. We partition $\Lambda_{(1+\frac{\delta}{3})r,(1+\frac{2\delta}{3})r}$ into $N$ annular tubes of width $\frac{\delta}{3N}$, then with positive probability $\eta'$ and $\eta$ are separated by one of these tubes (in gray). On the event $(\clubsuit) \cap (\spadesuit)$ (shown in green and orange), the event $A(r;\delta)$ occurs for $\omega'$ but not for $\omega$. (Figures are not to scale for illustrative reasons.)}
    \label{fig:event-A} 
    \end{figure}

    Let $\mathsf{c}=\mathsf{c}(\delta,q)$ be the constant in Lemma~\ref{lem:boost-bc}. Pick a positive integer $N=\lceil\frac{2}{\mathsf{c}}\rceil+1$ and let $\delta_1=\frac{\delta}{3N}>0$. We claim that for any $r \ge 100\delta^{-1}$, there exists $k \in \{1,2,\ldots,N\}$ (which may depend on $r$) such that
    \begin{equation}\label{eq:probability-k-1}
        \phi_{\mathcal{F}}^{\xi'}\Big[ \mathcal{E}_{(1+\frac{\delta}{3}+k\delta_1)r,(1+\delta)r}\Big] \ge \phi_{\mathcal{F}}^{\xi}\Big[ \mathcal{E}_{(1+\frac{\delta}{3}+(k-1)\delta_1)r,(1+\delta)r}\Big]+\frac{\mathsf{c}}{2} \,.
    \end{equation}
    Assume to the contrary that~\eqref{eq:probability-k-1} fails for all $k$, then by Lemma~\ref{lem:boost-bc} applied to $\delta'=\frac{\delta}{3}+(k-1)\delta_1$, we get
    \begin{align}\label{eq:probability-k-2}
        \phi_{\mathcal{F}}^{\xi'}\Big[ \mathcal{E}_{(1+\frac{\delta}{3}+k\delta_1)r,(1+\delta)r}\Big] &< \phi_{\mathcal{F}}^{\xi}\Big[ \mathcal{E}_{(1+\frac{\delta}{3}+(k-1)\delta_1)r,(1+\delta)r}\Big]+\frac{\mathsf{c}}{2} \nonumber \\
        & \le \phi_{\mathcal{F}}^{\xi'}\Big[ \mathcal{E}_{(1+\frac{\delta}{3}+(k-1)\delta_1)r,(1+\delta)r}\Big]-\frac{\mathsf{c}}{2} \,.
    \end{align}
    Summing up~\eqref{eq:probability-k-2} for $k \in \{1,2,\ldots,N\}$ yields
    \[ \phi_{\mathcal{F}}^{\xi'}\left[ \mathcal{E}_{(1+\frac{2\delta}{3})r,(1+\delta)r}\right]<\phi_{\mathcal{F}}^{\xi}\left[ \mathcal{E}_{(1+\frac{\delta}{3})r,(1+\delta)r}\right]-\frac{\mathsf{c}}{2} N<0 \,, \]
    which is a contradiction and proves~\eqref{eq:probability-k-1} thereby. Consider the event $F_k$ that
    \begin{itemize}
        \item[$(\heartsuit)$] $\eta'$ exists (i.e. $\tau_1'<\infty$), and $\eta' \in \Lambda_{(1+\frac{\delta}{3}+k\delta_1)r,(1+\delta)r}$;
        \item[$(\diamondsuit)$] either $\tau_1=\infty$, or $\eta \in \Lambda_{r,(1+\frac{\delta}{3}+(k-1)\delta_1)r}$.
    \end{itemize}
    By~\eqref{eq:probability-k-1}, for some $1 \le k \le N$ we have
    \begin{equation}\label{eq:probability-F}
        \mathbf{P}[F_k]=\mathbf{P}\Big[ \, \omega' \in \mathcal{E}_{(1+\frac{\delta}{3}+k\delta_1)r,(1+\delta)r} \,, \omega \notin \mathcal{E}_{(1+\frac{\delta}{3}+(k-1)\delta_1)r,(1+\delta)r}\Big] \ge \frac{\mathsf{c}}{2} \,.
    \end{equation}
    On the event $F_k$, we further consider the following event $H$ (for $\omega'$) that
    \begin{itemize}
        \item[$(\clubsuit)$] there exists an $\omega'$-dual open circuit in $\Lambda_{(1+\frac{\delta}{3}+(k-\frac{2}{3})\delta_1)r,(1+\frac{\delta}{3}+(k-\frac{1}{3})\delta_1)r}$;
        \item[$(\spadesuit)$] there exists an $\omega'$-open path connecting $\partial \Lambda_r$ and $\eta$.
    \end{itemize}
    Note that if $F_k$ and $H$ occur, then $A(r;\delta)$ occurs for $\omega'$ but not for $\omega$. Indeed, $(\heartsuit)$ and $(\clubsuit)$ guarantee that $A(r;\delta)$ occurs for $\omega'$. Since $\mathbf{P}$ is an increasing coupling in $D(\eta)$, the $\omega'$-open path in $(\spadesuit)$ is also an $\omega$-open path, which makes an $\omega$-dual open circuit between $\eta$ and $\partial \Lambda_r$ impossible.
    Remind that under the coupling $\mathbf{P}$ and conditioned on $\eta'$ and $\eta$, the law of $\omega'$ in $D(\eta')$ is exactly $\phi_{D(\eta')}^1$. 
    By standard RSW estimates in Theorem~\ref{thm:rsw},
    \begin{equation}\label{eq:probability-H}
        \mathbf{P}\left[ \, \omega' \in H \ \middle| \ (\omega',\omega) \mbox{ such that } F_k \mbox{ occurs}\right]>c_2 \,,
    \end{equation}
    for some $c_2>0$ only depending on $\delta_1$, $\delta$ and $q$, which in turn depends only on $\delta$ and $q$. Thus,
    \[ \phi_{\mathcal{F}}^{\xi'}[A(r;\delta)]-\phi_{\mathcal{F}}^{\xi}[A(r;\delta)] \ge \mathbf{P}\left[ (\omega',\omega) \in F_k \cap H \right] \ge \frac{\mathsf{c}}{2}c_2>0 \,, \]
    which completes the proof.
\end{proof}

\begin{proof}[Proof of Proposition~\ref{prop:event-A}]
    Let $\mathbb{P}$ be the increasing coupling of $\omega^0$ and $\omega^1$ on $\Lambda_{2r,R}$ as in Theorem~\ref{thm:coupling-fk}, and let $\mathbb{E}$ be the associated expectation. On the event $\{\tau<\infty\}$, the unexplored domain $\mathcal{F}_\tau$ is a $\frac{1}{2}$-well-separated inner flower domain on $\Lambda_{2r}$, and the boundary conditions $\xi \le \xi'$ induced by $\omega^0$ and $\omega^1$ on $\mathcal{F}_\tau$ form a boosting pair.
    On the event $\{\tau=\infty\}$, the boundary conditions $\zeta$ and $\zeta'$ induced by $\omega^0$ and $\omega^1$ on $\partial \Lambda_{2r}$ also satisfy $\zeta \le \zeta'$. In the same way as in the proof of Proposition~\ref{prop:boost-bc} that shows~\eqref{eq:phi_1-phi_0} is non-negative, we get that $\phi_{\Lambda_{2r}}^{\zeta'}[A(r;\delta)]-\phi_{\Lambda_{2r}}^\zeta [A(r;\delta)] \ge 0$ in this case.
    By Proposition~\ref{prop:boost-bc},
    \[ \phi_{\Lambda_R}^1[A(r;\delta)]-\phi_{\Lambda_R}^0[A(r;\delta)] \ge \mathbb{E}\left[\1_{\tau<\infty} \left[\phi_{\mathcal{F}_\tau}^{\xi'}[A(r;\delta)]-\phi_{\mathcal{F}_\tau}^{\xi}[A(r;\delta)]\right] \right] \ge c_1 \cdot \mathbb{P}[\tau<\infty] \ge c \cdot \Delta(r,R) \,, \]
    where we used~\eqref{eq:coupling-fk} in the last step. Since $\phi_{\Lambda_R}^0[A(r;\delta)] \le 1$, this proves the lower bound in~\eqref{eq:event-A}. Together with the upper bound~\eqref{eq:upper-bound}, the claim follows.
\end{proof}
\section{Derivation of the CLE mixing rate}
\label{sec:continuum}

We will review the background of CLE and discuss the scaling limit of FK percolation loop configuration in Section~\ref{subsec:cle-prelim}.
Then in Section~\ref{subsec:proof-iota}, we prove Theorem~\ref{thm:iota-continuum-easy} based on Theorem~\ref{thm:cle-partition} and hence derive the mixing rate exponent for CLE.

\subsection{Preliminaries on CLE}
\label{subsec:cle-prelim}

For $\kappa \in (4,8)$, the nested CLE$_\kappa$ is a random collection $\Gamma$ of non-simple loops, where each loop may touch each other or the boundary without crossing. It was first constructed in~\cite{Sheffield-CLE} using the continuum exploration tree. For $\kappa \in (4,8)$, the SLE$_\kappa(\kappa-6)$ process is a continuous curve~\cite{MS16a} with the target invariant property~\cite{SW05,Sheffield-CLE}. Let $a_1,a_2,\ldots$ be a countable dense set of points in $\mathbb{D}$. We can construct a branching SLE$_\kappa(\kappa-6)$ from 1 with force point at $e^{i0^-}$, targeting all points $a_1,a_2,\ldots$ simultaneously. For example, the SLE$_\kappa(\kappa-6)$ curves targeting $a_i$ and $a_j$ agree until the first time that it disconnects $a_i$ and $a_j$, and then evolve independently in the connected components containing $a_i$ and $a_j$ respectively. Let $\eta^{a_k}$ be the branch that targets $a_k$. For $a \notin (a_k)_{k \ge 1}$, the branch $\eta^a$ that targets $a$ is a radial SLE$_\kappa(\kappa-6)$ and can be a.s. uniquely determined by $(\eta^{a_{k_n}})_{n \ge 1}$ for some subsequence $(a_{k_n})_{n \ge 1}$ converging to $a$.

Without loss of generality, assume that $a_1=0$. The outermost loop $\mathcal{L}_1^0$ in nested CLE$_\kappa$ that surrounds the origin that exists a.s. is constructed as follows:

\begin{enumerate}[(i)]
    \item Let $\eta=\eta^0$ be a radial SLE$_\kappa(\kappa-6)$ in $\mathbb{D}$ from 1 to 0 with force point at $e^{i0^-}$. Let $\sigma_0=0$ and record the subsequent times $\sigma_1<\sigma_2<\cdots$ at which $\eta$ makes a closed loop around the origin in either the clockwise or counterclockwise direction, i.e., $\sigma_n$ is the first time $t>\sigma_{n-1}$ that $\eta[\sigma_{n-1},t]$ separates the origin from $\eta[0,\sigma_{n-1}]$.
    \item Let $\sigma_m$ be the first time that a loop is formed in the counterclockwise direction for some integer $m \ge 1$. Let $z$ be the leftmost intersection point of $\eta[\sigma_{m-1},\sigma_m] \cap \partial (\mathbb{D} \setminus \eta[0,\sigma_{m-1}])$ on the boundary of the connected component of $\mathbb{D} \setminus \eta[0,\sigma_{m-1}]$ that contains the origin; see Figure~\ref{fig:cle}. 
    \item Let $t_0$ be the last time before $\sigma_m$ that $\eta$ visits $z$. The trace $\eta[0,t_0]$ coincides with $\eta^z[0,t_0']$ for some $t_0'$, then $\mathcal{L}_1^0$ is defined to be the loop $\eta^z[t_0',\infty)$.
\end{enumerate}

By continuing $\eta$ beyond $\sigma_m$, we may construct a sequence of nested loops $(\mathcal{L}_j^0)_{j \ge 1}$ that surround the origin. For $k \ge 2$, the corresponding loops $(\mathcal{L}_j^{a_k})_{j \ge 1}$ surrounding $a_k$ can be constructed analogously. The nested CLE$_\kappa$ in $\mathbb{D}$ is then given by $\Gamma=\{\mathcal{L}_j^{a_k}:k \ge 1,\ j \ge 1\}$.

\begin{figure}
     \centering
     \includegraphics[scale=0.7]{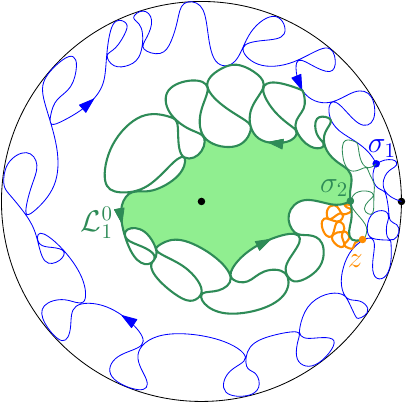}
     \caption{Construction of CLE$_\kappa$ via a branching SLE$_\kappa(\kappa-6)$ process. In this case $m=2$, the outermost loop $\mathcal{L}_1^0$ is the concatenation of the bold green and orange curves. By continuing $\eta$ beyond $\sigma_m$ towards the origin, we can construct a series of nested loops in the green region.}
     \label{fig:cle}
 \end{figure}

For $\kappa \in (8/3,4]$, the non-nested CLE$_\kappa$ can be constructed analogously using an exploration tree of SLE$_\kappa(\kappa-6)$ processes. The definition of these processes requires L\'evy compensation and we refer to~\cite{She07} for details. Another construction is given in~\cite{Sheffield-Werner-CLE} using Brownian loop soup. Specifically, given a simply-connected domain $D \subsetneq \mathbb{C}$, consider a Brownian loop soup on $D$ with intensity $c=(3\kappa-8)(6-\kappa)/2\kappa \in (0,1]$. Then the non-nested CLE$_\kappa$ is the family of disjoint loops consisting of the outer boundaries of outermost loop clusters. The nested CLE$_\kappa$ can be constructed by an iteration procedure.

To discuss the convergence of a collection of loops, we use the conventional setup as in~\cite{CN06,BH19}.
Each non-crossing loop $\eta \subset \mathbb{C}$ is identified as a continuous function $\eta:\partial \mathbb{D} \to \mathbb{C}$ up to time reparametrization, that is, $\eta_1$ and $\eta_2$ are equivalent if there exists a homeomorphism $\psi$ from $\partial \mathbb{D}$ to itself such that $\eta_1=\eta_2 \circ \psi$. The distance between two loops $\eta_1$ and $\eta_2$ is defined as
\[ d(\eta_1,\eta_2)=\inf_\psi \sup_{t \in \partial \mathbb{D}} |\eta_1(t)-\eta_2 \circ \psi(t)|, \]
where the infimum is taken over all homeomorphisms $\psi$ from $\partial \mathbb{D}$ to itself.
For two countable collections $\Gamma_1,\Gamma_2$ of loops, we say $\mathsf{d}(\Gamma_1,\Gamma_2) \le \varepsilon$ if and only if for any $\eta_1 \in \Gamma_1$ with $\mathrm{diam}(\eta_1) \ge \varepsilon$, there exists $\eta_2 \in \Gamma_2$ such that $d(\eta_1,\eta_2) \le \varepsilon$, and vice versa.

For a Jordan domain $D$, let $(D_n)_{n \ge 1}$ be a sequence of discrete domains in $\mathbb{Z}^2$ so that $\frac{1}{n} D_n$ converges to $D$ as $n \to \infty$. For $q \in (0,4]$, we sample $\omega_n$ from a critical FK$_q$ percolation on $D_n$ with \emph{free} boundary conditions and let $\Gamma_n$ be its loop configuration scaled by $1/n$. Let $\kappa=4\pi/\arccos(-\sqrt{q}/2) \in [4,8)$, and let $\Gamma$ be the nested CLE$_\kappa$ in $D$. The following is the conformal invariance conjecture of FK percolation.
\begin{conjecture}\label{conj:fk-to-cle}
    $\Gamma_n$ converges in distribution to $\Gamma$ with respect to the metric $\mathsf{d}$ as $n \to \infty$.
\end{conjecture}
In this setting, the loops of the nested CLE$_\kappa$ will alternatively correspond to inner or outer interfaces of primal clusters (that are respectively the outer and inner interfaces of the dual clusters).
By duality, the same should also be true under wired boundary conditions. So far, Conjecture~\ref{conj:fk-to-cle} is only known for the FK-Ising model ($q=2$ and $\kappa=16/3$), by~\cite{Smi10,KS19,KS16}.

\subsection{Proof of Theorem~\ref{thm:iota-continuum-easy}}\label{subsec:proof-iota}

Fix $\kappa \in (8/3,8)$ throughout this subsection.
Recall from Definition~\ref{def:modulus} that each loop $\eta$ in $\mathbb{D}$ is associated with a unique parameter $r(\eta) \in (0,1)$. For a simply connected domain $D \subsetneq \mathbb{C}$ and $z \in D$, the \emph{conformal radius} of $D$ seen from $z$ is defined as ${\rm CR}(z,D):=|f'(0)|$, where $f$ is a conformal map from $\mathbb{D}$ to $D$ with $f(0)=z$. For a loop $\eta$ in $\mathbb{D}$ surrounding the origin, let $D_\eta$ be the connected component of $\mathbb{D} \setminus \eta$ that contains the origin.

\begin{definition}\label{def:good-domain}
    Fix $\delta_0=\delta_0(\kappa)>0$ sufficiently small.
    For $\varepsilon \in (0,\frac{1}{2})$, an annular domain $\mathcal{A} \subset \mathbb{D}$ which disconnects 0 and $\partial \mathbb{D}$ (i.e. 0 lies in the simply connected component of $\mathbb{D} \setminus \mathcal{A}$) is called an \emph{$\varepsilon$-scale annular domain} if $r(\partial_i),r(\partial_o) \in (\frac{1}{2}\varepsilon,2\varepsilon)$ and $\frac{\mathrm{CR}(0,D_{\partial_i})}{\mathrm{CR}(0,D_{\partial_o})} \ge (1+\delta_0)^{-1}$, where $\partial_i$ and $\partial_o$ are the inner and outer boundaries of $\mathcal{A}$, respectively.
\end{definition}

Here, the requirement on conformal radii is a technical condition that guarantees that the annular domain is ``thin''; the choice of $\delta_0$ will be specified in Lemma~\ref{lem:exp-decay}. The mixing rate of CLE is encapsulated in the following proposition, which extends Theorem~\ref{thm:iota-continuum-easy} to all $\varepsilon$-scale annular domains.

\begin{proposition}\label{prop:iota-continuum}
    Fix $\kappa \in (8/3,8)$. For an $\varepsilon$-scale annular domain $\mathcal{A} \subset \mathbb{D}$, let $\mathsf{A}^{\mathrm{odd}}(\mathcal{A})$ (resp. $\mathsf{A}^{\mathrm{even}}(\mathcal{A})$) be the event that there exists a non-contractible $\CLE_\kappa$ loop in $\mathcal{A}$ with odd (resp. even) nesting level. Then the following holds for any $\varepsilon$-scale annular domain $\mathcal{A}$ (all constants in $\asymp$ depend only on $\kappa$).
    \begin{itemize}
        \item For $\kappa \in (8/3,6)$, as $\varepsilon \to 0$,
        \begin{equation}\label{eq:cle-iota-46}
            \frac{\mathbb{P}[\mathsf{A}^{\mathrm{odd}}(\mathcal{A})]-\mathbb{P}[\mathsf{A}^{\mathrm{even}}(\mathcal{A})]}{\mathbb{P}[\mathsf{A}^{\mathrm{even}}(\mathcal{A})]} \asymp \varepsilon^{\frac{3\kappa}{8}-1} \,.
        \end{equation}
        \item For $\kappa=6$, \ $\mathbb{P}[\mathsf{A}^{\mathrm{odd}}(\mathcal{A})]=\mathbb{P}[\mathsf{A}^{\mathrm{even}}(\mathcal{A})]$ for all $\varepsilon \in (0, \frac12)$.
        \item For $\kappa \in (6,8)$, as $\varepsilon \to 0$,
        \begin{equation}\label{eq:cle-iota-68}
            \frac{\mathbb{P}[\mathsf{A}^{\mathrm{even}}(\mathcal{A})]-\mathbb{P}[\mathsf{A}^{\mathrm{odd}}(\mathcal{A})]}{\mathbb{P}[\mathsf{A}^{\mathrm{odd}}(\mathcal{A})]} \asymp \varepsilon^{\frac{3\kappa}{8}-1} \,.
        \end{equation}
    \end{itemize}
\end{proposition}

We now proceed to the proof of Proposition~\ref{prop:iota-continuum}, based on Theorem~\ref{thm:cle-partition}.
Fix $\kappa \in (8/3,8)$. Let $\Gamma$ be the nested CLE$_\kappa$ in $\mathbb{D}$, and $\Gamma_0 \subset \Gamma$ be the collection of loops surrounding the origin. Let $\Gamma_0^{\mathrm{odd}}$ (resp. $\Gamma_0^{\mathrm{even}}$) be the collection of loops in $\Gamma_0$ with odd (resp. even) nesting levels.
For $k \ge 1$, let $\mathsf{m}_\kappa^{k,\mathrm{odd}}$ (resp. $\mathsf{m}_\kappa^{k,\mathrm{even}}$) denote the law of $k$ different nested loops sampled from the counting measure of non-boundary-touching CLE$_\kappa$ loops from $\Gamma_0^{\mathrm{odd}}$ (resp. $\Gamma_0^{\mathrm{even}}$).
Formally, let $\mathcal{L}_k:=\{(\eta_1,\ldots,\eta_k) :\eta_{j+1} \subset \overline{\eta_j^o}, \ \eta_1 \cap \partial \mathbb{D}=\emptyset, 0 \in \eta_k^o\}$ be the space of $k$-tuple of nested non-boundary-touching loops surrounding the origin, then $\mathsf{m}_\kappa^{k,\mathrm{odd}}$ is supported on $\mathcal{L}_k$ so that for any non-negative measurable function $F$,
\[ \idotsint F(\eta_1,\ldots,\eta_k) \mathsf{m}_\kappa^{k,\mathrm{odd}}(\dd \eta_1,\ldots,\dd \eta_k)=\mathbb{E}\Big[\!\sum_{(\eta_1,\ldots,\eta_k) \in (\Gamma_0^{\mathrm{odd}})^k \cap \mathcal{L}_k} F(\eta_1,\ldots,\eta_k)\Big]. \]
The measure $\mathsf{m}_\kappa^{k,\mathrm{even}}$ is defined analogously. Note that $\mathsf{m}_\kappa^{1,\mathrm{odd}}=\mathsf{m}_\kappa^{\mathrm{odd}}$ and $\mathsf{m}_\kappa^{1,\mathrm{even}}=\mathsf{m}_\kappa^{\mathrm{even}}$.
We can extend these notions to any simply connected domain $D \subsetneq \mathbb{C}$ and any interior point $z \in D$ using conformal map.
In particular, the measures $\mathsf{m}_\kappa^{k,\mathrm{odd}}$ and $\mathsf{m}_\kappa^{k,\mathrm{even}}$ can also be recursively defined via
\begin{align*}
    \mathsf{m}_\kappa^{k,\mathrm{odd}}(\dd \eta_1,\ldots,\dd \eta_k)&=\mathsf{m}_\kappa^{\mathrm{odd}}(\dd \eta_1) \cdot \psi_*\mathsf{m}_\kappa^{k-1,\mathrm{even}}(\dd \eta_2,\ldots,\dd \eta_k), \\
    \mathsf{m}_\kappa^{k,\mathrm{even}}(\dd \eta_1,\ldots,\dd \eta_k)&=\mathsf{m}_\kappa^{\mathrm{even}}(\dd \eta_1) \cdot \psi_*\mathsf{m}_\kappa^{k-1,\mathrm{even}}(\dd \eta_2,\ldots,\dd \eta_k),
\end{align*}where $\psi$ is a conformal map from $\mathbb{D}$ to $D_{\eta_1}$ fixing the origin, and $\psi_* \mathsf{m}(\cdot)=\mathsf{m}(\psi^{-1}(\cdot))$ is the push-forward measure. In view of these, the following lemma is immediate.

\begin{lemma}\label{lem:RN-derivative}
    For any $k \ge 1$ and $(\eta_1,\ldots,\eta_k) \in \mathcal{L}_k$, we have
    \begin{equation*}
        \frac{\mathsf{m}_\kappa^{k,\mathrm{odd}}(\dd \eta_1,\ldots, \dd \eta_k)}{\mathsf{m}_\kappa^{k,\mathrm{even}}(\dd \eta_1,\ldots,\dd \eta_k)}=\frac{\mathsf{m}_\kappa^{\mathrm{odd}}(\dd \eta_1)}{\mathsf{m}_\kappa^{\mathrm{even}}(\dd \eta_1)} \,.
    \end{equation*}
\end{lemma}

Using these notations and the inclusion-exclusion principle, we have
\begin{equation}\label{eq:pie}
\begin{split}
    \mathbb{P}[\mathsf{A}^{\mathrm{odd}}(\mathcal{A})] &= \sum_{k=1}^\infty (-1)^{k-1} \mathsf{m}_\kappa^{k,\mathrm{odd}}[\eta_1,\ldots,\eta_k \subset \mathcal{A}] \,, \\
    \mathbb{P}[\mathsf{A}^{\mathrm{even}}(\mathcal{A})] &= \sum_{k=1}^\infty (-1)^{k-1} \mathsf{m}_\kappa^{k,\mathrm{even}}[\eta_1,\ldots,\eta_k \subset \mathcal{A}] \,.
\end{split}
\end{equation}
Now we prove that if $\delta_0$ in Definition~\ref{def:good-domain} is chosen sufficiently small, the $k \ge 2$ terms on the right-hand sides of~\eqref{eq:pie} are negligible for any $\varepsilon$-scale annular domain $\mathcal{A}$. Then it will be sufficient to compare the first terms, for which Theorem~\ref{thm:cle-partition} may apply.

\begin{lemma}\label{lem:cle-cr}
    For any $\kappa \in (8/3,8)$, we have
    \[ \sup_{D: \CR(0,D) \ge r} \mathsf{m}_\kappa^{\mathrm{even}}[\eta \cap \overline{D}=\emptyset] \to 0 \quad \mbox{as } r \uparrow 1 \,, \]
    where the supremum runs over every simply connected domain $D \subset \mathbb{D}$ that contains the origin.
\end{lemma}

\begin{proof}
    Let $\Gamma$ be the nested $\CLE_\kappa$ in $\mathbb{D}$, and for each $j \ge 1$, let $\gamma_j$ be the $(2j)$-th level loop in $\Gamma$ surrounding the origin which a.s. exists. Let $D_j$ be the connected component of $\mathbb{D} \setminus \gamma_j$ that contains the origin, and $D_0=\mathbb{D}$. By~\cite[Theorem 1]{SSW09},
    \[ X_j=\log \CR(0,D_{j-1})-\log \CR(0,D_j) \,, \quad j=1,2,\ldots, \]
    are \ i.i.d. random variables with continuous density function supported on $(0,\infty)$. Write $t=-\log \CR(0,D)$. Note that $\gamma_j \cap \overline{D}=\emptyset$ necessarily implies $\CR(0,D_j) \ge \CR(0,D)$, hence
    \[ \mathbb{P}[\gamma_j \cap \overline{D}=\emptyset] \le \mathbb{P}[X_1+\cdots+X_j \le t] \le \mathbb{P}[X_1 \le t]^j \,. \]
    Therefore,
    \[ \mathsf{m}_\kappa^{\mathrm{even}}[\eta \cap \overline{D}=\emptyset]=\mathbb{E} \Big[ \sum_{j \ge 1} \1_{\gamma_j \cap \overline{D}=\emptyset} \Big] \le \sum_{j \ge 1} \mathbb{P}[X_1 \le t]^j=\frac{\mathbb{P}[X_1 \le t]}{1-\mathbb{P}[X_1 \le t]} \,, \]
    which tends to zero as $t \downarrow 0$.
\end{proof}

\begin{lemma}\label{lem:exp-decay}
    For any $\kappa \in (8/3,8)$, there exists $\delta_0=\delta_0(\kappa)>0$ such that for any $\varepsilon$-scale annular domain $\mathcal{A} \subset \mathbb{D}$ and $k \ge 1$, we have
    \begin{align}
        \mathsf{m}_\kappa^{k,\mathrm{odd}}[\eta_1,\ldots,\eta_k \subset \mathcal{A}] &\le 100^{-(k-1)} \mathsf{m}_\kappa^{\mathrm{odd}}[\eta_1 \subset \mathcal{A}] \label{eq:exp-decay-odd} \\
        \mbox{and }\quad \mathsf{m}_\kappa^{k,\mathrm{odd}}[\eta_1,\ldots,\eta_k \subset \mathcal{A}] &\le 100^{-(k-1)} \mathsf{m}_\kappa^{\mathrm{even}}[\eta_1 \subset \mathcal{A}] \label{eq:exp-decay-even} \,.
    \end{align}
\end{lemma}

\begin{proof}
    By Lemma~\ref{lem:cle-cr}, we may choose $\delta_0>0$ so that $\sup_{D: \CR(0,D) \ge (1+\delta_0)^{-1}} \mathsf{m}_\kappa^{\mathrm{even}}[\eta \cap \overline{D}=\emptyset]<100^{-1}$. On the event $\{ \eta_1 \subset \mathcal{A} \}$, we have $\frac{\CR(0,D_{\partial_i})}{\CR(0,D_{\eta_1})} \ge (1+\delta_0)^{-1}$, hence
    \begin{align*}
        \mathsf{m}_\kappa^{2,\mathrm{odd}}[\eta_1,\eta_2 \subset \mathcal{A}] &=\mathbb{E} \bigg[ \sum_{\eta_1 \in \Gamma_0^{\mathrm{odd}}} \1_{\eta_1 \subset \mathcal{A}} \ \mathbb{E} \Big[ \sum_{\eta_2 \in \Gamma_0^{\mathrm{odd}}, \eta_2 \subset D_{\eta_1}} \1_{\eta_2 \cap D_{\partial_i}=\emptyset} \ \Big| \ \eta_1 \Big] \bigg] \\
        &\le \mathbb{E} \Big[ 100^{-1} \sum_{\eta_1 \in \Gamma_0^{\mathrm{odd}}} \1_{\eta_1 \subset \mathcal{A}} \Big]=100^{-1} \mathsf{m}_\kappa^{\mathrm{odd}}[\eta_1 \subset \mathcal{A}] \,,
    \end{align*}
    and~\eqref{eq:exp-decay-odd} follows from simple induction. ~\eqref{eq:exp-decay-even} can be proved analogously.
\end{proof}

\begin{proof}[Proof of Proposition~\ref{prop:iota-continuum} assuming Theorem~\ref{thm:cle-partition}]
    Fix an $\varepsilon$-scale annular domain $\mathcal{A} \subset \mathbb{D}$.
    
    For $\kappa=6$, it follows from Theorem~\ref{thm:cle-partition} and Lemma~\ref{lem:RN-derivative} that $\mathsf{m}_\kappa^{k,\mathrm{odd}}[\eta_1,\ldots,\eta_k \subset \mathcal{A}]=\mathsf{m}_\kappa^{k,\mathrm{even}}[\eta_1,\ldots,\eta_k \subset \mathcal{A}]$ for each $k$, hence $\mathbb{P}[\mathsf{A}^{\mathrm{odd}}(\mathcal{A})]=\mathbb{P}[\mathsf{A}^{\mathrm{even}}(\mathcal{A})]$.
    
    For $\kappa \in (8/3,6)$, on the event $\{ \eta \subset \mathcal{A} \}$, $r=r(\eta)$ satisfies $\frac{1}{2} \varepsilon \le r \le 2\varepsilon$. By Theorem~\ref{thm:cle-partition} and Lemma~\ref{lem:RN-derivative}, for every $k \ge 1$ and $\varepsilon>0$ sufficiently small,
    \begin{equation*}
        0<\mathtt{c}_1 \varepsilon^{\frac{3\kappa}{8}-1} \le \frac{\mathsf{m}_\kappa^{k,\mathrm{odd}}[\eta_1,\ldots,\eta_k \subset \mathcal{A}]}{\mathsf{m}_\kappa^{k,\mathrm{even}}[\eta_1,\ldots,\eta_k \subset \mathcal{A}]}-1 \le \mathtt{c}_2 \varepsilon^{\frac{3\kappa}{8}-1},
    \end{equation*}
    where $\mathtt{c}_1=\frac{1}{2} \cos(\tfrac{\kappa-4}{4}\pi)$ and $\mathtt{c}_2=20\cos(\tfrac{\kappa-4}{4}\pi)$. By~\eqref{eq:pie} and Lemma~\ref{lem:exp-decay},
    \begin{align*}
        \mathbb{P}[\mathsf{A}^{\mathrm{odd}}(\mathcal{A})]-\mathbb{P}[\mathsf{A}^{\mathrm{even}}(\mathcal{A})] &\ge \mathsf{m}_\kappa^{\mathrm{odd}}[\eta_1 \subset \mathcal{A}]-\mathsf{m}_\kappa^{\mathrm{even}}[\eta_1 \subset \mathcal{A}] \\
        &-\sum_{k \ge 2, {\rm even}} \left( \mathsf{m}_\kappa^{k,\mathrm{odd}}[\eta_1,\ldots,\eta_k \subset \mathcal{A}]-\mathsf{m}_\kappa^{k,\mathrm{even}}[\eta_1,\ldots,\eta_k \subset \mathcal{A}] \right) \\
        &\ge \mathtt{c}_1 \varepsilon^{\frac{3\kappa}{8}-1} \mathsf{m}_\kappa^{\mathrm{even}}[\eta_1 \subset \mathcal{A}]-\mathtt{c}_2 \varepsilon^{\frac{3\kappa}{8}-1} \sum_{k \ge 2, {\rm even}} \mathsf{m}_\kappa^{k,\mathrm{even}}[\eta_1,\ldots,\eta_k \subset \mathcal{A}] \\
        &\ge \mathtt{c}_1 \varepsilon^{\frac{3\kappa}{8}-1} \mathsf{m}_\kappa^{\mathrm{even}}[\eta_1 \subset \mathcal{A}]-\mathtt{c}_2 \varepsilon^{\frac{3\kappa}{8}-1} \sum_{k \ge 2, {\rm even}} 100^{-(k-1)} \mathsf{m}_\kappa^{\mathrm{even}}[\eta_1 \subset \mathcal{A}] \\
        &\ge \left( \mathtt{c}_1-\mathtt{c}_2 \frac{1/100}{1-(1/100)^2} \right) \varepsilon^{\frac{3\kappa}{8}-1} \mathsf{m}_\kappa^{\mathrm{even}}[\eta_1 \subset \mathcal{A}]>\frac{\mathtt{c}_1}{2} \varepsilon^{\frac{3\kappa}{8}-1} \mathsf{m}_\kappa^{\mathrm{even}}[\eta_1 \subset \mathcal{A}].
    \end{align*}
    Combining $\mathbb{P}[\mathsf{A}^{\mathrm{even}}(\mathcal{A})] \le \mathsf{m}_\kappa^{\mathrm{even}}[\eta_1 \subset \mathcal{A}]$, this proves the lower bound of~\eqref{eq:cle-iota-46}. Similarly, we have
    \[ \mathbb{P}[\mathsf{A}^{\mathrm{odd}}(\mathcal{A})]-\mathbb{P}[\mathsf{A}^{\mathrm{even}}(\mathcal{A})] \le \frac{\mathtt{c}_2}{1-(1/100)^2} \varepsilon^{\frac{3\kappa}{8}-1} \mathsf{m}_\kappa^{\mathrm{even}}[\eta_1 \subset \mathcal{A}]<2 \mathtt{c}_2 \varepsilon^{\frac{3\kappa}{8}-1} \mathsf{m}_\kappa^{\mathrm{even}}[\eta_1 \subset \mathcal{A}], \]
    and
    \[ \mathbb{P}[\mathsf{A}^{\mathrm{even}}(\mathcal{A})] \ge \mathsf{m}_\kappa^{\mathrm{even}}[\eta_1 \subset \mathcal{A}]-\mathsf{m}_\kappa^{2,\mathrm{even}}[\eta_1,\eta_2 \subset \mathcal{A}] \ge \frac{1}{2} \mathsf{m}_\kappa^{\mathrm{even}}[\eta_1 \subset \mathcal{A}], \]
    which together establish the upper bound of~\eqref{eq:cle-iota-46}. The case for $\kappa \in (6,8)$ can be proved analogously, except that the sign is reversed due to $4\cos(\tfrac{\kappa-4}{4}\pi)<0$ in Theorem~\ref{thm:cle-partition}.
\end{proof}
\section{Proof of Theorem~\ref{thm:cle-partition} via Liouville quantum gravity}
\label{sec:lqg}

In this section, we prove Theorem~\ref{thm:cle-partition} for $\kappa \in (4,8)$ based on the coupling of CLE and LQG. In Section~\ref{subsec:lqg-prelim}, we provide the necessary background on Liouville quantum gravity surfaces and present some quantum surfaces that will be used in this paper. In Section~\ref{subsec:welding} we prove the conformal welding results for a CLE loop from $\mathrm{m}_\kappa^{\mathrm{odd}}$ or $\mathrm{m}_\kappa^{\mathrm{even}}$. In Sections~\ref{subsec:symmetry} and~\ref{subsec:pfn}, we derive the laws of the quantum annuli that appear in conformal welding, and conclude the proof.

Throughout this section, we fix $\kappa \in (4,8)$ and $\gamma=\frac{4}{\sqrt{\kappa}} \in (\sqrt{2},2)$. In this section, we work with non-probability measures and extend the terminology of ordinary probability to this setting. For a finite or $\sigma$-finite measure space $(\Omega, \mathcal{F}, M)$, we say $X$ is a random variable if $X$ is an $\mathcal{F}$-measurable function with its \emph{law} defined via the push-forward measure $M_X=X_*M$. In this case we say $X$ is \emph{sampled} from $M_X$ and write $M_X[f]:=\int f(x)M_X(dx)$. \emph{Weighting} the law of $X$ by $f(X)$ refers to working with the measure $\dd\widetilde{M}_X$ with Radon-Nikodym derivative $\frac{\dd\widetilde{M}_X}{\dd M_X} = f$. \emph{Conditioning} on some event $E\in\mathcal{F}$ (with $0<M[E]<\infty$) refers to the probability measure $\frac{M[E\cap \cdot]}{M[E]} $  on the measurable space $(E, \mathcal{F}_E)$ with $\mathcal{F}_E = \{A\cap E: A\in\mathcal{F}\}$.

\subsection{Liouville fields and quantum surfaces}
\label{subsec:lqg-prelim}

First, we review the two-dimensional Gaussian free field (GFF) on simply and doubly connected domains. We direct the reader to~\cite{She07,WP21} for detailed background. Let $D \subset \mathbb{C}$ be a planar domain that is conformally equivalent to a disk or an annulus, and let $\rho(\dd x)$ be a compactly supported probability measure such that $\iint_{D^2} -\log|x-y| \rho(\dd x) \rho(\dd y)<\infty$. Let $H(D;\rho)$ be the Hilbert closure of the space $\{f \in C^\infty(D): \int_D f(x)\rho(\dd x)=0\}$ under the Dirichlet inner product $\langle f,g \rangle_\nabla=(2\pi)^{-1} \int_D \nabla f \cdot \nabla g$. Let $(f_n)_{n \ge 1}$ be an orthogonal basis of $H(D)$, and $(\alpha_n)_{n \ge 1}$ be a sequence of i.i.d. standard Gaussian random variables. The random series $\sum_{n \ge 1} \alpha_n f_n$ converges almost surely as a random generalized function, and the limit $h$ is called a \emph{free boundary Gaussian free field} on $D$ with average zero over $\rho(\dd x)$.

We now introduce the Liouville fields on the annulus using the Gaussian free field.
For convenience, we focus on the horizontal cylinder: For $\tau>0$, let $\mathcal{C}_\tau=[0,\tau] \times [0,1]/\!\!\sim$ where the equivalence relation $\sim$ identifies $[0,\tau] \times \{0\}$ with $[0,\tau] \times \{1\}$; then the modulus of $\mathcal{C}_\tau$ is $\tau$. For future reference, let $\partial_1 \mathcal{C}_\tau=\{0\} \times [0,1]/\!\!\sim$ and $\partial_2 \mathcal{C}_\tau=\{\tau\} \times [0,1]/\!\!\sim$ be its two boundary components. Suppose that $\rho$ is a probability measure on $\mathcal{C}_\tau$ defined as before, and let $P_{\tau,\rho}$ be the law of the free boundary Gaussian free field on $\mathcal{C}_\tau$ with average zero over $\rho(\dd x)$.

\begin{definition}\label{def:lf-annulus}
    For $\tau>0$, let $(h,\mathbf{c})$ be sampled from $P_{\tau,\rho} \times \dd c$ and $\phi=h+c$. We say $\phi$ is a Liouville field on $\mathcal{C}_\tau$, and write $\LF_\tau$ for its law.
\end{definition}

The measure $\LF_\tau$ turns out to be independent of the choice of $\rho$, thanks to the translation invariance of the Lebesgue measure $\dd c$.

For $\gamma \in (0,2)$, let $Q=\frac{\gamma}{2}+\frac{2}{\gamma}$. We now introduce $\gamma$-Liouville quantum gravity (LQG) surfaces, or simply \emph{quantum surfaces}. Consider the space of pairs $(D,h)$ where $D \subseteq \mathbb{C}$ is a planar domain and $h$ is a distribution on $D$, we say $(D,h) \sim_\gamma (\widetilde D,\widetilde h)$ if and only if there is a conformal map $g:D \to \widetilde D$ such that
\begin{equation}\label{eq:equiv-rel}
    \widetilde h=g \bullet_\gamma h := h \circ g^{-1}+Q \log |(g^{-1})'| \,.
\end{equation}
A quantum surface $S$ is then an equivalence class of pairs $(D,h)$ under the equivalence relation $\sim_\gamma$, and a particular choice of such $(D,h)$ is called an embedding of $S$. This notion naturally extends to quantum surfaces decorated with marked points or curves. For example, $(D,h,x,\eta)/\!\!\sim_\gamma$ is a quantum surface with one marked point $x \in \overline{D}$ and one curve $\eta$ on $\overline{D}$, where $\sim_\gamma$ also requires that the marked point and curve are preserved under the conformal map $g$ satisfying~\eqref{eq:equiv-rel}.

For a $\gamma$-quantum surface $(D,h)/\!\!\sim_\gamma$, its \emph{quantum area measure} is defined by $\mu_h=\lim_{\varepsilon \to 0} \e^{\gamma^2/2} e^{\gamma h_\varepsilon(z)} \dd^2 z$, where $\dd^2 z$ is the Lebesgue measure and $h_\varepsilon(z)$ is the average of $h(z)$ over the circle $\{w:|w-z|=\varepsilon\}$.
If $\partial D$ contains a straight line segment $L$ (in particular when $D$ is the upper half plane), we may also define the \emph{quantum boundary length measure} by $\nu_h=\lim_{\varepsilon \to 0} \varepsilon^{\gamma^2/4} e^{\frac{\gamma}{2} h_\varepsilon(x)} \dd x$, where for $x \in L$, $h_\e(x)$ is the average of $h(x)$ over the semi-circle $\{w \in D:|w-x|=\varepsilon\}$. The measures $\mu_h$ and $\nu_h$ are well-defined when $h$ is a variant of the GFF~\cite{DS11,SW16} considered in this paper. For example, for the Liouville field $\phi$ on $\mathcal{C}_\tau$, we can define $\nu_\phi=\lim_{\varepsilon \to 0} \varepsilon^{\gamma^2/4} e^{\frac{\gamma}{2} \phi_\varepsilon(x)} \dd x$ on its boundaries $\partial_1 \mathcal{C}_\tau$ and $\partial_2 \mathcal{C}_\tau$. The definitions of quantum area and length measures are consistent with $\sim_\gamma$, so $\mu_h$ and $\nu_h$ can be extended to other domains using conformal maps. 

A beaded domain is a closed set such that each component of its interior together with its prime-end boundary is homeomorphic to the closed disk. Consider pairs $(D,h)$ where $D$ is a beaded domain and $h$ is a distribution on $D$.
We extend the equivalence relation $\sim_\gamma$ so that $g$ is allowed to be any homeomorphism from $D$ to $\wt D$ that is conformal on each component of the interior of $D$. A \emph{beaded quantum surface} $S$ is an equivalence class of pairs $(D,h)$ under the extended equivalence relation $\sim_\gamma$, and a particular choice of such $(D,h)$ is called an embedding of $S$. Beaded quantum surfaces decorated with marked points or curves can be defined analogously.

Now we present some typical quantum surfaces that will be used in this paper. We start with the (two-pointed) thick quantum disk introduced in~\cite[Section 4.5]{DMS21}. Let $h$ be a free boundary GFF on the upper half plane $\mathbb{H}$, then it admits a unique decomposition $h=h^1+h^2$, where $h^1$ (resp. $h^2$) is constant (resp. has mean zero) on $\{z \in \mathbb{H}:|z|=r\}$ for each $r>0$, and they are independent.

\begin{definition}[Thick quantum disks]
\label{def:thick-qd}
    For $W \ge \frac{\gamma^2}{2}$, let $\beta=\gamma+\frac{2-W}{\gamma} \le Q$. Let $(B_t)_{t \ge 0}$ be a standard Brownian motion conditioned on $B_{2t}-(Q-\beta)t<0$ for all $t>0$, and $(\widetilde B_t)_{t \ge 0}$ be its independent copy. Let
    \begin{equation*}
        Y_t=
        \begin{cases}
            B_{2t}+\beta t & \mbox{ for }\,t \ge 0, \\
            \widetilde B_{2t}+(2Q-\beta)t & \mbox{ for }\,t<0 \,.
        \end{cases}
    \end{equation*}
    and let $h_1(z)=Y_{-\log |z|}$ for each $z \in \mathbb{H}$. Let $h_2$ be a random generalized function with the same law as $h^2$ defined above.
    Independently sample $\mathbf{c}$ from the measure $\frac{\gamma}{2} e^{(\beta-Q)c} \dd c$, and let $\psi(z)=h_1(z)+h_2(z)+\mathbf{c}$. The infinite measure $\Md_2(W)$ is defined as the law of $(\mathbb{H},\psi,0,\infty)/\!\!\sim_\gamma$, and a sample from $\Md_2(W)$ is called a \emph{(two-pointed) quantum disk of weight $W$}.
    The left and right boundary lengths of $(\mathbb{H},\psi,0,\infty)$ are referred to as $\nu_\psi((-\infty,0))$ and $\nu_\psi((0,\infty))$, respectively.
\end{definition}

The thick quantum disk of weight 2 has the special property that its two marked points are typical with respect to the quantum boundary length measure~\cite[Proposition A.8]{DMS21}.
By forgetting the marked points (which induces an unweighting by the boundary length), we can define a quantum disk $\QD$ with no marked points.
By sampling more interior (resp. boundary) marked points according to the quantum area (resp. length) measure, we may define general quantum disks $\QD_{m,n}$ with $m$ interior and $n$ boundary quantum typical points. For our purpose, we only introduce $\QD_{1,0}$ with a single interior point. For a finite measure $M$, let $M^{\#}=|M|^{-1} M$ be the probability measure proportional to $M$.

\begin{definition}\label{def:QD-mn}
    Let $\QD$ be the law of $(\mathbb{H},\phi)/\!\!\sim_\gamma$ where $(\mathbb{H},\phi,0,\infty)/\!\!\sim_\gamma$ is sampled from the weighted measure $\nu_\phi(\partial \mathbb{H})^{-2} \Md_2(2)$.
    Let $\QD_{1,0}$ be the law of $(\mathbb{H},\phi,z)/\!\!\sim_\gamma$ where $(\mathbb{H},\phi,0,\infty)/\!\!\sim_\gamma$ is sampled from the weighted measure $\nu_\phi(\partial \mathbb{H})^{-2} \mu_\phi(\partial \mathbb{H}) \Md_2(2)$ and $z$ is independently sampled from $\mu_\phi^{\#}$.
\end{definition}

Next, we recall the pinched (thin) quantum annulus $\widetilde{\mathrm{QA}}(W)$ introduced in~\cite{LSYZ24,SXZ24}.

\begin{definition}\label{def:pinched-QA}
    For $W\in(0,\frac{\gamma^2}{2})$, define the measure $\widetilde{\mathrm{QA}}(W)$ on beaded surfaces as follows. First sample $T \sim (1-\frac{2W}{\gamma^2})^{-2} t^{-1}\1_{t>0} \dd t$, and let $S_T:=\{z \in \mathbb{C}:|z|=\frac{T}{2\pi} \}$ be the circle with perimeter $T$. Then sample a Poisson point process $\{(u,\mathcal{D}_u)\}$ from the measure $\mathrm{Leb}_{S_T} \times \Md_2(\gamma^2-W)$, and concatenate the $\mathcal{D}_u$ to get a cyclic chain according to the ordering induced by $u$.
    The outer (resp. inner) boundary of $\widetilde{\mathrm{QA}}(W)$ is referred to as the concatenation of the left (resp. right) boundaries of all $\mathcal{D}_u$'s.
    A sample from $\widetilde{\mathrm{QA}}(W)$ is called a \emph{pinched quantum annulus of weight $W$}, and $T$ is called its quantum cut point measure.
\end{definition}

Now we review the forest lines and generalized quantum surfaces considered in~\cite{DMS21,MSW21-nonsimple,AHSY23}. For $\gamma \in (\sqrt{2},2)$, the forested lines were first constructed in~\cite{DMS21} based on the \emph{$\frac{4}{\gamma^2}$-stable looptrees} studied in~\cite{CK13looptree}.
Let $(X_t)_{t \ge 0}$ be a stable L\'evy process of index $\frac{4}{\gamma^2} \in (1,2)$ starting from 0 with only upward jumps. The graph of $X$ induces a tree of topological disks, where each disk corresponds to an upward jump of $X_t$, as illustrated in Figure~\ref{fig:forestline-def}. For each upward jump of size $\ell$, we replace the corresponding topological disk with an independent sample of a quantum disk of boundary length $\ell$ from the probability measure $\QD(\ell)^\#$, yielding a beaded surface.
The \emph{root} of this surface refers to the unique point corresponding to $(0,0)$ on the graph of $X$.  
The \emph{forested boundary arc} is the closure of the collection of the points on the boundaries of the quantum disks, and the \emph{line boundary arc} consists of the points corresponding to the running infimum of $(X_t)_{t \ge 0}$. Note that the quantum disks lie on the same side of the line boundary arc as $X$ only has upward jumps.

\begin{figure}
    \centering
    \begin{tabular}{cc} 
		\includegraphics[scale=0.4]{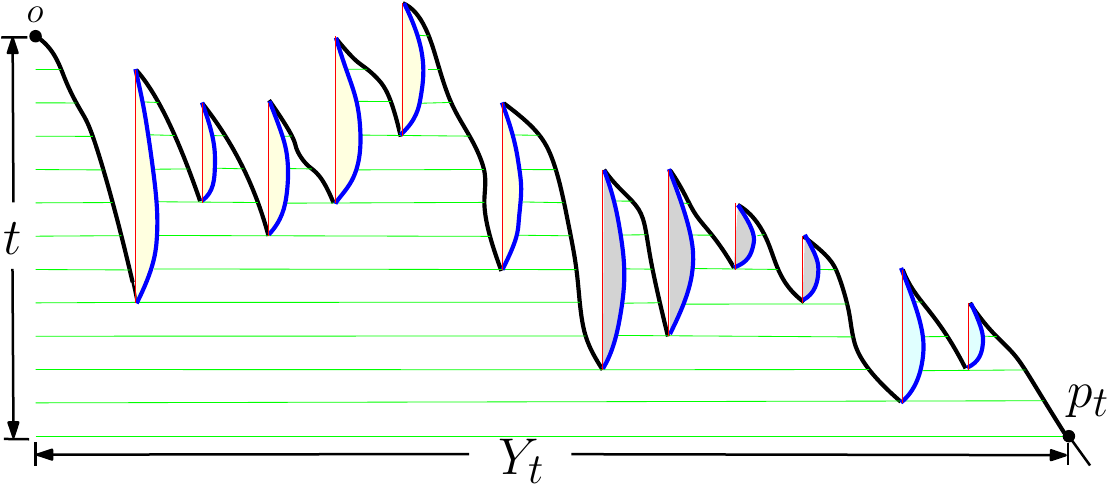}
		&
	   \includegraphics[scale=0.55]{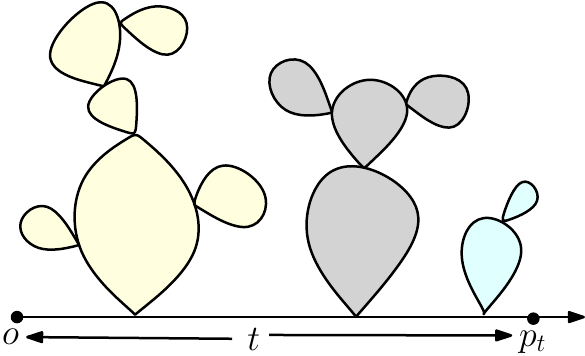}
	\end{tabular}
    \caption{\textbf{Left:} The graph of the $\frac{4}{\gamma^2}$-stable L\'evy process $(X_t)_{t \ge 0}$ with only upward jumps. We draw a blue curve on the right of a red vertical line which corresponds to a jump, and identify the points that are on the same green horizontal line below the graph. \textbf{Right:} The $\frac{4}{\gamma^2}$-stable looptree of disks constructed from $(X_t)_{t \ge 0}$. The points on the red vertical line are collapsed to a single point. By replacing each topological disk with a quantum disk $\QD$ conditioned on having the same quantum boundary length as the size of the jump, we obtain a forested line. The quantum length of the line segment between the root $o$ and the point $p_t$ is $t$, and the generalized quantum length of the forested boundary between $o$ and $p_t$ is $Y_t = \inf\{s>0:X_s \le -t\}$. }
    \label{fig:forestline-def}
\end{figure}

\begin{definition}[Forested line]\label{def:forested-line}
For $\gamma \in (\sqrt{2},2)$, let $(X_t)_{t \ge 0}$ be a stable L\'evy process of index $\frac{4}{\gamma^2} \in (1,2)$ with only upward jumps, and $X_0=0$. For $t>0$, let $Y_t=\inf\{s>0:X_s\le -t\}$, and fix the multiplicative constant of $X$ such that $\mathbb{E}[e^{-Y_1}] = e^{-1}$. The beaded surface as described above is called a forest line.

The line boundary arc and forested boundary arc are parametrized by quantum length and \emph{generalized quantum length} (i.e., the length of the corresponding interval of $(X_t)_{t \ge 0}$), respectively. For example, for a point $p_t$ on the line boundary arc with quantum length $t$ to the root $o$, the forested boundary arc between $o$ and $p_t$ has generalized quantum length $Y_t$.
\end{definition}

As discussed in~\cite{AHSY23}, we can perform a truncation operation on forested lines. Given a forested line $\mathcal{L}$ induced by $X$, for each $t>0$, the truncation of $\mathcal{L}$ at quantum length $t$ refers to the beaded surface $\mathcal{L}_t$ generated by $(X_s)_{0 \le s \le Y_t}$. In terms of Definition~\ref{def:forested-line}, $\mathcal{L}_t$ is the union of the line boundary arc and the quantum disks on the forested boundary arc between the root $o$ and $p_t$. The generalized quantum length of the forested boundary arc of $\mathcal{L}_t$ is $Y_t$.

\begin{definition}[Forested line segment]\label{def:line-segment}
    For $\gamma\in(\sqrt{2},2)$, define $\mathcal{M}_2^\mathrm{f.l.}$ as the law of the beaded surface obtained by first sampling $\mathbf{t}\sim \mathrm{Leb}_{\mathbb{R}_+}$ and truncating an independent forested line at quantum length $\mathbf{t}$. 
\end{definition}

Now we revisit the definition of generalized quantum surfaces in~\cite{AHSY23}. Let $(D,\phi,z_1,\ldots,z_n)$ be an embedding of a (possibly beaded) quantum surface $\mathcal{S}$, with $z_1,\ldots,z_n \in \partial D$ ordered clockwise.
We independently sample $n$ truncated forested lines with quantum lengths that match those of the boundary segments $[z_1,z_2],\ldots,[z_n,z_1]$, and glue them to $\partial D$ correspondingly. We call this procedure \emph{foresting the boundary of $\mathcal{S}$}. The law of the resulting beaded surface is denoted by $\mathcal{S}^f$.

\begin{definition}\label{def:forested-surface}
   The beaded quantum surface $\mathcal{S}^f$ is called a (finite volume) \emph{generalized quantum surface}. We say $\mathcal{S}^f$ is the forested version of $\mathcal{S}$, and $\mathcal{S}$ \emph{is the spine of} $\mathcal{S}^f$. 
\end{definition}

For quantum surfaces without boundary marked points, the foresting procedure still works once we specify a boundary point as root according to the quantum boundary length.
This can be made rigorous in terms of the forested circles introduced in~\cite{LSYZ24,SXZ24}.

\begin{definition}\label{def:forested-circle}
    Define a measure $\mathcal{M}^{\mathrm{f.r.}}$ as follows. First sample $\mathbf{t} \sim t^{-1} \1_{t>0} \dd t$, and sample a forested line $\mathcal{L}$ truncated so that its quantum length is $\mathbf{t}$.
    Consider a unit circle $\mathcal{C}$ that is assigned with quantum length $\mathbf{t}$, and sample a point $z$ on $\mathcal{C}$ according to the probability measure proportional to the quantum length measure. Then glue the truncated forested line $\mathcal{L}_{\mathbf{t}}$ to $\mathcal{C}$ with $z$ being the root. Let $\mathcal{M}^{\mathrm{f.r.}}$ be the resulting beaded surface. A sample from $\mathcal{M}^{\mathrm{f.r.}}$ is called a \emph{forested circle}, and $\mathbf{t}$ is its quantum length.
\end{definition}

Given a measure $\mathcal{M}$ on quantum surfaces, we can disintegrate it over the quantum lengths of the boundary arcs. For instance, we can disintegrate the pinched quantum annulus measure $\widetilde{\mathrm{QA}}(W)$ as
\begin{equation}\label{eq:disint}
    \widetilde{\mathrm{QA}}(W)=\iint_{\mathbb{R}_+^2} \widetilde{\mathrm{QA}}(W;\ell_1,\ell_2) \dd \ell_1 \dd \ell_2, 
\end{equation}
to obtain a measure $\widetilde{\mathrm{QA}}(W;\ell_1,\ell_2)$ supported on the pinched quantum disks with outer boundary length $\ell_1$ and inner boundary length $\ell_2$.

Similarly, by disintegrating over the value of $Y_t$, we can also define a disintegration of the forested line segments and forested circles. For example, we have
\begin{equation}\label{eq:disint-fl}
    \mathcal{M}^{\mathrm{f.r.}}=\int_{\mathbb{R}_+^2} \mathcal{M}^{\mathrm{f.r.}}(\ell,\ell') \dd \ell \dd \ell'.
\end{equation}
where $\mathcal{M}^{\mathrm{f.r.}}(\ell,\ell')$ is the measure on forest circles with quantum length $\ell$ for the line boundary arc and generalized quantum length $\ell'$ for the forested boundary arc. The law of the forested circles whose forested boundary arc has generalized quantum length $\ell'$ can be defined by $\mathcal{M}^{\mathrm{f.r.}}(\ell')=\int_{\mathbb{R}_+} \mathcal{M}^{\mathrm{f.r.}}(\ell,\ell') \dd \ell$.
We take this opportunity to remark that for the rest of this paper, quantum lengths are denoted by $L,L_i,\ell,\ell_i$, and generalized quantum lengths are denoted by $L',L_i',\ell',\ell_i'$.

Finally, we introduce the concept of \emph{conformal welding}. Let $\mathcal{M}^1$ and $\mathcal{M}^2$ be two measures on quantum surfaces with boundary marked points. For $i=1,2$, fix a boundary arc $e_i$ of a sample from $\mathcal{M}^i$ and define the measures $\{\mathcal{M}^i(\ell)\}_{\ell>0}$ through the disintegration $\mathcal{M}^i=\int_{\mathbb{R}_+} \mathcal{M}^i(\ell) \dd \ell$ over the quantum lengths of $e_i$, as discussed above. For each $\ell>0$, given a pair of quantum surfaces sampled from the product measure $\mathcal{M}^1(\ell) \otimes \mathcal{M}^2(\ell)$, we can conformally weld them together according to the quantum length to obtain a single quantum surface decorated with a curve, the welding interface. The law of the resulting curve-decorated quantum surface is denoted by $\Weld(\mathcal{M}^1(\ell),\mathcal{M}^2(\ell))$, and let
\[ \Weld(\mathcal{M}^1,\mathcal{M}^2)=\int_{\mathbb{R}_+} \Weld(\mathcal{M}^1(\ell),\mathcal{M}^2(\ell)) \dd \ell \]
be the conformal welding of $\mathcal{M}^1$ and $\mathcal{M}^2$ along the boundary arcs $e_1$ and $e_2$.
The conformal welding operation naturally extends to quantum surfaces without boundary marked points, as well as to generalized quantum surfaces.
In particular, foresting the boundary can also be interpreted as conformal welding with forested line segments or forested circles. For example,
\[ \QD_{1,0}^f=\int_{\mathbb{R}_+} \Weld(\QD_{1,0}(\ell),\mathcal{M}^{\mathrm{f.r.}}(\ell,\ell')) \ell \dd \ell \dd \ell', \]
and for a quantum surface measure $\mathcal{M}$ with (possibly pinched) annular topology,
\begin{equation}\label{eq:annulus-forested}
    \mathcal{M}^f=\iint_{\mathbb{R}_+^4} \Weld(\mathcal{M}(\ell_1,\ell_2), \mathcal{M}^{\mathrm{f.r.}}(\ell_1;\ell_1'), \mathcal{M}^{\mathrm{f.r.}}(\ell_2;\ell_2')) \ell_1 \ell_2 \dd\ell_1 \dd\ell_2 \dd\ell_1' \dd\ell_2'. 
\end{equation}

For the rest of this subsection, we gather some formulae about the law of the generalized quantum length of forested lines for future reference. The following is from~\cite[Lemmas 3.2 and 3.3]{AHSY23}; the original lemma is stated for real moments, but the verbatim proof also gives imaginary moments.

\begin{lemma}[L\'evy process moments]\label{lem:levy-moment}
    Let $(Y_t)_{t \ge 0}$ be the process in Definition~\ref{def:forested-line}, then $(Y_t)_{t \ge 0}$ is a stable subordinator with index $\frac{\gamma^2}{4}$. For any $t>0$ and $x \in \mathbb{R}$,
    \[ \mathbb{E}[Y_t^{ix}] = \mathbb{E}[Y_1^{ix}] t^{\frac{4}{\gamma^2}ix} = \frac{4}{\gamma^2} \frac{\Gamma(-i\frac{4}{\gamma^2}x)}{\Gamma(-ix)} t^{\frac{4}{\gamma^2}ix}. \]
\end{lemma}

\begin{lemma}\label{lem:remove-fl}
    Suppose $\mathcal{M}$ is the law of a quantum surface with annular topology, and $\mathcal{M}^f$ is its forested version. Let $L_1$ and $L_2$ (resp. $L_1'$ and $L_2'$) be the outer and inner boundary lengths of a sample from $\mathcal{M}$ (resp. $\mathcal{M}^f$). Then for any $t>0$ and $x \in \mathbb{R}$,
    \begin{equation}\label{eq:add-fl}
        \mathcal{M}^f[L_1' e^{-tL_1'} (L_2')^{ix}]=\frac{\Gamma(-\frac{4}{\gamma^2}ix)}{\Gamma(-ix)} t^{\frac{\gamma^2}{4}-1} \mathcal{M} \Big[ L_1 e^{-L_1 t^{\gamma^2/4}} L_2^{\frac{4}{\gamma^2}ix} \Big],
    \end{equation}
    and
    \begin{equation}\label{eq:remove-fl}
        \mathcal{M} \big[ L_1 e^{-t L_1} L_2^{ix} \big]=\frac{\Gamma(-\frac{\gamma^2}{4} ix)}{\Gamma(-ix)} t^{\frac{4}{\gamma^2}-1} \mathcal{M}^f \Big[ L_1' e^{-L_1' t^{4/\gamma^2}} (L_2')^{i\frac{\gamma^2}{4}x} \Big].
    \end{equation}
\end{lemma}

\begin{proof}
    Let $(Y_t)_{t \ge 0}$ be as in Definition~\ref{def:forested-line}. By our normalization, we have $\mathbb{E}[e^{-\lambda Y_t}]=e^{-t \lambda^{\gamma^2/4}}$ for $\lambda>0$ and $t>0$, and thus $\E[Y_te^{-\lambda Y_t}]=\frac{\gamma^2}{4}t\lambda^{\frac{\gamma^2}{4}-1}e^{-t \lambda^{\gamma^2/4}}$. Let $(\overline Y_t)_{t \ge 0}$ be an independent copy of $(Y_t)_{t \ge 0}$. By~\eqref{eq:annulus-forested} and Definition~\ref{def:forested-circle},
    \begin{align*}
        &\quad \mathcal{M}^f[L_1' e^{-tL_1'} (L_2')^{ix}] = \mathcal{M} \big[ \mathbb{E}[Y_{L_1} e^{-tY_{L_1}}] \cdot \mathbb{E}[ (\overline Y_{L_2})^{ix}] \big] \\
        &=\mathcal{M} \Big[ \frac{\gamma^2}{4} t^{\frac{\gamma^2}{4}-1} L_1 e^{-L_1 t^{\gamma^2/4}} \cdot \frac{4}{\gamma^2} \frac{\Gamma(-\frac{4}{\gamma^2} ix)}{\Gamma(-ix)} L_2^{\frac{4}{\gamma^2}ix} \Big] = \frac{\Gamma(-\frac{4}{\gamma^2}ix)}{\Gamma(-ix)} t^{\frac{\gamma^2}{4}-1} \mathcal{M} \Big[ L_1 e^{-L_1 t^{\gamma^2/4}} L_2^{\frac{4}{\gamma^2}ix} \Big],
    \end{align*}
    where the second equality follows from Lemma~\ref{lem:levy-moment}. This gives~\eqref{eq:add-fl}. By a change of variable $(t^{\frac{\gamma^2}{4}},\frac{4}{\gamma^2}x)=(t',x')$, we obtain~\eqref{eq:remove-fl}. 
\end{proof}

Let $\mathcal{M}^{\mathrm{circ}}$ be the quantum surface obtained by conformally welding along the forested boundaries of two forested circles from $\mathcal{M}^{\mathrm{f.r.}}$ conditioned on having
the same generalized quantum length, that is,
\[ \mathcal{M}^{\mathrm{circ}}=\int_{\mathbb{R}_+} \ell' \Weld(\mathcal{M}^{\mathrm{f.r.}}(\ell'),\mathcal{M}^{\mathrm{f.r.}}(\ell')) \dd \ell'. \]
Let $\mathcal{M}^{\mathrm{circ}}(\ell_1,\ell_2)$ be the disintegration of $\mathcal{M}^{\mathrm{circ}}$ over the quantum lengths of the outer and inner boundary arcs, in the same sense as~\eqref{eq:disint}.

\begin{proposition}\label{prop:forest-circle-weld}
    For any $\ell_1>0$ and $x \in \mathbb{R}$,
    \[ \int_{\mathbb{R}_+} |\mathcal{M}^{\mathrm{circ}}(\ell_1,\ell_2)| \ell_2^{ix} \dd \ell_2=\frac{\sinh(\frac{\gamma^2}{4}\pi x)}{\sinh(\pi x)} \ell_1^{ix-1} \,. \]
\end{proposition}

\begin{proof}
    Let $(Y_t)_{t \ge 0}$ be as in Definition~\ref{def:forested-line}, then $\int_a^b |\mathcal{M}_2^{\mathrm{f.c.}}(\ell_2,\ell')| \dd \ell'=\ell_2^{-1} \mathbb{P}[Y_{\ell_2} \in [a,b]]$ for any $0<a<b$. We first prove that for any $\ell'>0$,
    \begin{equation}\label{eq:forest-circle-law}
        \int_{\mathbb{R}_+} |\mathcal{M}^{\mathrm{f.r.}}(\ell_2,\ell')| \ell_2^{ix} \dd \ell_2=\frac{\Gamma(ix)}{\Gamma(\frac{\gamma^2}{4}ix)} (\ell')^{\frac{\gamma^2}{4}ix-1} \,.
    \end{equation}
    Indeed, for $0<a<b$,
    \begin{align*}
        &\quad \int_a^b \int_{\mathbb{R}_+} |\mathcal{M}^{\mathrm{f.r.}}(\ell_2,\ell')| \ell_2^{ix} \dd \ell_2 \dd \ell'=\int_{\mathbb{R}_+} \ell_2^{ix-1} \mathbb{P}[Y_{\ell_2} \in [a,b]] \dd \ell_2 \\
        &=\int_{\mathbb{R}_+} \ell_2^{ix-1} \int \1_{\ell_2^{4/\gamma^2} Y_1 \in [a,b]} \dd \mathbb{P} \dd \ell_2=\frac{\gamma^2}{4} \int \int_a^b s^{\frac{\gamma^2}{4}ix-1} Y_1^{-\frac{\gamma^2}{4}ix} \dd \ell_2 \dd \mathbb{P} \\
        &=\frac{\gamma^2}{4} \mathbb{E}[Y_1^{-\frac{\gamma^2}{4}ix}] \int_a^b s^{\frac{\gamma^2}{4}ix-1} \dd s \,,
    \end{align*}
    where we applied Fubini's theorem and a change of variable $s=\ell_2^{4/\gamma^2} Y_1$. Then~\eqref{eq:forest-circle-law} follows by Lemma~\ref{lem:levy-moment}. By definition, $|\mathcal{M}^{\mathrm{circ}}(\ell_1,\ell_2)|=\int_{\mathbb{R}_+} \ell' |\mathcal{M}^{\mathrm{f.r.}}(\ell_1,\ell')| |\mathcal{M}^{\mathrm{f.r.}}(\ell_2,\ell')| \dd \ell'$. Then by~\eqref{eq:forest-circle-law} and Lemma~\ref{lem:levy-moment} again,
    \begin{align*}
        &\quad \int_{\mathbb{R}_+} |\mathcal{M}^{\mathrm{circ}}(\ell_1,\ell_2)| \ell_2^{ix} \dd \ell_2=\int_{\mathbb{R}_+^2} \ell' |\mathcal{M}^{\mathrm{f.r.}}(\ell_1,\ell')| |\mathcal{M}^{\mathrm{f.r.}}(\ell_2,\ell')| \ell_2^{ix} \dd \ell_2 \dd \ell' \\
        &=\frac{\Gamma(ix)}{\Gamma(i\frac{\gamma^2}{4}x)} \int_{\mathbb{R}_+^2} (\ell')^{\frac{\gamma^2}{4}ix} |\mathcal{M}^{\mathrm{f.r.}}(\ell_1,\ell')| \dd \ell'=\frac{\Gamma(ix)}{\Gamma(i\frac{\gamma^2}{4}x)} \ell_1^{-1} \mathbb{E}[Y_{\ell_1}^{\frac{\gamma^2}{4}ix}]=\frac{\Gamma(ix)}{\Gamma(i\frac{\gamma^2}{4}x)} \frac{4}{\gamma^2} \frac{\Gamma(-ix)}{\Gamma(-\frac{\gamma^2}{4}ix)} \ell_1^{ix-1} \,.
    \end{align*}
    The conclusion follows readily using the identities $\Gamma(ix)\Gamma(-ix)=\pi/(x\sinh(\pi x))$ and $\Gamma(\frac{\gamma^2}{4}ix)\Gamma(-\frac{\gamma^2}{4}ix)=\pi/(\frac{\gamma^2}{4}x\sinh(\frac{\gamma^2}{4} \pi x))$.
\end{proof}

\subsection{Conformal welding for CLE loops}
\label{subsec:welding}

Recall $\mathsf{m}_\kappa^{\mathrm{odd}}$ and $\mathsf{m}_\kappa^{\mathrm{even}}$, the laws of a non-boundary touching CLE$_\kappa$ loop, defined in Theorem~\ref{thm:cle-partition} (or Section~\ref{subsec:proof-iota}). The main result of this subsection is the following. 

\begin{theorem}\label{thm:odd-even-loop-welding}
    There exist measures $\QA_1^f$ and $\QA_2^f$ on generalized quantum surfaces, such that
    \begin{align}
        \QD_{1,0}^f \otimes \mathsf{m}_\kappa^{\mathrm{odd}}=\int_{\mathbb{R}_+} \ell' \Weld(\QA_1^f(\ell'), \QD_{1,0}^f(\ell')) \dd \ell', \label{eq:odd-loop-welding} \\
        \QD_{1,0}^f \otimes \mathsf{m}_\kappa^{\mathrm{even}}=\int_{\mathbb{R}_+} \ell' \Weld(\QA_2^f(\ell'), \QD_{1,0}^f(\ell')) \dd \ell', \label{eq:even-loop-welding}
    \end{align}
    where the right-hand side represents the uniform conformal welding along the inner forested boundary of a sample from $\QA_1^f$ (resp. $\QA_2^f$) and the forested boundary of a sample from $\QD_{1,0}^f$ conditioned on having the same generalized quantum length.
\end{theorem}

We begin with the conformal welding for the outermost CLE loop. Let $\mu$ be the law of the outermost CLE$_\kappa$ loop $\eta$ surrounding the marked point, and $T=\{\eta \cap \partial \mathbb{D} \neq \emptyset\}$.

\begin{proposition}\label{prop:single-loop-welding}
    There exist measures $\QA^f$ and $\QA_T^f$ on generalized quantum surfaces, such that
    \begin{align}
        \QD_{1,0}^f \otimes \mu &=\int_{\mathbb{R}_+} \ell' \Weld(\QA^f(\ell'), \QD_{1,0}^f(\ell')) \dd \ell', \label{eq:single-loop-welding} \\
        \QD_{1,0}^f \otimes \mu \1_T &=\int_{\mathbb{R}_+} \ell' \Weld(\QA_T^f(\ell'), \QD_{1,0}^f(\ell')) \dd \ell', \label{eq:touch-loop-welding}
    \end{align}
    where the right-hand side represents the uniform conformal welding along the inner forested boundary of a sample from $\QA^f$ (resp. $\QA_T^f$) and the forested boundary of a sample from $\QD_{1,0}$ conditioned on having the same generalized quantum length. See Figure~\ref{fig:single-loop-welding}.
\end{proposition}

\begin{proof}
    The proof is essentially identical to that of~\cite[Proposition 4.21]{SXZ24} except that here we consider the law of the $\CLE_\kappa$ loop rather than that of its inner boundary, so we do not weld together the two forested circles squeezed between the two quantum surfaces. More precisely, by (4.5) of~\cite[Lemma 4.22]{SXZ24}, $\QD_{1,0}^f \otimes \mu \1_T=C \int_{\mathbb{R}_+} \Weld(\widetilde{\mathrm{QA}}^f(\gamma^2-2;\ell),\QD_{1,0}^f(\ell)) \dd \ell$ for some constant $C>0$, so we may take $\QA_T^f=C \widetilde{\mathrm{QA}}^f(\gamma^2-2)$ (in~\cite{SXZ24}, the authors take $\mathcal{QA}_T=C \int_{\mathbb{R}_+} \ell \Weld(\widetilde{\mathrm{QA}}(\gamma^2-2;\ell),\mathcal{M}^{\mathrm{circ}}(\ell)) \dd \ell$). Similarly, the measure $\QA^f$ can be described using the uniform conformal welding of $\widetilde{\mathrm{QA}}^f(\gamma^2-2)$ and $\widetilde{\mathrm{QA}}^f(2-\frac{\gamma^2}{2})$.
\end{proof}

\begin{figure}
     \centering
     \includegraphics[scale=0.56]{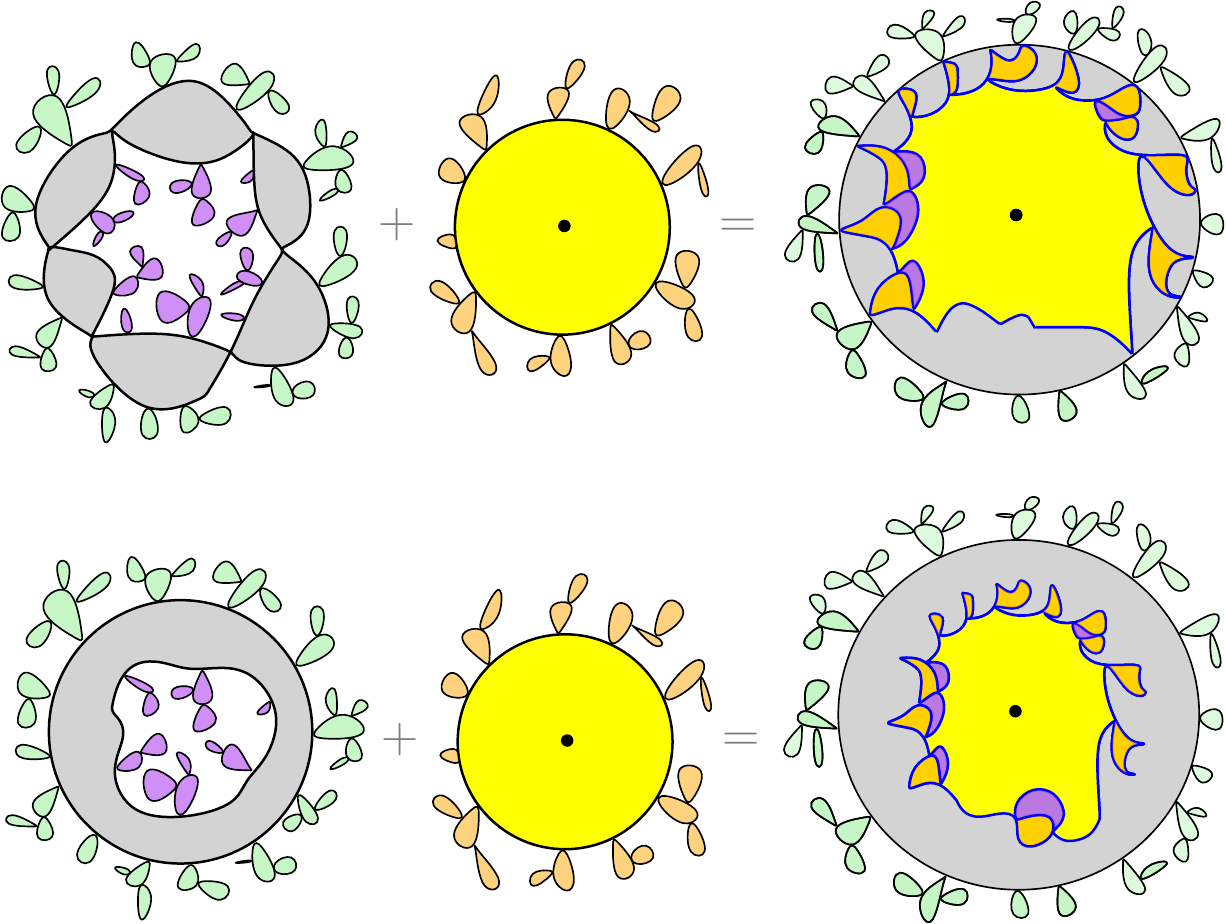}
     \caption{An illustration of Proposition~\ref{prop:single-loop-welding}. \textbf{Top}: A forested pinched thin quantum annulus from $\QA_T^f$. Its conformal welding with the yellow forested disk from $\QD_{1,0}^f$ gives another sample from $\QD_{1,0}^f$ decorated with a boundary-touching CLE loop. The quantum surface $\mathcal{QA}_T$ considered in~\cite{SXZ24} is the union of gray, purple and orange parts in the top-right panel, and is the conformal welding of $\QA_T$ and $\mathcal{M}^{\mathrm{circ}}$. \textbf{Bottom}: A forested quantum annulus from $\QA^f-\QA_T^f$. Its conformal welding with a forested disk from $\QD_{1,0}^f$ gives another sample from $\QD_{1,0}^f$ decorated with a non-boundary-touching CLE loop.}
     \label{fig:single-loop-welding}
 \end{figure}

In light of Proposition~\ref{prop:single-loop-welding}, the law of CLE loops of higher nested levels can be characterized by the uniform welding of several independent copies of $\QA^f$. Let $\QA^{1,f}=\QA^f$. For $k \ge 2$, the measure $\QA^{k,f}$ is defined recursively as follows:
\begin{equation}\label{eq:QA-k-f}
    \QA^{k,f}=\int_{\mathbb{R}_+} \ell' \Weld(\QA^{k-1,f}(\ell'), \QA^f(\ell')) \dd \ell'.
\end{equation}
Then we define
\begin{equation}\label{eq:QA12-f}
    \QA_1^f=\sum_{k \ge 1} \QA^{2k-1,f}-\QA_T^f, \quad \mbox{and} \quad \QA_2^f=\sum_{k \ge 1} \QA^{2k,f}.
\end{equation}
Here the subtraction makes sense as $\QD_{1,0}^f \otimes \mu \1_{T^c} =\int_{\mathbb{R}_+} \ell' \Weld((\QA^f-\QA_T^f)(\ell'), \QD_{1,0}^f(\ell')) \dd \ell'$, so that $\QA^f-\QA_T^f$ (and hence $\QA_1^f$) is a positive measure.

\begin{proof}[Proof of Theorem~\ref{thm:odd-even-loop-welding}]
    We show that the measures $\QA_1$ and $\QA_2$ described above satisfy~\eqref{eq:odd-loop-welding} and~\eqref{eq:even-loop-welding}. For $k \ge 1$, let $\mu^k$ be the law of the $k$-th level $\CLE_\kappa$ loop surrounding the marked point. By iteratively sampling the outermost CLE loop surrounding the marked point on $\QD_{1,0}^f$, we derive $\QD_{1,0}^f \otimes \mu^k=\int_{\mathbb{R}_+} \ell \Weld(\QA^{k,f}(\ell),\QD_{1,0}^f(\ell')) \dd \ell'$ from~\eqref{eq:single-loop-welding}. By the construction of nested CLE$_\kappa$ using radial SLE$_\kappa(\kappa-6)$, it follows from~\cite[Proposition 4.9]{ig4} that almost surely, the second level CLE$_\kappa$ loop does not touch the boundary. Hence $\mathsf{m}_\kappa^{\mathrm{odd}}=\mu \1_{T^c}+\sum_{k \ge 2} \mu^{2k-1}$ and $\mathsf{m}_\kappa^{\mathrm{even}}=\sum_{k \ge 1} \mu^{2k}$ and the conclusion follows.
\end{proof}

Next, we derive the laws of the boundary lengths of $\QA_1$ and $\QA_2$ (the spines of $\QA_1^f$ and $\QA_2^f$). We start by computing those of $\QA_T^f$ and $\QA^{k,f}$ for $k \ge 1$. Let $g=\frac{4}{\kappa}=\frac{\gamma^2}{4} \in (\frac{1}{2},1)$ and $\chi=\pi(1-g) \in (0,\frac{\pi}{2})$.

\begin{proposition}\label{prop:qat-law}
    For $\gamma \in (\sqrt{2},2)$, let $L_1'$ and $L_2'$ be the outer and inner boundary lengths of a sample from $\QA_T^f$. For any $t>0$ and $x \in \mathbb{R}$,
    \begin{equation}\label{eq:qat-law}
        \QA_T^f[L_1' e^{-tL_1'} (L_2')^{ix}]=t^{-ix-1} \Gamma(ix+1) \frac{2\cos\chi \sinh((\frac{4}{\gamma^2}-1)\pi x)}{\sinh(\frac{4}{\gamma^2} \pi x)} \,.
    \end{equation}
\end{proposition}

\begin{proof}
    By~\cite[Lemma 5.3]{SXZ24}\footnote{The original lemma was stated for $\gamma=\sqrt{8/3}$, but the formulae and proofs are actually valid for all $\gamma \in (\sqrt{2},2)$.}, for $\ell_1>0$ and $x \in \mathbb{R}$, we have
    \begin{equation}\label{eq:qat-cal-1}
        \int_{\mathbb{R}_+} |\mathcal{QA}_T(\ell_1,\ell_2)| \ell_2^{ix} \dd \ell_2=\frac{2\cos\chi \sinh(\pi(1-\frac{\gamma^2}{4})x)}{\sinh(\pi x)} \ell_1^{ix-1},
    \end{equation}
    where, as noted in the proof of Proposition~\ref{prop:single-loop-welding}, $\mathcal{QA}_T=\int_{\mathbb{R}_+} \ell \Weld(\QA_T(\ell), \mathcal{M}^{\mathrm{circ}}(\ell)) \dd \ell$, and $\QA_T$ is the spine of $\QA_T^f$. By Lemma~\ref{prop:forest-circle-weld} and Fubini's theorem,
    \begin{equation}\label{eq:qat-cal-2}
        \int_{\mathbb{R}_+} |\mathcal{QA}_T(\ell_1,\ell_2)| \ell_2^{ix} \dd \ell_2=\int_{\mathbb{R}_+^2} \ell |\QA_T(\ell_1,\ell)| |\mathcal{M}^{\mathrm{circ}}(\ell,\ell_2)| \ell_2^{ix} \dd \ell_2 \dd \ell=\frac{\sinh(\frac{\gamma^2}{4}\pi x)}{\sinh(\pi x)} \int_{\mathbb{R}_+} |\QA_T(\ell_1,\ell)| \ell^{ix} \dd \ell.
    \end{equation}
    A comparison between~\eqref{eq:qat-cal-1} and~\eqref{eq:qat-cal-2} implies that $\int_{\mathbb{R}_+} |\QA_T(\ell_1,\ell_2)| \ell_2^{ix} \dd \ell_2=\frac{2\cos\chi \sinh(\pi(1-\frac{\gamma^2}{4})x)}{\sinh(\frac{\gamma^2}{4}\pi x)} \ell_1^{ix-1}$. Thus, if we denote by $L_1$ and $L_2$ the outer and inner boundary lengths of a sample from $\QA_T$, then
    \begin{equation*}
        \QA_T[L_1 e^{-tL_1} L_2^{ix}]=\int_{\mathbb{R}_+^2} |\QA_T(\ell_1,\ell_2)| \ell_1 e^{-t \ell_1} \ell_2^{ix} \dd \ell_1 \dd \ell_2=t^{-ix-1} \Gamma(ix+1) \frac{2\cos\chi \sinh(\pi(1-\frac{\gamma^2}{4})x)}{\sinh(\frac{\gamma^2}{4}\pi x)},
    \end{equation*}
    where we used the identity $\int_{\mathbb{R}_+} \ell_1^{ix} e^{-t \ell_1} \dd \ell_1=t^{-ix-1} \Gamma(ix+1)$. The conclusion follows by an application of Lemma~\ref{lem:remove-fl} to $\mathcal{M}=\QA_T$, and using the identities $\Gamma(1+ix)\Gamma(-ix)=i \pi/\sinh(\pi x)$ and $\Gamma(1+i \frac{4}{\gamma^2} x)\Gamma(-i \frac{4}{\gamma^2} x)=i \pi/\sinh(\frac{4}{\gamma^2} \pi x)$.
\end{proof}

\begin{proposition}\label{prop:qaf-law}
    For $\gamma \in (\sqrt{2},2)$ and $k \ge 1$, let $L_1'$ and $L_2'$ be the outer and inner boundary lengths of a sample from $\QA^{k,f}$. For each $k \ge 1$, any $t>0$ and $x \in \mathbb{R}$,
    \begin{equation}\label{eq:qaf-law}
        \QA^{k,f}[L_1' e^{-tL_1'} (L_2')^{ix}]=t^{-ix-1} \Gamma(ix+1) (2\cos\chi)^k (2\cosh(\pi x))^{-k} \,.
    \end{equation}
\end{proposition}

\begin{proof}
    Let $\QA$ be the spine of $\QA^f$, then the measure $\mathcal{QA}$ defined in~\cite[Proposition 4.21]{SXZ24} is the uniform conformal welding of $\QA$ and $\mathcal{M}^{\mathrm{circ}}$. Following the proof of Proposition~\ref{prop:qat-law}, the formula $\int_{\mathbb{R}_+} |\mathcal{QA}(\ell_1,\ell_2)| \ell_2^{ix} \dd \ell_2=\frac{2\cos\chi}{2\cosh(\frac{\gamma^2}{4} \pi x)} \ell_1^{ix-1}$ from~\cite[Lemma 5.3]{SXZ24} implies~\eqref{eq:qaf-law} for $k=1$.
    This also shows that $\int_{\mathbb{R}_+} |\QA^f(\ell_1',\ell_2')| (\ell_2')^{ix} \dd \ell_2'=\frac{2\cos\chi}{2\cosh(\pi x)} (\ell_1')^{ix-1}$. By~\eqref{eq:QA-k-f} and Fubini's theorem,
    \begin{align*}
        &\quad \int_{\mathbb{R}_+} |\QA^{2,f}(\ell_1',\ell_2')| (\ell_2')^{ix} \dd \ell_2'=\int_{\mathbb{R}_+^2} \ell' |\QA^f(\ell_1',\ell')| |\QA^f(\ell',\ell_2')| (\ell_2')^{ix} \dd \ell_2' \dd \ell' \\
        &=\frac{2\cos\chi}{2\cosh(\pi x)} \int_{\mathbb{R}_+^2} |\QA^f(\ell_1',\ell')| (\ell')^{ix} \dd \ell'=\Big( \frac{2\cos\chi}{2\cosh(\pi x)} \Big)^2 (\ell_1')^{ix-1},
    \end{align*}
    which, after integrating against $\ell_1' e^{-t \ell_1'} \dd \ell_1'$, implies~\eqref{eq:qaf-law} for $k=2$. The proof for general $k \ge 2$ follows by simple induction.
\end{proof}

\begin{proposition}\label{prop:qa-1-2}
    For $\gamma \in (\sqrt{2},2)$, let $L_1$ and $L_2$ be the outer and inner boundary lengths of a sample from $\QA_1$ or $\QA_2$. For any $t>0$ and $x \in \mathbb{R}$, we have
    \begin{equation}\label{eq:qa-1}
        \QA_1[L_1 e^{-tL_1} L_2^{ix}]=t^{-ix-1} \Gamma(1+ix) \frac{2\cos\chi \sinh(\pi x)}{\sinh(\frac{\gamma^2}{4} \pi x)} \left( \frac{2 \cosh(\frac{\gamma^2}{4} \pi x)}{(2 \cosh(\frac{\gamma^2}{4} \pi x))^2-(2 \cos\chi)^2}- \frac{\sinh((1-\frac{\gamma^2}{4})\pi x)}{\sinh(\pi x)}\right),
    \end{equation}
    and
    \begin{equation}\label{eq:qa-2}
        \QA_2[L_1 e^{-tL_1} L_2^{ix}]=t^{-ix-1} \Gamma(1+ix) \frac{2\cos\chi \sinh(\pi x)}{\sinh(\frac{\gamma^2}{4} \pi x)} \cdot \frac{2\cos\chi}{(2 \cosh(\frac{\gamma^2}{4} \pi x))^2-(2 \cos\chi)^2} \,.
    \end{equation}
\end{proposition}

\begin{proof}
    Let $L_1'$ and $L_2'$ be the outer and inner boundary lengths of a sample from $\QA_1^f$ or $\QA_2^f$. By~\eqref{eq:QA12-f}, Propositions~\ref{prop:qat-law} and~\ref{prop:qaf-law}, we have
    \begin{align}\label{eq:qaf-1}
        \QA_1^f[L_1' e^{-tL_1'} (L_2')^{ix}] &=\sum_{k \ge 1}\QA^{2k-1, f}[L_1' e^{-tL_1'} (L_2')^{ix}]-\QA_T^f[L_1' e^{-tL_1'} (L_2')^{ix}] \nonumber \\
        &=t^{-ix-1} \Gamma(ix+1) \left( \sum_{k \ge 1} (2\cos\chi)^{2k-1} (2\cosh(\pi x))^{-(2k-1)} - \frac{2\cos\chi \sinh((\frac{4}{\gamma^2}-1)\pi x)}{\sinh(\frac{4}{\gamma^2} \pi x)} \right) \nonumber \\
        &=t^{-ix-1} \Gamma(ix+1) \cdot 2\cos\chi \left( \frac{2\cosh(\pi x)}{(2 \cosh(\pi x))^2-(2 \cos\chi)^2}-\frac{\sinh((\frac{4}{\gamma^2}-1)\pi x)}{\sinh(\frac{4}{\gamma^2} \pi x)} \right).
    \end{align}
    Similarly, we have
    \begin{align}\label{eq:qaf-2}
        \QA_2^f[L_1' e^{-tL_1'} (L_2')^{ix}] &=\sum_{k \ge 1}\QA^{2k-1, f}[L_1' e^{-tL_1'} (L_2')^{ix}] \nonumber \\
        &=t^{-ix-1} \Gamma(ix+1) \sum_{k \ge 1} (2\cos\chi)^{2k} (2\cosh(\pi x))^{-2k} \nonumber \\
        &=t^{-ix-1} \Gamma(ix+1) \cdot \frac{(2\cos\chi)^2}{(2 \cosh(\pi x))^2-(2 \cos\chi)^2} \,.
    \end{align}
    By Lemma~\ref{lem:remove-fl} applied to $\mathcal{M}=\QA_1$, $\mathcal{M}=\QA_2$ and using the identities $\Gamma(1+ix)\Gamma(-ix)=i \pi/\sinh(\pi x)$ and $\Gamma(1+i \frac{\gamma^2}{4} x)\Gamma(-i \frac{\gamma^2}{4} x)=i \pi/\sinh(\frac{\gamma^2}{4} \pi x)$, we derive~\eqref{eq:qa-1} and~\eqref{eq:qa-2} from~\eqref{eq:qaf-1} and~\eqref{eq:qaf-2}.
\end{proof}

\subsection{Identification of the Liouville field}
\label{subsec:symmetry}

In this short subsection, we will follow the argument in~\cite{ARS22} and show that for both $\QA_1$ and $\QA_2$, the conditional law of $\phi$ given $\tau$ can be described by the Liouville field $\LF_\tau$.

\begin{proposition}\label{prop:idenfity-lf}
    There exists a measure $m_1(\dd \tau)$ on $(0,\infty)$ such that if we sample $(\tau,\phi)$ from $\LF_{\tau}(\dd \phi) m_1(\dd \tau)$, then the law of the quantum surface $(\mathcal{C}_\tau,\phi)/\!\!\sim_\gamma$ is $\QA_1$. Similarly, there exists a measure $m_2(\dd \tau)$ on $(0,\infty)$ such that if we sample $(\tau,\phi)$ from $\LF_{\tau}(\dd \phi) m_2(\dd \tau)$, then the law of the quantum surface $(\mathcal{C}_\tau,\phi)/\!\!\sim_\gamma$ is $\QA_2$.
\end{proposition}

To prove Proposition~\ref{prop:idenfity-lf}, it suffices to show the symmetry of the law of the quantum annulus.

\begin{proposition}\label{prop:QA-12-symmetry}
    For each $a,b>0$, we have $\QA_1(a,b)=\QA_1(b,a)$ and $\QA_2(a,b)=\QA_2(b,a)$.
\end{proposition}

\begin{proof}
    First, by~\cite[Proposition 7.6]{ACSW24} we have $\QA^f(a,b)=\QA^f(b,a)$\footnote{Our $\QA^f$ is denoted by $\mathrm{GA}$ in~\cite{ACSW24}; see Proposition 5.15 therein.}, and it follows from a simple induction that $\QA^{k,f}(a,b)=\QA^{k,f}(b,a)$ for any $k \ge 1$. Moreover, by the proof of Proposition~\ref{prop:single-loop-welding}, $\QA_T$ is a constant multiple of $\widetilde{\mathrm{QA}}(\gamma^2-2)$, which is symmetric by Definition~\ref{def:pinched-QA}, hence $\QA_T(a,b)=\QA_T(b,a)$ and thus $\QA_T^f(a,b)=\QA_T^f(b,a)$. Therefore, we have $\QA_1^f(a,b)=\QA_1^f(b,a)$ and $\QA_2^f(a,b)=\QA_2^f(b,a)$. By removing the attached forested lines, we get $\QA_1(a,b)=\QA_1(b,a)$ and $\QA_2(a,b)=\QA_2(b,a)$.
\end{proof}

\begin{proof}[Proof of Proposition~\ref{prop:idenfity-lf}]
    The proof follows the same line as~\cite[Proposition 4.4]{ARS22}, except that the symmetry property (Lemma 4.7 therein) is replaced by Proposition~\ref{prop:QA-12-symmetry}.
\end{proof}

\subsection{Random moduli of odd and even loops}
\label{subsec:pfn}

In this subsection, we compute the random moduli $m_1(\dd \tau)$ and $m_2(\dd \tau)$ derived in Proposition~\ref{prop:idenfity-lf}, and finally provide a proof for Theorem~\ref{thm:cle-partition}. To this end, the following result from~\cite[Theorem 1.6]{ARS22} will be a key input to extract the law of the modulus of a quantum annulus described by Liouville fields from its joint distribution of boundary length.
Recall that $\partial_1 \mathcal{C}_\tau=\{0\} \times [0,1]/\!\!\sim$ and $\partial_2 \mathcal{C}_\tau=\{\tau\} \times [0,1]/\!\!\sim$ are two boundary components of $\mathcal{C}_\tau$, on which the quantum length measure $\nu_\phi$ is well-defined.

\begin{proposition}\label{prop:laplace-m}
    Let $\gamma \in (0,2)$ and $m(\dd \tau)$ be any measure on $(0,\infty)$. For a measurable function $f$ on $\mathbb{R}_+^2$, we write $\langle f(L_1,L_2) \rangle_\gamma:=\int f(\nu_\phi(\partial_1 \mathcal{C}_\tau), \nu_\phi(\partial_2 \mathcal{C}_\tau)) \LF_\tau(\dd \phi) m(\dd \tau)$, then for any $x \in \mathbb{R}$,
    \[ \int_{\mathbb{R}_+} \exp\left(-\frac{\pi \gamma^2 x^2 \tau}{4} \right) m(\dd \tau)=\frac{2\sinh(\frac{\gamma^2}{4}\pi x)}{\pi \gamma x \Gamma(1+ix)} \langle L_1 e^{-L_1} L_2^{ix} \rangle_\gamma \,. \]
\end{proposition}

Write $\mathsf{q}=e^{-\pi/\tau}$. Recall that $g=\frac{4}{\kappa}=\frac{\gamma^2}{4} \in (\frac{1}{2},1)$, $\chi=\pi(1-g)$, and let $\mathbf{c}=1-\frac{6(1-g)^2}{g}$ be the central charge. The following theorem gives the explicit formulae for $m_1(\dd \tau)$ and $m_2(\dd \tau)$ in terms of $\mathsf{q}$.

\begin{theorem}\label{thm:mod}
    The measures $m_1(\dd \tau)$ and $m_2(\dd \tau)$ in Proposition~\ref{prop:idenfity-lf} are given by
    \[ m_1(\dd \tau)=\mathcal{Z}_{\mathrm{odd}}(\tau) \frac{2\cos\chi}{\sqrt{2} \pi} \eta(2i\tau) \dd \tau \ \mbox{ and} \quad m_2(\dd \tau)=\mathcal{Z}_{\mathrm{even}}(\tau) \frac{2\cos\chi}{\sqrt{2} \pi} \eta(2i\tau) \dd \tau \,, \]
    where $\eta(i\tau)=e^{-\frac{\pi \tau}{12}} \prod_{k=1}^\infty (1-e^{-2\pi k\tau})$ is the Dedekind eta function, and
    \begin{align}
        \mathcal{Z}_{\mathrm{odd}}(\tau)&=\mathsf{q}^{-\frac{\mathbf{c}}{24}} \prod_{k=1}^{\infty} (1-\mathsf{q}^k)^{-1} \sum_{p \in \mathbb{Z}, 2 \mid p} \frac{\sin(p+1)\chi}{\sin\chi} \mathsf{q}^{\frac{gp^2}{4}-\frac{(1-g)p}{2}} \,, \label{eq:odd-pfn} \\
        \mathcal{Z}_{\mathrm{even}}(\tau)&=\mathsf{q}^{-\frac{\mathbf{c}}{24}} \prod_{k=1}^{\infty} (1-\mathsf{q}^k)^{-1} \sum_{p \in \mathbb{Z}, 2 \nmid p} \frac{\sin(p+1)\chi}{\sin\chi} \mathsf{q}^{\frac{gp^2}{4}-\frac{(1-g)p}{2}} \label{eq:even-pfn} \,.
    \end{align}
    Here, $2 \mid p$ means that $p$ is even and $2 \nmid p$ means that $p$ is odd.
\end{theorem}

\begin{remark}\label{rem:pfn}
    The partition functions here match with those of critical dense O$(n,n')$ model on the annulus~\cite[Equation (2)]{Cardy06}. Let $n=2\cos\chi$ and $Z(\tau;n,n')$ be Cardy's annulus partition function, then $\mathcal{Z}_{\mathrm{odd}}(\tau)+\mathcal{Z}_{\mathrm{even}}(\tau)=Z(\tau;n,n'=n)$ and $\mathcal{Z}_{\mathrm{odd}}(\tau)-\mathcal{Z}_{\mathrm{even}}(\tau)=Z(\tau;n,n'=-n)$. The $\frac{2\cos\chi}{\sqrt{2} \pi} \eta(2i\tau)$ term corresponds to the product of the GFF partition function and the ghost partition function; see~\cite[Section 5.2]{Rem18} for details.
\end{remark}

\begin{proof}
    On the one hand, by Proposition~\ref{prop:laplace-m} and~\eqref{eq:qa-1}, we have
    \begin{align}\label{eq:qa1-ars}
        \int_{\mathbb{R}_+} \exp\left(-\frac{\pi \gamma^2 x^2 \tau}{4} \right) m_1(\dd \tau) &=\frac{2\sinh(\frac{\gamma^2}{4}\pi x)}{\pi \gamma x \Gamma(1+ix)} \QA_1[L_1 e^{-L_1} L_2^{ix}] \nonumber \\
        &=\frac{4\cos\chi \sinh(\pi x)}{\pi \gamma x} \left( \frac{2 \cosh(\frac{\gamma^2}{4} \pi x)}{(2 \cosh(\frac{\gamma^2}{4} \pi x))^2-(2 \cos\chi)^2}- \frac{\sinh((1-\frac{\gamma^2}{4})\pi x)}{\sinh(\pi x)}\right).
    \end{align}
    On the other hand, by the modular transformation of Dedekind eta function, $\eta(2i\tau)=\frac{1}{\sqrt{2\tau}} \eta(\frac{i}{2\tau})=\frac{1}{\sqrt{2\tau}} \mathsf{q}^{\frac{1}{24}} \prod_{k=1}^{\infty} (1-\mathsf{q}^k)$. For $x>0$, we have
    \begin{align}\label{eq:odd-pfn-laplace}
        &\quad \int_{\mathbb{R}_+} \exp\left(-\frac{\pi \gamma^2 x^2 \tau}{4} \right) \mathcal{Z}_{\mathrm{odd}}(\tau) \frac{2\cos\chi}{\sqrt{2} \pi} \eta(2i\tau) \dd \tau \nonumber\\
        &=\frac{2\cos\chi}{\pi} \int_{\mathbb{R}_+} \frac{1}{2\sqrt \tau} \exp\left(-\frac{\pi \gamma^2 x^2 \tau}{4} \right) \sum_{p \in \mathbb{Z}, 2 \mid p} \frac{\sin(p+1)\chi}{\sin\chi} \mathsf{q}^{\frac{gp^2}{4}-\frac{(1-g)p}{2}+\frac{1-\mathbf{c}}{24}} \dd \tau \nonumber \\
        &=\frac{2\cos\chi}{\pi} \sum_{p \in \mathbb{Z}, 2 \mid p} \frac{\sin(p+1)\chi}{\sin\chi} \int_{\mathbb{R}_+} \frac{1}{2\sqrt{\tau}} \exp\left(-\frac{\pi \gamma^2 x^2 \tau}{4} \right) \exp\left(-\frac{g}{4}\left(p-\frac{1-g}{g} \right)^2 \cdot \frac{\pi}{\tau} \right) \dd \tau \nonumber \\
        &=\frac{2\cos\chi}{\pi \gamma x} \sum_{p \in \mathbb{Z}, 2 \mid p} \frac{\sin(p+1)\chi}{\sin\chi} \exp\left( -\frac{\gamma^2}{4} \pi x |p+1-\frac{4}{\gamma^2}| \right),
    \end{align}
    where we apply Fubini's theorem to interchange the summation and the integral on the third line (absolute convergence can be verified via a similar computation as below). The last step follows by a change of variable $\tau=u^2$ and the identity that $\int_{\mathbb{R}_+} \exp(-au^2-b/u^2) \dd u=\frac{1}{2} \sqrt{\pi/a} \exp(-2\sqrt{ab})$ for $a,b>0$. Note that the right hand side of~\eqref{eq:odd-pfn-laplace} is the imaginary part of
    \begin{equation*}
        \mathsf{S}:=\frac{2\cos\chi}{\pi \gamma x \sin\chi} \sum_{p \in \mathbb{Z}, 2 \mid p} e^{i(p+1)\chi} e^{-\frac{\gamma^2}{4} \pi x |p+1-\frac{4}{\gamma^2}|} \,.
    \end{equation*}
    Since $\frac{4}{\gamma^2} \in (1,2)$, we may split the summation into $p>0$ and $p \le 0$, and compute $\mathsf{S}$ by
    \begin{align*}
        \frac{\pi \gamma x \sin\chi}{2\cos\chi} \cdot \mathsf{S}&=\sum_{p \in \mathbb{Z}_{+}, 2 \mid p} e^{i(p+1)\chi} e^{-\frac{\gamma^2}{4} \pi x (p+1-\frac{4}{\gamma^2})} + \sum_{p \in \mathbb{Z}_{\le 0}, 2 \mid p} e^{i(p+1)\chi} e^{\frac{\gamma^2}{4} \pi x (p+1-\frac{4}{\gamma^2})} \\
        &=\frac{e^{3i\chi} e^{-\frac{\gamma^2}{4} \pi x (3-\frac{4}{\gamma^2})}}{1-e^{2i\chi} e^{-\frac{\gamma^2}{2} \pi x}}+\frac{e^{i\chi} e^{\frac{\gamma^2}{4} \pi x (1-\frac{4}{\gamma^2})}}{1-e^{-2i\chi} e^{-\frac{\gamma^2}{2} \pi x}} \\
        &=\frac{e^{3i\chi} \sinh((1-\frac{\gamma^2}{4})\pi x) - e^{i\chi} \sinh((1-\frac{3\gamma^2}{4})\pi x)}{\cosh(\frac{\gamma^2}{2}\pi x)-\cos(2\chi)} \,.
    \end{align*}
    Therefore,
    \begin{align}\label{eq:odd-pfn-laplace-2}
        &\quad \int_{\mathbb{R}_+} \exp\left(-\frac{\pi \gamma^2 x^2 \tau}{4} \right) \mathcal{Z}_{\mathrm{odd}}(\tau) \frac{2\cos\chi}{\sqrt{2} \pi} \eta(2i\tau) \dd \tau = \mathrm{Im}(\mathsf{S}) \nonumber \\
        &=\frac{2\cos\chi}{\pi \gamma x \sin\chi} \cdot \frac{\sin(3\chi) \sinh((1-\frac{\gamma^2}{4})\pi x) - \sin\chi \sinh((1-\frac{3\gamma^2}{4})\pi x)}{\cosh(\frac{\gamma^2}{2}\pi x)-\cos(2\chi)} \nonumber \\
        &=\frac{2\cos\chi}{\pi \gamma x} \left( \frac{2 \sinh(\pi x) \cosh(\frac{\gamma^2}{4} \pi x)}{\cosh(\frac{\gamma^2}{2}\pi x)-\cos(2\chi)}- 2\sinh((1-\frac{\gamma^2}{4})\pi x) \right).
    \end{align}
    The above equation follows from elementary identities for trigonometric and hyperbolic functions and we omit the detailed steps of calculation. Comparing~\eqref{eq:qa1-ars},~\eqref{eq:odd-pfn-laplace-2} and taking the inverse Laplace transform, we get $m_1(\dd \tau)=\mathcal{Z}_{\mathrm{odd}}(\tau) \frac{2\cos\chi}{\sqrt{2} \pi} \eta(2i\tau) \dd \tau$. A similar calculation shows that $m_2(\dd \tau)=\mathcal{Z}_{\mathrm{even}}(\tau) \frac{2\cos\chi}{\sqrt{2} \pi} \eta(2i\tau) \dd \tau$.
\end{proof}

To analyze the behavior of the partition functions~\eqref{eq:odd-pfn} and~\eqref{eq:even-pfn} as $\tau \to \infty$, we will rewrite them in terms of $r=e^{-2\pi \tau}$; this is the so-called closed channel expansion in conformal field theory.

\begin{lemma}\label{lem:pfn-r}
    Let $\mathcal{Z}_{\mathrm{odd}}(\tau)$ and $\mathcal{Z}_{\mathrm{even}}(\tau)$ be as in Theorem~\ref{thm:mod}. Then
    \begin{align}
        \mathcal{Z}_{\mathrm{odd}}(\tau)&=\frac{1}{\sqrt{2g}} \cdot r^{-\frac{\mathbf{c}}{12}} \prod_{k=1}^{\infty} (1-r^{2k})^{-1} \sum_{m \in \mathbb{Z}} (-1)^m \frac{\sin(\frac{\chi+\pi m}{g})}{\sin\chi} r^{\frac{(\chi+\pi m)^2-\chi^2}{2\pi^2 g}} \,, \label{eq:odd-pfn-r} \\
        \mathcal{Z}_{\mathrm{even}}(\tau)&=\frac{1}{\sqrt{2g}} \cdot r^{-\frac{\mathbf{c}}{12}} \prod_{k=1}^{\infty} (1-r^{2k})^{-1} \sum_{m \in \mathbb{Z}} \frac{\sin(\frac{\chi+\pi m}{g})}{\sin\chi} r^{\frac{(\chi+\pi m)^2-\chi^2}{2\pi^2 g}} \label{eq:even-pfn-r} \,.
    \end{align}
\end{lemma}

\begin{proof}
    By~\eqref{eq:odd-pfn}, we have
    \begin{align}\label{eq:odd-pfn-trans}
        \mathcal{Z}_{\mathrm{odd}}(\tau)&=\left[ \mathsf{q}^{-\frac{1}{24}} \prod_{k=1}^{\infty} (1-\mathsf{q}^k)^{-1} \right] \mathsf{q}^{\frac{1-\mathbf{c}}{24}} \sum_{p \in \mathbb{Z}} \frac{\sin(2p+1)\chi}{\sin\chi} \mathsf{q}^{gp^2-(1-g)p} \nonumber \\
        &=\left[ \mathsf{q}^{-\frac{1}{24}} \prod_{k=1}^{\infty} (1-\mathsf{q}^k)^{-1} \right] \mathsf{q}^{\frac{1-\mathbf{c}}{24}} (\sin\chi)^{-1} \ \mathrm{Im} \left( e^{i\chi} \sum_{p \in \mathbb{Z}} e^{2ip\chi} \mathsf{q}^{gp^2-(1-g)p} \right).
    \end{align}
    Using the Poisson sum formula and recalling that $\mathsf{q}=e^{-\pi/\tau}$, the summation in the bracket is equal to
    \begin{equation}\label{eq:poisson-sum-formula-odd}
        \sum_{m \in \mathbb{Z}} \int_{\mathbb{R}} e^{-\frac{g\pi}{\tau} p^2+[(1-g) \frac{\pi}{\tau}+i(2\chi+2\pi m)]p} \dd p
        =\sqrt{\tau/g} \sum_{m \in \mathbb{Z}} \exp \left( \tfrac{(1-g)^2 \pi}{4g\tau}-\tfrac{\tau}{g\pi}(\chi+\pi m)^2+i \tfrac{1-g}{g}(\chi+\pi m) \right).
    \end{equation}
    The first term in the exponential cancels out with $\mathsf{q}^{\frac{1-\mathbf{c}}{24}}$ in~\eqref{eq:odd-pfn-trans}. Also, by the modular transformation,
    \begin{equation}\label{eq:modular-trans}
        \mathsf{q}^{1/24} \prod_{k=1}^{\infty} (1-\mathsf{q}^k)=\eta(\tfrac{i}{2\tau})=\sqrt{2\tau} \eta(2i\tau)=\sqrt{2\tau} \cdot r^{1/12} \prod_{k=1}^{\infty} (1-r^{2k}) \,.
    \end{equation}
    Combining~\eqref{eq:odd-pfn-trans},~\eqref{eq:poisson-sum-formula-odd} and~\eqref{eq:modular-trans}, we derive that
    \begin{align*}
        \mathcal{Z}_{\mathrm{odd}}(\tau)&=\left[ \frac{1}{\sqrt{2\tau}} \cdot r^{-\frac{1}{12}} \prod_{k=1}^{\infty} (1-r^{2k})^{-1} \right] (\sin\chi)^{-1} \sqrt{\tau/g} \sum_{m \in \mathbb{Z}} e^{-\frac{\tau}{g\pi}(\chi+\pi m)^2} \sin(\tfrac{1-g}{g}(\chi+\pi m)+\chi) \\
        &=\frac{1}{\sqrt{2g}} \cdot r^{-\frac{\mathbf{c}}{12}} \prod_{k=1}^{\infty} (1-r^{2k})^{-1} \sum_{m \in \mathbb{Z}} (-1)^m \frac{\sin(\frac{\chi+\pi m}{g})}{\sin\chi} r^{\frac{(\chi+\pi m)^2-\chi^2}{2\pi^2 g}} \,.
    \end{align*}
    This proves~\eqref{eq:odd-pfn-r};~\eqref{eq:even-pfn-r} can be proved analogously.
\end{proof}

\begin{proof}[Proof of Theorem~\ref{thm:cle-partition} for $\kappa \in (4,8)$]
    Recall from Definition~\ref{def:modulus} that for a non-boundary-touching loop $\eta$ in $\mathbb{D}$, \ $\tau=\tau(\eta)$ is the modulus of the annular domain $A_\eta$. By Proposition~\ref{prop:idenfity-lf} and Theorem~\ref{thm:mod},
    \begin{equation*}
        \QA_1=\mathbf{1}_{\tau>0} \cdot \mathcal{Z}_{\mathrm{odd}}(\tau) \frac{2\cos\chi}{\sqrt{2} \pi} \eta(2i\tau) \LF_\tau(\dd \phi) \dd \tau \quad \mbox{and} \quad
        \QA_2=\mathbf{1}_{\tau>0} \cdot \mathcal{Z}_{\mathrm{even}}(\tau) \frac{2\cos\chi}{\sqrt{2} \pi} \eta(2i\tau) \LF_\tau(\dd \phi) \dd \tau \,.
    \end{equation*}
    Together with Theorem~\ref{thm:odd-even-loop-welding} and Lemma~\ref{lem:pfn-r}, we get that $\mathsf{m}_\kappa^{\mathrm{odd}}$ and $\mathsf{m}_\kappa^{\mathrm{even}}$ are absolutely continuous and
    \begin{equation}\label{eq:pfn-ratio}
        \frac{\mathsf{m}_\kappa^{\mathrm{odd}}(\dd \eta)}{\mathsf{m}_\kappa^{\mathrm{even}}(\dd \eta)} = \frac{\mathcal{Z}_{\mathrm{odd}}(\tau)}{\mathcal{Z}_{\mathrm{even}}(\tau)}=\frac{\sum_{m \in \mathbb{Z}} (-1)^m \sin(\frac{\chi+\pi m}{g}) r^{\frac{(\chi+\pi m)^2-\chi^2}{2\pi^2 g}}}{\sum_{m \in \mathbb{Z}} \sin(\frac{\chi+\pi m}{g}) r^{\frac{(\chi+\pi m)^2-\chi^2}{2\pi^2 g}}} \,,
    \end{equation}
    which simplifies to~\eqref{eq:cle-partition}. Moreover, one can easily check that for each $\kappa \in (4,8)$,
    \begin{align}
        \frac{\mathsf{m}_\kappa^{\mathrm{odd}}(\dd \eta)}{\mathsf{m}_\kappa^{\mathrm{even}}(\dd \eta)} &=\frac{\sum_{m \in \mathbb{Z}} (-1)^m \sin(\frac{\kappa}{4}(m+1)\pi) r^{\frac{\kappa}{8}m^2+(\frac{\kappa}{4}-1)m}}{\sum_{m \in \mathbb{Z}} \sin(\frac{\kappa}{4}(m+1)\pi) r^{\frac{\kappa}{8}m^2+(\frac{\kappa}{4}-1)m}} \nonumber \\
        &=\frac{\sin(\frac{\kappa-4}{4}\pi)+\sin(\frac{\kappa-4}{2}\pi) r^{\frac{3\kappa}{8}-1}-\sin(\frac{\kappa-4}{4}\pi) r^2+o(r^2)}{\sin(\frac{\kappa-4}{4}\pi)-\sin(\frac{\kappa-4}{2}\pi) r^{\frac{3\kappa}{8}-1}-\sin(\frac{\kappa-4}{4}\pi) r^2+o(r^2)} \nonumber \\
        &=1+4\cos(\tfrac{\kappa-4}{4}\pi) r^{\frac{3\kappa}{8}-1}+O(r^{2(\frac{3\kappa}{8}-1)}) \,.
    \end{align}
    This completes the proof.
\end{proof}
\section{Mixing rate of simple CLE}\label{sec:simple}

In this section, we briefly discuss how the proof in Section~\ref{sec:lqg} can be adapted to the case for simple CLE.
The proof will be simpler than in the case of $\kappa \in (4,8)$, as it involves only ordinary quantum surfaces and no forested lines.

Fix $\kappa \in (8/3,4)$ and let $\gamma=\sqrt{\kappa}$. First, in place of Theorem~\ref{thm:odd-even-loop-welding}, we prove conformal welding results for odd and even level CLE$_\kappa$ loops.
Sample CLE$_\kappa$ on a quantum disk $\QD_{1,0}$ with one interior marked point, and let $\mu$ be the law of the outermost CLE$_\kappa$ loop surrounding the marked point.
By~\cite[Theorem 4.2]{ARS22}, there exist a measure $\mathtt{QA}$ on quantum surfaces with annular topology such that
\begin{equation}\label{eq:single-loop-welding-simple}
    \QD_{1,0} \otimes \mu=\int_{\mathbb{R}_+} \ell \Weld(\mathtt{QA}(\ell),\QD_{1,0}(\ell)) \dd \ell.
\end{equation}
For $k \ge 1$, define $\mathtt{QA}^k$ as the uniform welding of $k$ independent copies of $\mathtt{QA}$,
and let $\mathtt{QA}_1=\sum_{k \ge 1} \mathtt{QA}^{2k-1}$ and $\mathtt{QA}_2=\sum_{k \ge 1} \mathtt{QA}^{2k}$. Then~\eqref{eq:single-loop-welding-simple} implies the following.

\begin{theorem}\label{thm:odd-even-loop-welding-simple}
    The measures $\mathtt{QA}_1$ and $\mathtt{QA}_2$ on quantum surfaces satisfy that
    \begin{align}
        \QD_{1,0} \otimes \mathsf{m}^{\mathrm{odd}}=\int_{\mathbb{R}_+} \ell \Weld(\mathtt{QA}_1(\ell), \QD_{1,0}(\ell)) \dd \ell, \label{eq:odd-loop-welding-simple} \\
        \QD_{1,0} \otimes \mathsf{m}^{\mathrm{even}}=\int_{\mathbb{R}_+} \ell \Weld(\mathtt{QA}_2(\ell), \QD_{1,0}(\ell)) \dd \ell. \label{eq:even-loop-welding-simple}
    \end{align}
\end{theorem}

\begin{proof}
    The proof follows the same line as that of Theorem~\ref{thm:odd-even-loop-welding}, except that for $\kappa \in (8/3,4)$, CLE$_\kappa$ loops never touch the boundary.
\end{proof}

The following proposition gives the laws of the boundary lengths of $\mathtt{QA}_1$ and $\mathtt{QA}_2$.

\begin{proposition}\label{prop:qa-1-2-simple}
    For $\gamma \in (\sqrt{8/3},2)$, let $L_1$ and $L_2$ be the outer and inner boundary lengths of a sample from $\mathtt{QA}_1$ or $\mathtt{QA}_2$. For any $t>0$ and $x \in \mathbb{R}$, we have
    \begin{equation}\label{eq:qa-1-simple}
        \mathtt{QA}_1[L_1 e^{-tL_1} L_2^{ix}]=t^{-ix-1} \Gamma(1+ix) \frac{4\cos\chi \cosh(\pi x)}{(2 \cosh(\pi x))^2-(2 \cos\chi)^2},
    \end{equation}
    and
    \begin{equation}\label{eq:qa-2-simple}
        \mathtt{QA}_2[L_1 e^{-tL_1} L_2^{ix}]=t^{-ix-1} \Gamma(1+ix) \frac{(2\cos\chi)^2}{(2 \cosh(\pi x))^2-(2 \cos\chi)^2}.
    \end{equation}
\end{proposition}

\begin{proof}
    By~\cite[Lemmas 4.8 and 5.5]{ARS22}, for each $k \ge 1$, any $t>0$ and $x \in \mathbb{R}$, we have
    \begin{equation}\label{eq:qa-law-simple}
        \mathtt{QA}^k[L_1 e^{-tL_1} L_2^{ix}]=t^{-ix-1} \Gamma(1+ix) (2\cos\chi)^k (2\cosh(\pi x))^{-k}.
    \end{equation}
    Summing up~\eqref{eq:qa-law-simple} yields~\eqref{eq:qa-1-simple} and~\eqref{eq:qa-2-simple}.
\end{proof}

By~\cite[Proposition 7.5]{ACSW24}, the law of $\mathtt{QA}$ is symmetric, which implies that the laws of $\mathtt{QA}_1$ and $\mathtt{QA}_2$ are also symmetric. This leads to an analog of Proposition~\ref{prop:idenfity-lf}.

\begin{proposition}\label{prop:idenfity-lf-simple}
    For $i=1,2$, there exists a measure $\widetilde m_i(\dd \tau)$ on $(0,\infty)$ such that if we sample $(\tau,\phi)$ from $\LF_{\tau}(\dd \phi) \widetilde m_i(\dd \tau)$, then the law of the quantum surface $(\mathcal{C}_\tau,\phi)/\!\!\sim_\gamma$ is $\mathtt{QA}_i$.
\end{proposition}

Let $g=\frac{4}{\kappa}$ and $\chi=\pi(1-g)$ be as before; note that in this case $g \in (1,\frac{3}{2})$ and $\chi \in (-\frac{\pi}{2},0)$. We may use Proposition~\ref{prop:laplace-m} again to extract the random moduli $\widetilde m_i(\dd \tau)$ from Proposition~\ref{prop:qa-1-2-simple}. The detailed steps of calculation are omitted.

\begin{theorem}\label{thm:mod-simple}
    The measures $\widetilde m_1(\dd \tau)$ and $\widetilde m_2(\dd \tau)$ in Proposition~\ref{prop:idenfity-lf-simple} are given by
    \[ \widetilde m_1(\dd \tau)=\mathcal{Z}_{\mathrm{odd}}(\tau) \frac{2\cos\chi}{\sqrt{2} \pi} \eta(2i\tau) \dd \tau \ \mbox{ and} \quad \widetilde m_2(\dd \tau)=\mathcal{Z}_{\mathrm{even}}(\tau) \frac{2\cos\chi}{\sqrt{2} \pi} \eta(2i\tau) \dd \tau \,, \]
    where $\eta(\cdot)$ is the Dedekind eta function, $\mathcal{Z}_{\mathrm{odd}}(\tau)$ and $\mathcal{Z}_{\mathrm{even}}(\tau)$ are given by~\eqref{eq:odd-pfn} and~\eqref{eq:even-pfn}.
\end{theorem}

The fact that the formulae in Theorem~\ref{thm:mod} and Lemma~\ref{lem:pfn-r} still hold in this case is not surprising since the simple CLE should correspond alternatively to the critical dilute $\mathrm{O}(n,n')$ model on the annulus, whose partition function has the same form as that of the dense one; see~\cite[Equation (2)]{Cardy06}.
Combining these results, we prove Theorem~\ref{thm:cle-partition} for $\kappa \in (8/3,4)$. The $\kappa=4$ case follows by sending $\kappa \uparrow 4$. See~\cite[Appendix A]{ACSW24} for the needed continuity in $\kappa$.

\bibliographystyle{alpha}
\bibliography{ref}

\begin{thebibliography}{DCMT21}

\bibitem[ACSW21]{ACSW-three-pt}
Morris {Ang}, Gefei {Cai}, Xin {Sun}, and Baojun {Wu}.
\newblock {Integrability of Conformal Loop Ensemble: Imaginary DOZZ Formula and Beyond}.
\newblock {\em arXiv e-prints}, page arXiv:2107.01788, July 2021.

\bibitem[ACSW24]{ACSW24}
Morris {Ang}, Gefei {Cai}, Xin {Sun}, and Baojun {Wu}.
\newblock {SLE Loop Measure and Liouville Quantum Gravity}.
\newblock {\em arXiv e-prints}, page arXiv:2409.16547, September 2024.

\bibitem[AHS24]{AHS21}
Morris Ang, Nina Holden, and Xin Sun.
\newblock Integrability of {SLE} via conformal welding of random surfaces.
\newblock {\em Comm. Pure Appl. Math.}, 77(5):2651--2707, 2024.

\bibitem[AHSY23]{AHSY23}
Morris {Ang}, Nina {Holden}, Xin {Sun}, and Pu~{Yu}.
\newblock {Conformal welding of quantum disks and multiple SLE: the non-simple case}.
\newblock {\em arXiv e-prints}, page arXiv:2310.20583, October 2023.

\bibitem[ARS22]{ARS22}
Morris {Ang}, Guillaume {Remy}, and Xin {Sun}.
\newblock {The moduli of annuli in random conformal geometry}.
\newblock {\em arXiv e-prints}, page arXiv:2203.12398, March 2022.

\bibitem[ASYZ24]{ASYZ24}
Morris {Ang}, Xin {Sun}, Pu~{Yu}, and Zijie {Zhuang}.
\newblock {Boundary touching probability and nested-path exponent for non-simple CLE}.
\newblock {\em arXiv e-prints}, page arXiv:2401.15904, January 2024.

\bibitem[BDC12]{BDC12}
Vincent Beffara and Hugo Duminil-Copin.
\newblock The self-dual point of the two-dimensional random-cluster model is critical for {$q\geq 1$}.
\newblock {\em Probab. Theory Related Fields}, 153(3-4):511--542, 2012.

\bibitem[BDC16]{BDC-notes}
Vincent Beffara and Hugo Duminil-Copin.
\newblock Critical point and duality in planar lattice models.
\newblock In {\em Probability and statistical physics in {S}t. {P}etersburg}, volume~91 of {\em Proc. Sympos. Pure Math.}, pages 51--98. Amer. Math. Soc., Providence, RI, 2016.

\bibitem[BH19]{BH19}
St\'ephane Benoist and Cl\'ement Hongler.
\newblock The scaling limit of critical {I}sing interfaces is {$\mathrm{CLE}_3$}.
\newblock {\em Ann. Probab.}, 47(4):2049--2086, 2019.

\bibitem[BPZ84]{BPZ84a}
A.~A. Belavin, A.~M. Polyakov, and A.~B. Zamolodchikov.
\newblock Infinite conformal symmetry in two-dimensional quantum field theory.
\newblock {\em Nuclear Phys. B}, 241(2):333--380, 1984.

\bibitem[Car06]{Cardy06}
John Cardy.
\newblock The {${\rm O}(n)$} model on the annulus.
\newblock {\em J. Stat. Phys.}, 125(1):1--21, 2006.

\bibitem[CF24]{CF24a-fk}
Federico {Camia} and Yu~{Feng}.
\newblock {Conformal covariance of connection probabilities in the 2D critical FK-Ising model}.
\newblock {\em arXiv e-prints}, page arXiv:2411.01467, November 2024.

\bibitem[CK14]{CK13looptree}
Nicolas Curien and Igor Kortchemski.
\newblock Random stable looptrees.
\newblock {\em Electron. J. Probab.}, 19:no. 108, 35, 2014.

\bibitem[CN06]{CN06}
Federico Camia and Charles~M. Newman.
\newblock Two-dimensional critical percolation: the full scaling limit.
\newblock {\em Comm. Math. Phys.}, 268(1):1--38, 2006.

\bibitem[DC20]{DC-notes}
Hugo Duminil-Copin.
\newblock Lectures on the {I}sing and {P}otts models on the hypercubic lattice.
\newblock In {\em Random graphs, phase transitions, and the {G}aussian free field}, volume 304 of {\em Springer Proc. Math. Stat.}, pages 35--161. Springer, Cham, [2020] \copyright 2020.

\bibitem[DCHN11]{DCHN11}
Hugo Duminil-Copin, Cl\'ement Hongler, and Pierre Nolin.
\newblock Connection probabilities and {RSW}-type bounds for the two-dimensional {FK} {I}sing model.
\newblock {\em Comm. Pure Appl. Math.}, 64(9):1165--1198, 2011.

\bibitem[DCM22]{DCM22}
Hugo Duminil-Copin and Ioan Manolescu.
\newblock Planar random-cluster model: scaling relations.
\newblock {\em Forum Math. Pi}, 10:Paper No. e23, 83, 2022.

\bibitem[DCMT21]{DCMT21}
Hugo Duminil-Copin, Ioan Manolescu, and Vincent Tassion.
\newblock Planar random-cluster model: fractal properties of the critical phase.
\newblock {\em Probab. Theory Related Fields}, 181(1-3):401--449, 2021.

\bibitem[DCST17]{DCST17}
Hugo Duminil-Copin, Vladas Sidoravicius, and Vincent Tassion.
\newblock Continuity of the phase transition for planar random-cluster and {P}otts models with {$1 \leq q \leq 4$}.
\newblock {\em Comm. Math. Phys.}, 349(1):47--107, 2017.

\bibitem[DCT20]{DCT20}
Hugo Duminil-Copin and Vincent Tassion.
\newblock Renormalization of crossing probabilities in the planar random-cluster model.
\newblock {\em Mosc. Math. J.}, 20(4):711--740, 2020.

\bibitem[DKRV16]{DKRV-sphere}
Fran\c{c}ois David, Antti Kupiainen, R\'{e}mi Rhodes, and Vincent Vargas.
\newblock Liouville quantum gravity on the {R}iemann sphere.
\newblock {\em Comm. Math. Phys.}, 342(3):869--907, 2016.

\bibitem[DMS21]{DMS21}
Bertrand Duplantier, Jason Miller, and Scott Sheffield.
\newblock Liouville quantum gravity as a mating of trees.
\newblock {\em Ast\'{e}risque}, (427):viii+257, 2021.

\bibitem[dN83]{dN83}
Marcel den Nijs.
\newblock Extended scaling relations for the magnetic critical exponents of the {P}otts model.
\newblock {\em Phys. Rev. B (3)}, 27(3):1674--1679, 1983.

\bibitem[DS11]{DS11}
Bertrand Duplantier and Scott Sheffield.
\newblock Liouville quantum gravity and {KPZ}.
\newblock {\em Invent. Math.}, 185(2):333--393, 2011.

\bibitem[Dub09]{Dub09}
Julien Dub\'edat.
\newblock S{LE} and the free field: partition functions and couplings.
\newblock {\em J. Amer. Math. Soc.}, 22(4):995--1054, 2009.

\bibitem[Fis67]{Fisher-1967}
Michael~E Fisher.
\newblock The theory of equilibrium critical phenomena.
\newblock {\em Reports on progress in physics}, 30(2):615, 1967.

\bibitem[FK72]{FK72}
C.~M. Fortuin and P.~W. Kasteleyn.
\newblock On the random-cluster model. {I}. {I}ntroduction and relation to other models.
\newblock {\em Physica}, 57:536--564, 1972.

\bibitem[GC06]{CG06}
Adam Gamsa and John Cardy.
\newblock Correlation functions of twist operators applied to single self-avoiding loops.
\newblock {\em J. Phys. A}, 39(41):12983--13003, 2006.

\bibitem[GKR24]{Liouville-review}
Colin {Guillarmou}, Antti {Kupiainen}, and R{\'e}mi {Rhodes}.
\newblock {Review on the probabilistic construction and Conformal bootstrap in Liouville Theory}.
\newblock {\em arXiv e-prints}, page arXiv:2403.12780, March 2024.

\bibitem[Gri06]{Grimmett-rcm}
Geoffrey Grimmett.
\newblock {\em The random-cluster model}, volume 333 of {\em Grundlehren der mathematischen Wissenschaften [Fundamental Principles of Mathematical Sciences]}.
\newblock Springer-Verlag, Berlin, 2006.

\bibitem[Kes87]{Kes87a}
Harry Kesten.
\newblock Scaling relations for {$2$}d-percolation.
\newblock {\em Comm. Math. Phys.}, 109(1):109--156, 1987.

\bibitem[KS16]{KS16}
Antti {Kemppainen} and Stanislav {Smirnov}.
\newblock {Conformal invariance in random cluster models. II. Full scaling limit as a branching SLE}.
\newblock {\em arXiv e-prints}, page arXiv:1609.08527, September 2016.

\bibitem[KS19]{KS19}
Antti Kemppainen and Stanislav Smirnov.
\newblock Conformal invariance of boundary touching loops of {FK} {I}sing model.
\newblock {\em Comm. Math. Phys.}, 369(1):49--98, 2019.

\bibitem[KST23]{KST23}
Laurin K\"ohler-Schindler and Vincent Tassion.
\newblock Crossing probabilities for planar percolation.
\newblock {\em Duke Math. J.}, 172(4):809--838, 2023.

\bibitem[Law09]{Lawler-partition}
Gregory~F. Lawler.
\newblock Partition functions, loop measure, and versions of {SLE}.
\newblock {\em J. Stat. Phys.}, 134(5-6):813--837, 2009.

\bibitem[LSYZ24]{LSYZ24}
Haoyu {Liu}, Xin {Sun}, Pu~{Yu}, and Zijie {Zhuang}.
\newblock {The bulk one-arm exponent for the CLE$_{\kappa'}$ percolations}.
\newblock {\em arXiv e-prints}, page arXiv:2410.12724, October 2024.

\bibitem[MN04]{MN04}
Saibal Mitra and Bernard Nienhuis.
\newblock Exact conjectured expressions for correlations in the dense o(1) loop model on cylinders.
\newblock {\em Journal of Statistical Mechanics: Theory and Experiment}, 2004(10):P10006, oct 2004.

\bibitem[MS16]{MS16a}
Jason Miller and Scott Sheffield.
\newblock Imaginary geometry {I}: interacting {SLE}s.
\newblock {\em Probab. Theory Related Fields}, 164(3-4):553--705, 2016.

\bibitem[MS17]{ig4}
Jason Miller and Scott Sheffield.
\newblock Imaginary geometry {IV}: interior rays, whole-plane reversibility, and space-filling trees.
\newblock {\em Probab. Theory Related Fields}, 169(3-4):729--869, 2017.

\bibitem[MSW17]{MSW2017}
Jason Miller, Scott Sheffield, and Wendelin Werner.
\newblock C{LE} percolations.
\newblock {\em Forum Math. Pi}, 5:e4, 102, 2017.

\bibitem[MSW21]{MSW21-nonsimple}
Jason Miller, Scott Sheffield, and Wendelin Werner.
\newblock Non-simple conformal loop ensembles on {L}iouville quantum gravity and the law of {CLE} percolation interfaces.
\newblock {\em Probab. Theory Related Fields}, 181(1-3):669--710, 2021.

\bibitem[Pel19]{Pel19}
Eveliina Peltola.
\newblock Toward a conformal field theory for {S}chramm-{L}oewner evolutions.
\newblock {\em J. Math. Phys.}, 60(10):103305, 39, 2019.

\bibitem[Rem18]{Rem18}
Guillaume Remy.
\newblock Liouville quantum gravity on the annulus.
\newblock {\em J. Math. Phys.}, 59(8):082303, 26, 2018.

\bibitem[Rus78]{Rus78}
Lucio Russo.
\newblock A note on percolation.
\newblock {\em Z. Wahrscheinlichkeitstheorie und Verw. Gebiete}, 43(1):39--48, 1978.

\bibitem[Rus81]{Rus81}
Lucio Russo.
\newblock On the critical percolation probabilities.
\newblock {\em Z. Wahrsch. Verw. Gebiete}, 56(2):229--237, 1981.

\bibitem[She07]{She07}
Scott Sheffield.
\newblock Gaussian free fields for mathematicians.
\newblock {\em Probab. Theory Related Fields}, 139(3-4):521--541, 2007.

\bibitem[She09]{Sheffield-CLE}
Scott Sheffield.
\newblock Exploration trees and conformal loop ensembles.
\newblock {\em Duke Math. J.}, 147(1):79--129, 2009.

\bibitem[She16]{She16}
Scott Sheffield.
\newblock Conformal weldings of random surfaces: {SLE} and the quantum gravity zipper.
\newblock {\em Ann. Probab.}, 44(5):3474--3545, 2016.

\bibitem[Smi10]{Smi10}
Stanislav Smirnov.
\newblock Conformal invariance in random cluster models. {I}. {H}olomorphic fermions in the {I}sing model.
\newblock {\em Ann. of Math. (2)}, 172(2):1435--1467, 2010.

\bibitem[SSW09]{SSW09}
Oded Schramm, Scott Sheffield, and David~B. Wilson.
\newblock Conformal radii for conformal loop ensembles.
\newblock {\em Comm. Math. Phys.}, 288(1):43--53, 2009.

\bibitem[SW78]{SW78}
P.~D. Seymour and D.~J.~A. Welsh.
\newblock Percolation probabilities on the square lattice.
\newblock {\em Ann. Discrete Math.}, 3:227--245, 1978.

\bibitem[SW05]{SW05}
Oded Schramm and David~B. Wilson.
\newblock S{LE} coordinate changes.
\newblock {\em New York J. Math.}, 11:659--669, 2005.

\bibitem[SW12]{Sheffield-Werner-CLE}
Scott Sheffield and Wendelin Werner.
\newblock Conformal loop ensembles: the {M}arkovian characterization and the loop-soup construction.
\newblock {\em Ann. of Math. (2)}, 176(3):1827--1917, 2012.

\bibitem[SW16]{SW16}
Scott {Sheffield} and Menglu {Wang}.
\newblock {Field-measure correspondence in Liouville quantum gravity almost surely commutes with all conformal maps simultaneously}.
\newblock {\em arXiv e-prints}, page arXiv:1605.06171, May 2016.

\bibitem[SXZ24]{SXZ24}
Xin {Sun}, Shengjing {Xu}, and Zijie {Zhuang}.
\newblock {Annulus crossing formulae for critical planar percolation}.
\newblock {\em arXiv e-prints}, page arXiv:2410.04767, October 2024.

\bibitem[Wer08]{Werner-loop}
Wendelin Werner.
\newblock The conformally invariant measure on self-avoiding loops.
\newblock {\em J. Amer. Math. Soc.}, 21(1):137--169, 2008.

\bibitem[Wid65]{Widom-1965}
Benjamin Widom.
\newblock Equation of state in the neighborhood of the critical point.
\newblock {\em The Journal of Chemical Physics}, 43(11):3898--3905, 1965.

\bibitem[WP21]{WP21}
Wendelin Werner and Ellen Powell.
\newblock {\em Lecture notes on the {G}aussian free field}, volume~28 of {\em Cours Sp\'ecialis\'es [Specialized Courses]}.
\newblock Soci\'et\'e{} Math\'ematique de France, Paris, 2021.

\end{thebibliography}

\end{document}